\pdfoutput=1
\newif\ifpersonal
\RequirePackage[l2tabu,orthodox]{nag} 
\documentclass{amsart} 
\usepackage{amsmath,amsthm,amssymb,mathrsfs,mathtools,bm,eucal,tensor} 
\usepackage{microtype,fixltx2e,lmodern} 
\usepackage[utf8]{inputenc} 
\usepackage[T1]{fontenc} 
\usepackage{enumerate,comment,braket,xspace,tikz-cd,csquotes} 
\usepackage[centering,vscale=0.7,hscale=0.7]{geometry}
\usepackage[hidelinks,linktoc=all]{hyperref}
\usepackage[capitalize]{cleveref}

\numberwithin{equation}{section}
\theoremstyle{plain}
\newtheorem{thm}[equation]{Theorem}
\newtheorem{lem}[equation]{Lemma}
\newtheorem{prop}[equation]{Proposition}

\newtheorem{cor}[equation]{Corollary}

\theoremstyle{definition}
\newtheorem{defin}[equation]{Definition}
\newtheorem{notation}[equation]{Notation}
\newtheorem{rem}[equation]{Remark}

\ifpersonal
\newcommand{\personal}[1]{\textcolor[rgb]{0,0,1}{(Personal: #1)}}
\newcommand{\discussion}[1]{\textcolor{violet}{(Discussion: #1)}}
\newcommand{\todo}[1]{\textcolor{red}{(Todo: #1)}}
\else
\newcommand{\personal}[1]{\ignorespaces}
\newcommand{\discussion}[1]{\ignorespaces}
\newcommand{\todo}[1]{\ignorespaces}
\fi

\newcommand{\C}{\mathbb C}

\newcommand{\Z}{\mathbb Z}

\newcommand{\rR}{\mathrm R}
\newcommand{\rT}{\mathrm T}

\newcommand{\cA}{\mathcal A}
\newcommand{\cB}{\mathcal B}
\newcommand{\cC}{\mathcal C}
\newcommand{\cD}{\mathcal D}
\newcommand{\cE}{\mathcal E}
\newcommand{\cF}{\mathcal F}
\newcommand{\cH}{\mathcal H}

\newcommand{\cJ}{\mathcal J}

\newcommand{\cO}{\mathcal O}

\newcommand{\cS}{\mathcal S}
\newcommand{\cT}{\mathcal T}
\newcommand{\cU}{\mathcal U}
\newcommand{\cV}{\mathcal V}
\newcommand{\cW}{\mathcal W}
\newcommand{\cX}{\mathcal X}
\newcommand{\cY}{\mathcal Y}
\newcommand{\cZ}{\mathcal Z}
\DeclareFontFamily{U}{BOONDOX-calo}{\skewchar\font=45 }
\DeclareFontShape{U}{BOONDOX-calo}{m}{n}{<-> s*[1.05] BOONDOX-r-calo}{}
\DeclareFontShape{U}{BOONDOX-calo}{b}{n}{<-> s*[1.05] BOONDOX-b-calo}{}
\DeclareMathAlphabet{\mathcalboondox}{U}{BOONDOX-calo}{m}{n}

\newcommand{\bbA}{\mathbb A}

\newcommand{\bbE}{\mathbb E}

\newcommand{\bA}{\mathbf A}
\newcommand{\bD}{\mathbf D}
\newcommand{\bP}{\mathbf P}

\makeatletter
\let\save@mathaccent\mathaccent
\newcommand*\if@single[3]{%
	\setbox0\hbox{${\mathaccent"0362{#1}}^H$}%
	\setbox2\hbox{${\mathaccent"0362{\kern0pt#1}}^H$}%
	\ifdim\ht0=\ht2 #3\else #2\fi
}
\newcommand*\rel@kern[1]{\kern#1\dimexpr\macc@kerna}
\newcommand*\widebar[1]{\@ifnextchar^{{\wide@bar{#1}{0}}}{\wide@bar{#1}{1}}}
\newcommand*\wide@bar[2]{\if@single{#1}{\wide@bar@{#1}{#2}{1}}{\wide@bar@{#1}{#2}{2}}}
\newcommand*\wide@bar@[3]{%
	\begingroup
	\def\mathaccent##1##2{%
		\let\mathaccent\save@mathaccent
		\if#32 \let\macc@nucleus\first@char \fi
		\setbox\z@\hbox{$\macc@style{\macc@nucleus}_{}$}%
		\setbox\tw@\hbox{$\macc@style{\macc@nucleus}{}_{}$}%
		\dimen@\wd\tw@
		\advance\dimen@-\wd\z@
		\divide\dimen@ 3
		\@tempdima\wd\tw@
		\advance\@tempdima-\scriptspace
		\divide\@tempdima 10
		\advance\dimen@-\@tempdima
		\ifdim\dimen@>\z@ \dimen@0pt\fi
		\rel@kern{0.6}\kern-\dimen@
		\if#31
		\overline{\rel@kern{-0.6}\kern\dimen@\macc@nucleus\rel@kern{0.4}\kern\dimen@}%
		\advance\dimen@0.4\dimexpr\macc@kerna
		\let\final@kern#2%
		\ifdim\dimen@<\z@ \let\final@kern1\fi
		\if\final@kern1 \kern-\dimen@\fi
		\else
		\overline{\rel@kern{-0.6}\kern\dimen@#1}%
		\fi
	}%
	\macc@depth\@ne
	\let\math@bgroup\@empty \let\math@egroup\macc@set@skewchar
	\mathsurround\z@ \frozen@everymath{\mathgroup\macc@group\relax}%
	\macc@set@skewchar\relax
	\let\mathaccentV\macc@nested@a
	\if#31
	\macc@nested@a\relax111{#1}%
	\else
	\def\gobble@till@marker##1\endmarker{}%
	\futurelet\first@char\gobble@till@marker#1\endmarker
	\ifcat\noexpand\first@char A\else
	\def\first@char{}%
	\fi
	\macc@nested@a\relax111{\first@char}%
	\fi
	\endgroup
}
\makeatother

\newcommand{\tu}{\widetilde u}

\newcommand{\tU}{\widetilde U}

\newcommand{\PSh}{\mathrm{PSh}}
\newcommand{\Sh}{\mathrm{Sh}}

\newcommand{\infcat}{$\infty$-category\xspace}
\newcommand{\infcats}{$\infty$-categories\xspace}

\newcommand{\infsite}{$\infty$-site\xspace}

\newcommand{\inftopos}{$\infty$-topos\xspace}
\newcommand{\inftopoi}{$\infty$-topoi\xspace}

\newcommand{\sSet}{\mathrm{sSet}}

\newcommand{\Ab}{\mathrm{Ab}}
\newcommand{\DAb}{\cD(\Ab)}

\newcommand{\tauet}{\tau_\mathrm{\acute{e}t}}

\newcommand{\bPsm}{\bP_\mathrm{sm}}

\newcommand{\Modh}{\textrm{-}\mathrm{Mod}^\heartsuit}
\newcommand{\Mod}{\textrm{-}\mathrm{Mod}}
\newcommand{\sMod}{\textrm{-}\mathrm{sMod}}
\newcommand{\Coh}{\mathrm{Coh}}

\newcommand{\Cohh}{\mathrm{Coh}^\heartsuit}

\newcommand{\St}{\mathrm{St}}

\newcommand{\Sch}{\mathrm{Sch}}

\newcommand{\Afd}{\mathrm{Afd}}
\newcommand{\Top}{\mathcal T\mathrm{op}}

\newcommand{\dAn}{\mathrm{dAn}}

\newcommand{\dAnk}{\mathrm{dAn}_k}
\newcommand{\Ank}{\mathrm{An}_k}
\newcommand{\cTan}{\cT_{\mathrm{an}}}
\newcommand{\cTank}{\cT_{\mathrm{an}}(k)}

\newcommand{\cTdisck}{\cT_{\mathrm{disc}}(k)}

\newcommand{\cTetk}{\cT_{\mathrm{\acute{e}t}}(k)}
\newcommand{\Str}{\mathrm{Str}}
\newcommand{\Strloc}{\mathrm{Str}^\mathrm{loc}}
\newcommand{\RTop}{\tensor*[^\rR]{\Top}{}}

\newcommand{\RHTop}{\tensor*[^\rR]{\mathcal{H}\Top}{}}

\newcommand{\dAfd}{\mathrm{dAfd}}
\newcommand{\dAfdk}{\mathrm{dAfd}_k}

\newcommand{\trunc}{\mathrm{t}_0}
\newcommand{\Hyp}{\mathrm{Hyp}}

\newcommand{\Ring}{\mathrm{Ring}}
\newcommand{\CRing}{\mathrm{CRing}}
\newcommand{\sCRing}{\mathrm{sCRing}}

\newcommand{\Cat}{\mathrm{Cat}}
\newcommand{\Catinf}{\mathrm{Cat}_\infty}

\newcommand{\fib}{\mathrm{fib}}
\newcommand{\DerAn}{\mathrm{Der}\an}
\newcommand{\anL}{\mathbb L\an}
\newcommand{\AnRing}{\mathrm{AnRing}}

\newcommand{\cTanc}{\cTan(\mathbb C)}

\newcommand{\dAff}{\mathrm{dAff}}
\newcommand{\afp}{^{\mathrm{afp}}}

\newcommand{\cHom}{\cH \mathrm{om}}

\newcommand{\cTAb}{\cT_{\Ab}}
\newcommand{\Catlex}{\Cat_\infty^{\mathrm{lex}}}
\newcommand{\underover}[1]{#1//#1}

\newcommand{\llb}{[\![}
\newcommand{\rrb}{]\!]}

\newcommand{\an}{^\mathrm{an}}
\newcommand{\alg}{^\mathrm{alg}}

\newcommand{\ev}{\mathrm{ev}}

\newcommand{\inv}{^{-1}}
\newcommand{\id}{\mathrm{id}}

\newcommand{\kanal}{$k$-analytic\xspace}

\newcommand{\op}{^\mathrm{op}}

\newcommand{\DM}{Deligne-Mumford\xspace}

\usetikzlibrary{decorations.markings} 
\tikzset{
  closed/.style = {decoration = {markings, mark = at position 0.5 with { \node[transform shape, xscale = .8, yscale=.4] {/}; } }, postaction = {decorate} },
  open/.style = {decoration = {markings, mark = at position 0.5 with { \node[transform shape, scale = .7] {$\circ$}; } }, postaction = {decorate} }
}

\DeclareMathOperator{\Fun}{Fun}

\DeclareMathOperator{\Hom}{Hom}

\DeclareMathOperator{\Map}{Map}

\DeclareMathOperator{\Sp}{Sp}

\DeclareMathOperator{\Spec}{Spec}

\DeclareMathOperator{\Sym}{Sym}

\DeclareMathOperator*{\colim}{colim}

\DeclareMathOperator*{\cotimes}{\widehat{\otimes}}

\renewenvironment{abstract}{%
	\quotation
	\small
	\textbf{\textit{\abstractname.}} 
}{\endquotation}

\makeatletter
\@namedef{subjclassname@2020}{%
	\textup{2020} Mathematics Subject Classification}
\makeatother

\begin{document}
\title[Representability theorem in derived analytic geometry]{Representability theorem in derived analytic geometry}

\author{Mauro PORTA}
\address{Mauro PORTA, Institut de Recherche Mathématique Avancée, 7 Rue René Descartes, 67000 Strasbourg, France}
\email{porta@math.unistra.fr}

\author{Tony Yue YU}
\address{Tony Yue YU, Department of Mathematics M/C 253-37, California Institute of Technology, 1200 E California Blvd, Pasadena, CA 91125, USA}
\email{yuyuetony@gmail.com}
\date{April 5, 2017 (Revised on March 16, 2022)}
\subjclass[2020]{Primary 14D23; Secondary 14G22, 32G13, 14A30}
\keywords{representability, deformation theory, analytic cotangent complex, derived geometry, rigid analytic geometry, complex geometry, derived stacks}

\maketitle

\begin{abstract}
We prove the representability theorem in derived analytic geometry.
The theorem asserts that an analytic moduli functor is a derived analytic stack if and only if it is compatible with Postnikov towers, has a global analytic cotangent complex, and its truncation is an analytic stack.
Our result applies to both derived complex analytic geometry and derived non-archimedean analytic geometry (rigid analytic geometry).
The representability theorem is of both philosophical and practical importance in derived geometry.
The conditions of representability are natural expectations for a moduli functor.
So the theorem confirms that the notion of derived analytic space is natural and sufficiently general.
On the other hand, the conditions are easy to verify in practice.
So the theorem enables us to enhance various classical moduli spaces with derived structures, thus provides plenty of down-to-earth examples of derived analytic spaces.
For the purpose of proof, we study analytification, square-zero extensions, analytic modules and cotangent complexes in the context of derived analytic geometry.
We will explore applications of the representability theorem in our subsequent works.
In particular, we will establish the existence of derived mapping stacks via the representability theorem.
\end{abstract}

\tableofcontents

\section{Introduction} \label{sec:introduction}

Derived algebraic geometry is a far reaching enhancement of algebraic geometry.
We refer to Toën \cite{Toen_Derived_2014} for an overview, and to Lurie \cite{Lurie_Thesis,DAG-V} and Toën-Vezzosi \cite{HAG-I,HAG-II} for foundational works.
A fundamental result in derived algebraic geometry is Lurie's representability theorem.
It gives sufficient and necessary conditions for a moduli functor to possess the structure of a derived algebraic stack.
The representability theorem enables us to enrich various classical moduli spaces with derived structures, thus bring derived geometry into the study of important moduli problems.
Examples include derived Picard schemes, derived Hilbert schemes, Weil restrictions, derived Betti moduli spaces, derived de Rham moduli spaces, and derived Dolbeault moduli spaces (cf.\ \cite{DAG-XIV,Lurie_Thesis,Pantev_Shifted_2013,Simpson_Algebraic_aspects_2002,Simpson_Geometricity_2009,Gaitsgory_A_study_2016}).

Algebraic geometry is intimately related to analytic geometry.
In \cite{DAG-IX}, Lurie proposed a framework for derived complex analytic geometry.
In \cite{Porta_Yu_Derived_non-archimedean_analytic_spaces}, we started to develop the foundation for derived non-archimedean analytic geometry.
Our motivation comes from enumerative problems in the study of mirror symmetry of Calabi-Yau manifolds.
We refer to the introduction of \cite{Porta_Yu_Derived_non-archimedean_analytic_spaces} for a more detailed discussion on the motivations.
Our results in \cite{Porta_Yu_Derived_non-archimedean_analytic_spaces} include the existence of fiber products, and a comparison theorem between discrete derived analytic spaces and non-archimedean analytic \DM stacks.

As in the algebraic case, the theory of derived analytic geometry cannot be useful without a representability theorem.
So we establish the representability theorem in derived analytic geometry in this paper.
We cover both the complex analytic case and the non-archimedean analytic case using a unified approach.
Let us now state our main result:

\begin{thm}[Representability, cf.\ {\cref{thm:representability}}]\label{thmintro:representability}
	Let $F$ be a stack over the étale site of derived analytic spaces.
	The followings are equivalent:
	\begin{enumerate}
		\item $F$ is a derived analytic stack;
		\item $F$ is compatible with Postnikov towers, has a global analytic cotangent complex, and its truncation $\trunc(F)$ is an (underived) analytic stack.
	\end{enumerate}	
\end{thm}

As in derived algebraic geometry, the representability theorem is of both philosophical and practical importance.
Since the conditions in \cref{thmintro:representability}(2) are natural expectations for a moduli functor $F$, the theorem confirms that our notion of derived analytic space is natural and sufficiently general.
On the other hand, these conditions are easy to verify in practice.
So \cref{thmintro:representability} provides us at the same time plenty of down-to-earth examples of derived analytic spaces.

The main ingredient in the proof of the representability theorem is \emph{derived analytic deformation theory}, which we develop in the body of this paper.
Central to this theory is the notion of \emph{analytic cotangent complex}.
Although this concept is similar to its algebraic counterpart, new ideas are needed in the analytic setting, especially in the non-archimedean case when the ground field has positive characteristic.

Let us give an informal account of the ideas involved.
Intuitively, a derived analytic space is a topological space equipped with a sheaf of derived analytic rings.
A derived analytic ring is a derived ring (e.g.\ a simplicial commutative ring) equipped with an extra analytic structure.
The extra analytic structure consists of informations about norms, convergence of power series, as well as composition rules among convergent power series.
In \cite{DAG-IX,Porta_Yu_Derived_non-archimedean_analytic_spaces}, this heuristic idea is made precise using the theory of pregeometry and structured topos introduced by Lurie \cite{DAG-V} (we recall it in \cref{sec:basic_notions}).
All analytic information is encoded in a pregeometry $\cTank$, where $k$ is either $\C$ or a non-archimedean field.
Then a derived analytic space $X$ is a pair $(\cX,\cO_X)$ consisting of an \inftopos $\cX$ and a $\cTank$-structure $\cO_X$ on $\cX$ satisfying some local finiteness condition (cf.\ \cref{def:derived_analytic_space}).
One should think of $\cX$ as the underlying topological space, and $\cO_X$ as the structure sheaf.
A derived analytic ring is formally defined as a $\cTank$-structure on a point.

Intuitively, the analytic cotangent complex of a derived analytic space represents the derived cotangent spaces.
We will construct it via the space of derivations.
Recall that given a $k$-algebra $A$ and an $A$-module $M$, a \emph{derivation} of $A$ into $M$ is a $k$-linear map $d\colon A\to M$ satisfying the Leibniz rule:
\[d(ab)=b\, d(a)+a\, d(b).\]

In the context of derived analytic geometry, we take $A$ to be a derived analytic ring.
Let $A\alg$ denote the underlying derived ring of $A$, obtained by forgetting the analytic structure.
We define $A$-modules to be simply $A\alg$-modules, (we will see later that this is a reasonable definition.)
Let $M$ be an $A$-module and we want to define what are analytic derivations of $A$ into $M$.
However, the Leibniz rule above is problematic in derived analytic geometry for two reasons.

The first problem concerns analytic geometry.
It follows from the Leibniz rule that for any element $a\in A$ and any polynomial in one variable $f$, we have
\[d(f(a))=f'(a)d(a).\]
In analytic geometry, it is natural to demand the same formula not only for polynomials but also for every convergent power series $f$.
This means that we have to add infinitely many new rules.

The second problem concerns derived geometry.
For derived rings, we are obliged to demand the Leibniz rule up to homotopy.
This results in an infinite chain of higher homotopies and becomes impossible to manipulate.

In order to solve the two problems, note that in the classical case, a derivation of $A$ into $M$ is equivalent to a section of the projection from the split square-zero extension $A\oplus M$ to $A$.
So we can reduce the problem of formulating the Leibniz rule involving convergent power series as well as higher homotopies to the problem of constructing split square-zero extensions of derived analytic rings.
In other words, given a derived analytic ring $A$ and an $A$-module $M$, we would like to construct a derived analytic ring structure on the direct sum $A\oplus M$.

For this purpose, we need to interpret the notion of $A$-module in a different way, which is the content of the following theorem.

\begin{thm}[Reinterpretation of modules, cf.\ {\cref{thm:equivalence_of_modules}}] \label{thmintro:modules}
Let $X = (\cX, \cO_X)$ be a derived analytic space.
We have an equivalence of stable $\infty$-categories
\[\cO_X\Mod \simeq \Sp(\Ab(\AnRing_k(\cX)_{/\cO_X})),\]
where $\AnRing_k(\cX)_{/\cO_X}$ denotes the \infcat of sheaves of derived $k$-analytic rings on $\cX$ over $\cO_X$, $\Ab(-)$ denotes the \infcat of abelian group objects, and $\Sp(-)$ denotes the \infcat of spectrum objects.
\end{thm}

We have natural functors
\[\Sp(\Ab(\AnRing_k(\cX)_{/\cO_X})))\xrightarrow{\Omega^\infty}\Ab(\AnRing_k(\cX)_{/\cO_X})\xrightarrow{U}\AnRing_k(\cX)_{/\cO_X}.\]
We will show that given $\cF\in \cO_X\Mod$, the underlying sheaf of derived rings of $U(\Omega^\infty(\cF))$ is equivalent to the algebraic split square-zero extension of $\cO_X\alg$ by $\cF$ (cf.\ \cref{cor:split_square-zero_extension}).
So we define $U(\Omega^\infty(\cF))$ to be the \emph{analytic split square-zero extension} of $\cO_X$ by $\cF$, which we denote by $\cO_X \oplus \cF$.

\cref{thmintro:modules} also confirms that our definition of module over a derived analytic ring $A$ as $A\alg$-module is reasonable because it can be reinterpreted in a purely analytic way without forgetting the analytic structure.

Let us now explain the necessity of taking abelian group objects in the statement of \cref{thmintro:modules}.
Given an $\bbE_\infty$-ring $A$, the \infcat of $A$-modules is equivalent to the \infcat $\Sp(\mathbb E_\infty \textrm{-} \Ring_{/A})$, where $\mathbb E_\infty \textrm{-} \Ring_{/A}$ denotes the \infcat of $\bbE_\infty$-rings over $A$ (cf.\ \cite[7.3.4.14]{Lurie_Higher_algebra}).
However, our approach to derived analytic geometry via structured topoi is simplicial in nature.
For a simplicial commutative ring $A$, the \infcat of $A$-modules is in general not equivalent to the \infcat $\Sp(\CRing_{/ A})$, where $\CRing_{/A}$ denotes the \infcat of simplicial commutative rings over $A$.
This problem can be solved by taking abelian group objects before taking spectrum objects.
More precisely, in \cref{sec:modules}, for any simplicial commutative ring $A$, we prove the following equivalence of stable \infcats
\begin{equation} \label{eq:algebraic_modules}
A\Mod\simeq\Sp(\Ab(\CRing_{/A})).
\end{equation}

The proof of \cref{thmintro:modules} is rather involved.
Let us give a quick outline for the convenience of the reader:
By \cref{eq:algebraic_modules}, we are reduced to prove an equivalence
\begin{equation} \label{eq:SpAbAnAlg}
\Sp(\Ab(\AnRing_k(\cX)_{/\cO_X})) \xrightarrow{\sim} \Sp(\Ab(\CRing_k(\cX)_{/\cO_X\alg})).
\end{equation}
The functor above is induced by the underlying algebra functor forgetting the analytic structure
\[ (-)\alg \colon \AnRing_k(\cX) \to \CRing_k(\cX). \]
Via a series of reduction steps in \cref{sec:reduction_to_connected_objects}, we can deduce \cref{eq:SpAbAnAlg} from the following equivalence
\[	(-)\alg \colon \AnRing_k(\cX)_{\underover{\cO_X}}^{\ge 1} \simeq \CRing_k(\cX)_{\underover{\cO_X\alg}}^{\ge 1} .\]
In \cref{sec:reduction_to_spaces}, we make a further reduction to the case of a point, i.e.\ when $\cX$ is the \infcat of spaces $\cS$.
The proof is finally achieved in \cref{sec:flatness} via flatness arguments.

With the preparations above, we are ready to introduce the notions of analytic derivation and analytic cotangent complex.

Let $X = (\cX, \cO_X)$ be a derived analytic space, let $\cA \in \AnRing_k(\cX)_{/\cO_X}$ and $\cF \in \cO_X\Mod^{\ge 0}$.
The space of \emph{$\cA$-linear analytic derivations of $\cO_X$ into $\cF$} is by definition
\[ \DerAn_{\cA}(\cO_X, \cF) \coloneqq \Map_{\AnRing_k(\cX)_{\cA // \cO_X}}( \cO_X, \cO_X \oplus \cF ). \]
In \cref{subsec:analytic_cotangent_complex}, we show that the functor $\DerAn_\cA(\cO_X, - )$ is representable by an $\cO_X$-module which we denote by $\anL_{\cO_X/\cA}$, and call the \emph{analytic cotangent complex} of $\cO_X/\cA$.
For a map of derived analytic spaces $f\colon X\to Y$, we define its analytic cotangent complex $\anL_{X/Y}$ to be $\anL_{\cO_X/f\inv\cO_Y}$.

Important properties of the analytic cotangent complex are established in \cref{sec:cotangent_complex}, and we summarize them in the following theorem:

\begin{thm}[Properties of the analytic cotangent complex] \label{thmintro:properties_cotangent_complex}
	The analytic cotangent complex satisfy the following properties:
	\begin{enumerate}
		\item For any map of derived analytic spaces $f\colon X\to Y$, the analytic cotangent complex $\anL_{X/Y}$ is connective and coherent.
				\item For any sequence of maps $X \xrightarrow{f} Y \xrightarrow{g} Z$, we have a fiber sequence
		\[ f^* \anL_{Y/Z} \to \anL_{X / Z} \to \anL_{X / Y} . \]
		\item For any pullback square of derived analytic spaces
		\[ \begin{tikzcd}
			X' \arrow{r} \arrow{d}{g} & Y' \arrow{d}{f} \\
			X \arrow{r} & Y ,
		\end{tikzcd} \]
		we have a canonical equivalence
		\[ \anL_{X' / Y'} \simeq g^* \anL_{X / Y} . \]
		\item For any derived algebraic \DM stack $X$ locally almost of finite presentation over $k$, its analytification $X\an$ is a derived analytic space (cf.\ \cref{sec:analytification}).
		We have a canonical equivalence
		\[ (\mathbb L_X)\an \simeq \anL_{X\an} . \]
		\item For any closed immersion of derived analytic spaces $X \hookrightarrow Y$, we have a canonical equivalence
		\[ \anL_{X / Y} \simeq \mathbb L_{X\alg / Y\alg} . \]
		\item (Analytic Postnikov tower) For any derived analytic space $X$, every $n \ge 0$, the canonical map $\mathrm t_{\le n}(X) \hookrightarrow \mathrm t_{\le n+1}(X)$ is an analytic square-zero extension. In other words, there exists an analytic derivation
		\[ d \colon \anL_{\tau_{\le n} \cO_X} \to \pi_{n+1}(\cO_X)[n+2] \]
		such that the square
		\[ \begin{tikzcd}
			\tau_{\le n +1} \cO_X \arrow{r} \arrow{d} & \tau_{\le n} \cO_X \arrow{d}{\eta_d} \\
			\tau_{\le n} \cO_X \arrow{r}{\eta_0} & \tau_{\le n} \cO_X \oplus \pi_{n+1}(\cO_X)[n+2]
		\end{tikzcd} \]
		is a pullback, where $\eta_d$ is the map associated to the derivation $d$ and $\eta_0$ is the map associated to the zero derivation.
		\item A morphism $f \colon X \to Y$ of derived analytic spaces is smooth if and only if its truncation $\trunc(f)$ is smooth and the analytic cotangent complex $\anL_{X/Y}$ is perfect and in tor-amplitude $0$.
	\end{enumerate}
\end{thm}

The properties (1) - (7) correspond respectively to \cref{cor:finiteness_cotangent_complex}, \cref{prop:transitivity_cotangent_complex}, \cref{prop:base_change_cotangent_complex_analytic}, \cref{thm:analytification_cotangent_complex}, \cref{cor:cotangent_complex_closed_immersion}, \cref{cor:Postnikov_tower_analytic_square_zero} and \cref{prop:characterization_smooth_morphisms}.

Using Properties (2), (4) and (5), we can compute the analytic cotangent complex of any derived analytic space via local embeddings into affine spaces.

In \cref{sec:gluing}, we use the analytic Postnikov tower decomposition to construct pushout of derived analytic spaces along closed immersions:

\begin{thm}[Gluing along closed immersions, cf.\ {\cref{thm:pushout_closed_immersions}}] \label{thmintro:gluing_closed_immersions}
	Let
	\[ \begin{tikzcd}
		X \arrow{r}{i} \arrow{d}{j} & X' \arrow{d} \\
		Y \arrow{r} & Y'
	\end{tikzcd} \]
	be a pushout square of $\cTank$-structured topoi.
	Suppose that $i$ and $j$ are closed immersions and $X$, $X'$, $Y$ are derived analytic spaces.
	Then $Y'$ is also a derived analytic space.
\end{thm}

In other words, the theorem asserts that derived analytic spaces can be glued together along closed immersions.
In particular, it has the following important consequence:

\begin{cor}[Representability of analytic square-zero extensions] \label{corintro:square-zero}
	Let $X$ be a derived analytic space and let $\cF \in \Coh^{\ge 1}(X)$.
	Let $X[\cF]$ be the analytic split square-zero extension of $X$ by $\cF$.
	Let $i_d \colon X[\cF] \to X$ be the map associated to an analytic derivation $d$ of $\cO_X$ into $\cF$.
	Let $i_0 \colon X[\cF] \to X$ be the map associated to the zero derivation.
	Then the pushout
	\[ \begin{tikzcd}
		X[\cF] \arrow{r}{i_d} \arrow{d}{i_0} & X \arrow{d} \\
		X \arrow{r} & X_d[\cF]
	\end{tikzcd} \]
	is a derived analytic space.
\end{cor}

The corollary gives one more evidence that our notions of analytic derivation and analytic cotangent complex are correct.
If we replace $d$ by an algebraic derivation, the pushout will no longer be a derived analytic space in general.

Now let us give a sketch of the proof of \cref{thmintro:representability}, the main theorem of this paper.

The implication (1)$\Rightarrow$(2) is worked out in \cref{sec:properties_of_derived_analytic_stacks}.
We first prove that (2) holds for derived analytic spaces.
We deduce it from the various properties of the analytic cotangent complex explained above as well as the gluing along closed immersions.
After that, we prove (2) for derived analytic stacks by induction on the geometric level.

The proof of the implication (2)$\Rightarrow$(1) is more involved.
By induction on the geometric level of the truncation $\trunc(F)$, it suffices to lift a smooth atlas $U_0\to\trunc(F)$ of $\trunc(F)$ to a smooth atlas $\tU\to F$ of $F$.
To obtain such a lifting, we proceed by constructing successive approximations:
\[ \begin{tikzcd}
U_0 \arrow[hook]{r}{j_0} \arrow[bend right = 5]{drrrrr}[swap, near start]{u_0} & U_1 \arrow[hook]{r}{j_1} \arrow[bend right = 2]{drrrr}{u_1} & \cdots \arrow[hook]{r}{j_{m-1}} & U_m \arrow{drr}[bend right = 1.5]{u_m} \arrow[hook]{r}{j_m} & \cdots \\
{} & & & & & F
\end{tikzcd} \]
where
\begin{enumerate}
	\item $U_m$ is $m$-truncated;
	\item $U_m \to U_{m+1}$ induces an equivalence on $m$-truncations;
	\item $\anL_{U_m / F}$ is flat to order $m+1$ (cf.\ \cref{def:n_flat}).
\end{enumerate}
The construction goes by induction on $m$.
The notion of flatness to order $n$ is the key idea behind the induction step.
Indeed, combining the fact that $\anL_{U_m / F}$ is flat to order $m+1$ with the fact that $U_m$ is $m$-truncated guarantees that the truncation $\tau_{\le m} \anL_{U_m / F}$ is flat as a sheaf on $U_m$.
It is not hard to deduce from here that $\tau_{\le m} \anL_{U_m / F}$ must be perfect.
From this, we can choose a splitting
\[ \anL_{U_m / F} \simeq \tau_{\le m} \anL_{U_m / F} \oplus \tau_{\ge m+1} \anL_{U_m / F} . \]
The choice of the splitting determines the passage to the next level of the approximation.
We remark that the splitting above is in general not unique, and thus the choice of the lifting $\tU \to F$ of $U_0 \to \trunc(F)$ is not unique.
When $F$ is \DM, the lifting is unique, in other words, an atlas of $\trunc(F)$ determines canonically an atlas of $F$.

To complete the proof, we set $\tU\coloneqq\colim_m U_m$.
The construction above guarantees that $U_m \simeq \mathrm t_{\le m}(\tU)$.
Since $F$ is compatible with Postnikov towers, we obtain a canonical map $\tU \to F$.
The induction hypothesis on the geometric level of $F$ guarantees that this map is representable by geometric stacks.
In order to check that the map $\tU \to F$ is also smooth, we use an infinitesimal lifting property that we establish in \cref{prop:infinitesimal_characterization_smooth}.

Finally, we would like to stress that our approach to the representability theorem in derived analytic geometry is by no means a simple repetition of the proof of the representability theorem in derived algebraic geometry.
As we have explained above, the presence of the extra analytic structure has obliged us to make take different paths at various stages.
This also leads to a more conceptual understanding of the proof of the representability theorem in derived algebraic geometry.

We will explore applications of the representability theorem in our subsequent works (cf.\ \cite{Porta_Yu_Derived_Hom_spaces,Porta_Yu_Non-archimedean_quantum_K-invariants,Porta_Yu_Non-archimedean_Gromov-Witten_invariants}).

\bigskip
\paragraph{\textbf{Notations and terminology}}

We refer to \cite{Porta_Yu_Derived_non-archimedean_analytic_spaces} for the framework of derived non-archimedean analytic geometry, and to \cite{DAG-IX} for the framework of derived complex analytic geometry.
We give a unified review of the basic notions in \cref{sec:basic_notions}.

The letter $k$ denotes either the field $\C$ of complex numbers or a non-archimedean field with nontrivial valuation.
By \kanal spaces (or simply analytic spaces), we mean complex analytic spaces when $k=\C$, and rigid \kanal spaces when $k$ is non-archimedean.

We denote by $\Ank$ the category of \kanal spaces, and by $\dAnk$ the \infcat of derived \kanal spaces.
We denote by $\Afd_k$ the category of $k$-affinoid spaces when $k$ is non-archimedean, and the category of Stein spaces when $k=\C$.
We denote by $\dAfd_k$ the \infcat of derived affinoid spaces when $k$ is non-archimedean, and the \infcat of derived Stein spaces when $k=\C$.

For $n\in\Z_{\ge 0}$, we denote by $\bbA^n_k$ the algebraic $n$-dimensional affine space over $k$, by $\bA^n_k$ the analytic $n$-dimensional affine space over $k$, and by $\bD^n_k$ the $n$-dimensional closed unit polydisk over $k$.

For an \inftopos $\cX$, we denote by $\AnRing_k(\cX)$ the \infcat of sheaves of derived \kanal rings over $\cX$, and by $\CRing_k(\cX)$ the \infcat of sheaves of simplicial commutative $k$-algebras over $\cX$.

We denote by $\cS$ the \infcat of spaces.
An \infsite $(\cC,\tau)$ consists of a small \infcat $\cC$ equipped with a Grothendieck topology $\tau$.
A stack over an \infsite $(\cC,\tau)$ is by definition a hypercomplete sheaf with values in $\cS$ over the \infsite (cf.\ \cite[\S 2]{Porta_Yu_Higher_analytic_stacks_2014}).
We denote by $\St(\cC,\tau)$ the \infcat of stacks over $(\cC,\tau)$.

Throughout this paper, we use homological indexing conventions, i.e., the differential in chain complexes lowers the degree by 1.

A commutative diagram of $\infty$-categories
\[ \begin{tikzcd}
	\cC \arrow{r}{p} \arrow{d}{g} & \cC' \arrow{d}{g'} \\
	\cD \arrow{r}{q} & \cD'
\end{tikzcd} \]
is called \emph{left adjointable} if the functors $g$ and $g'$ have left adjoints $f \colon \cD \to \cC$, $f' \colon \cD' \to \cC'$ and if the push-pull transformation
\[ \gamma \colon f' \circ q \to p \circ f \]
is an equivalence (cf.\ \cite[7.3.1.1]{HTT}).

\bigskip
\paragraph{\textbf{Acknowledgements}}
We are very grateful to Antoine Chambert-Loir, Maxim Kontsevich, Jacob Lurie, Tony Pantev, Marco Robalo, Nick Rozenblyum, Carlos Simpson, Bertrand To\"en and Gabriele Vezzosi for valuable discussions.
Special thanks to Micheal Temkin for helping us with non-archimedean pinchings.
The authors would also like to thank each other for the joint effort.
Various stages of this research received supports from the Clay Mathematics Institute, Simons Foundation grant number 347070, Fondation Sciences Mathématiques de Paris, and from the Ky Fan and Yu-Fen Fan Membership Fund and the S.-S.\ Chern Endowment Fund of the Institute for Advanced Study.

\section{Basic notions of derived analytic geometry} \label{sec:basic_notions}

In this section we review the basic notions of derived complex analytic geometry and derived non-archimedean geometry in a unified framework.

First we recall the notions of pregeometry and structured topos introduced by Lurie in \cite{DAG-V}.

\begin{defin}[{\cite[3.1.1]{DAG-V}}]
	A \emph{pregeometry} is an \infcat $\cT$ equipped with a class of \emph{admissible} morphisms and a Grothendieck topology generated by admissible morphisms, satisfying the following conditions:
	\begin{enumerate}[(i)]
	\item The \infcat $\cT$ admits finite products.
	\item The pullback of an admissible morphism along any morphism exists.
	\item The class of admissible morphisms is closed under composition, pullback and retract.
	Moreover, for morphisms $f$ and $g$, if $g$ and $g\circ f$ are admissible, then $f$ is admissible.
	\end{enumerate}
\end{defin}

\begin{defin}[{\cite[3.1.4]{DAG-V}}] \label{def:structure}
	Let $\cT$ be a pregeometry, and let $\cX$ be an \inftopos.
	A \emph{$\cT$-structure} on $\cX$ is a functor $\cO\colon\cT\to\cX$ with the following properties:
	\begin{enumerate}[(i)]
		\item The functor $\cO$ preserves finite products.
		\item Suppose given a pullback diagram
		\[
		\begin{tikzcd}
		U' \arrow{r} \arrow{d} & U \arrow{d}{f} \\
		X' \arrow{r} & X
		\end{tikzcd}
		\]
		in $\cT$, where $f$ is admissible.
		Then the induced diagram
		\[
		\begin{tikzcd}
		\cO(U') \arrow{r} \arrow{d} & \cO(U) \arrow{d} \\
		\cO(X') \arrow{r} & \cO(X)
		\end{tikzcd}
		\]
		is a pullback square in $\cX$.
		\item Let $\{U_\alpha\to X\}$ be a covering in $\cT$ consisting of admissible morphisms.
		Then the induced map
		\[\coprod_\alpha\cO(U_\alpha)\to\cO(X)\]
		is an effective epimorphism in $\cX$.
			\end{enumerate}
	A morphism of $\cT$-structures $\cO\to\cO'$ on $\cX$ is \emph{local} if for every admissible morphism $U\to X$ in $\cT$, the resulting diagram
	\[ \begin{tikzcd}
	\cO(U) \arrow{r} \arrow{d} & \cO'(U) \arrow{d} \\
	\cO(X) \arrow{r} & \cO'(X)
	\end{tikzcd} \]
	is a pullback square in $\cX$.
	We denote by $\Strloc_\cT(\cX)$ the \infcat of $\cT$-structures on $\cX$ with local morphisms.
	
	A \emph{$\cT$-structured \inftopos} $X$ is a pair $(\cX,\cO_X)$ consisting of an \inftopos $\cX$ and a $\cT$-structure $\cO_X$ on $\cX$.
	We denote by $\RTop(\cT)$ the \infcat of $\cT$-structured \inftopoi (cf.\ \cite[Definition 1.4.8]{DAG-V}).
	Note that a 1-morphism $f\colon (\cX, \cO_X) \to (\cY, \cO_Y)$ in $\RTop(\cT)$ consists of a geometric morphism of \inftopoi $f_*\colon\cX\rightleftarrows\cY\colon f\inv$ and a local morphism of $\cT$-structures $f^\sharp \colon f\inv \cO_Y \to \cO_X$.
\end{defin}

Let $k$ denote either the field $\C$ of complex numbers or a complete non-archimedean field with nontrivial valuation.
We introduce three pregeometries $\cTank$, $\cTdisck$ and $\cTetk$ that are relevant to derived analytic geometry.

The pregeometry $\cTank$ is defined as follows:
	\begin{enumerate}[(i)]
		\item The underlying category of $\cTank$ is the category of smooth $k$-analytic spaces;
		\item A morphism in $\cTank$ is admissible if and only if it is étale;
		\item The topology on $\cTank$ is the étale topology.
		(Note that in the complex analytic case, the étale topology is equivalent to the usual analytic topology.)
	\end{enumerate}

The pregeometry $\cTdisck$ is defined as follows:
	\begin{enumerate}[(i)]
		\item The underlying category of $\cTdisck$ is the full subcategory of the category of $k$-schemes spanned by affine spaces $\mathbb A^n_k$;
		\item A morphism in $\cTdisck$ is admissible if and only if it is an isomorphism;
		\item The topology on $\cTdisck$ is the trivial topology, i.e.\ a collection of admissible morphisms is a covering if and only if it is nonempty.
	\end{enumerate}

The pregeometry $\cTetk$ is defined as follows:
\begin{enumerate}[(i)]
	\item The underlying category of $\cTetk$ is the category of smooth $k$-schemes;
	\item A morphism in $\cTetk$ is admissible if and only if it is étale;
	\item The topology on $\cTetk$ is the étale topology.
\end{enumerate}

We have a natural functor $\cTdisck \to \cTank$ induced by analytification.
Composing with this functor, we obtain an ``algebraization'' functor
\[
(-)^\mathrm{alg} \colon \Strloc_{\cTank}(\cX) \to \Strloc_{\cTdisck}(\cX).
\]
In virtue of \cite[Example 3.1.6, Remark 4.1.2]{DAG-V}, we have an equivalence induced by evaluation on the affine line
\[\Strloc_{\cTdisck}(\cX) \xrightarrow{\ \sim\ } \Sh_{\CRing_k}(\cX),\]
where $\Sh_{\CRing_k}(\cX)$ denotes the \infcat of sheaves on $\cX$ with values in the \infcat of simplicial commutative $k$-algebras.

\begin{defin} \label{def:derived_analytic_space}
	A \emph{derived \kanal space} $X$ is a $\cTank$-structured \inftopos $(\cX,\cO_X)$ such that $\cX$ is hypercomplete and there exists an effective epimorphism from $\coprod_i U_i$ to a final object of $\cX$ satisfying the following conditions, for every index $i$:
	\begin{enumerate}[(i)]
		\item The pair $(\cX_{/U_i}, \pi_0(\cO\alg_X | U_i))$ is equivalent to the ringed \inftopos associated to the étale site of a \kanal space $X_i$.
		\item For each $j\ge 0$, $\pi_j(\cO\alg_X | U_i)$ is a coherent sheaf of $\pi_0(\cO\alg_X | U_i)$-modules on $X_i$.
	\end{enumerate}
	We denote by $\dAnk$ the full subcategory of $\RTop(\cTank)$ spanned by derived \kanal spaces.
\end{defin}

\begin{defin}
	When $k$ is non-archimedean, a \emph{derived $k$-affinoid space} is by definition a derived $k$-analytic space $(\cX, \cO_X)$ whose truncation $(\cX, \pi_0(\cO_X))$ is a $k$-affinoid space.
	A \emph{derived Stein space} is by definition a derived $\C$-analytic space whose truncation is a Stein space.
	We denote the $\infty$-category of derived $k$-affinoid (resp.\ Stein) spaces by $\dAfdk$ (resp.\ $\dAfd_{\mathbb C}$).
\end{defin}

\section{Derived analytification} \label{sec:analytification}

In this section, we study the analytification of derived algebraic \DM stacks.

Let $\RHTop(\cTank)$ denote the full subcategory of $\RTop(\cTank)$ spanned by $\cTank$-structured \inftopoi whose underlying \inftopos is hypercomplete.
By \cite[Lemma 2.8]{Porta_Yu_Derived_non-archimedean_analytic_spaces}, the inclusion $\RHTop(\cTank)\allowbreak\hookrightarrow\RTop(\cTank)$ admits a right adjoint $\Hyp\colon\RTop(\cTank)\to\RHTop(\cTank)$.

By analytification, we have a transformation of pregeometries
\[ (-)\an \colon \cTetk \longrightarrow \cTank . \]
Precomposition with $(-)\an$ induces a forgetful functor
\[ (-)\alg \colon \RTop(\cTank) \longrightarrow \RTop(\cTetk) , \]
which admits a right adjoint in virtue of \cite[Theorem 2.1.1]{DAG-V}.
Composing with the right adjoint $\Hyp \colon \RTop(\cTank) \to \RHTop(\cTank)$, we obtain a functor
\[ \RTop(\cTetk) \longrightarrow \RHTop(\cTank) . \]
We call this functor the \emph{derived analytification functor}, and we denote it by $(-)\an$ again.
This notation is justified by the lemma below.


\begin{lem} \label{lem:analytification_formal_properties}
	\begin{enumerate}
		\item \label{item:analytification_affines} The diagram
		\[ \begin{tikzcd}
			\cTetk \arrow{r}{(-)\an} \arrow{d} & \cTank \arrow{d} \\
			\RTop(\cTetk) \arrow{r}{(-)\an} & \RHTop(\cTank)
		\end{tikzcd} \]
		commutes.
		\item \label{item:truncation_functors} Let us denote by $\RTop^{\le n}(\cTetk)$ (resp.\ $\RTop^{\le n}(\cTank)$) the full subcategory of $\RTop(\cTetk)$ (resp.\ $\RTop(\cTank)$) spanned by those $(\cX, \cO_X)$ such that $\cO_X$ is $n$-truncated.
		The diagram
		\[ \begin{tikzcd}
			\RTop(\cTetk) & \RHTop(\cTank) \arrow{l}[swap]{(-)\alg} \\
			\RTop^{\le n}(\cTetk) \arrow{u} & \RHTop^{\le n}(\cTank) \arrow{l}[swap]{(-)\alg} \arrow{u}
		\end{tikzcd} \]
		commutes, and the vertical arrows are left adjoint to the truncation functor $\mathrm t_{\le n}$.
		\item \label{item:analytification_truncation_adjointable} The functor $(-)\alg \colon \RHTop^{\le n}(\cTank) \to \RTop^{\le n}(\cTetk)$ admits a right adjoint which we denote by $\Psi_n$, and moreover the diagram
		\[ \begin{tikzcd}
			\Sch(\cTetk) \arrow{r}{(-)\an} \arrow{d}{\mathrm t_{\le n}} & \RHTop(\cTank) \arrow{d}{\mathrm t_{\le n}} \\
			\Sch^{\le n}(\cTetk) \arrow{r}{\Psi_n} & \RHTop^{\le n}(\cTank)
		\end{tikzcd} \]
		is left adjointable, where $\Sch^{\le n}(\cTetk)\coloneqq\Sch(\cTetk)\cap\RTop^{\le n}(\cTetk)$, and $\Sch(\cTetk)$ denotes the \infcat of $\cTetk$-schemes (cf.\ \cite[3.4.8]{DAG-V}).
	\end{enumerate}
\end{lem}

\begin{proof}
	Recall from \cite[6.5.2.9]{HTT} that truncated objects in an $\infty$-topos are hypercomplete.
	Then statement (\ref{item:analytification_affines}) follows from \cite[Proposition 2.3.8]{DAG-V}.
	Statement (\ref{item:truncation_functors}) is a consequence of the compatibility of $\cTetk$ and $\cTank$ with $n$-truncations for $n \ge 0$ (for $\cTetk$, we refer to \cite[4.3.28]{DAG-V}; for $\cTank$, we refer to \cite[Theorem 3.23]{Porta_Yu_Derived_non-archimedean_analytic_spaces} in the non-archimedean case and \cite[Proposition 11.4]{DAG-IX} in the complex case).
	Finally, statement (\ref{item:analytification_truncation_adjointable}) follows from \cite[Proposition 6.2]{Porta_DCAGI}.
\end{proof}

\begin{cor} \label{cor:analytification_closed_immersions}
	Let $j \colon Y \hookrightarrow X$ be a closed immersion in $\cTetk$.
	The induced map $j\an \colon Y\an \to X\an$ is a closed immersion in $\RHTop(\cTank)$.
\end{cor}

\begin{proof}
	Recall from \cite[Lemma 5.2]{Porta_Yu_Derived_non-archimedean_analytic_spaces} that the hypercompletion functor $\Hyp$ preserves closed immersions of $\infty$-topoi.
	At this point, in the non-archimedean case, the corollary is a consequence of \cref{lem:analytification_formal_properties}(\ref{item:analytification_affines}) and of \cite[Theorem 5.4]{Porta_Yu_Derived_non-archimedean_analytic_spaces}.
	In the complex case, the corollary is a consequence of \cref{lem:analytification_formal_properties}(\ref{item:analytification_affines}) and of \cite[7.3.2.11]{HTT}.
\end{proof}

Let us recall that a derived algebraic \DM stack over $k$ is by definition a $\cTetk$-scheme, which is in particular a $\cTetk$-structured topos (cf.\ \cite[4.3.20]{DAG-V}).
We refer to \cite{Porta_Comparison_2017} for a comparison with the definition of \DM stack via functor of points.

\begin{defin}
	A derived algebraic \DM stack $X = (\cX, \cO_X)$ is said to be \emph{locally almost of finite presentation} if its truncation $\trunc(X)=(\cX,\pi_0(\cO_X))$ is an underived algebraic \DM stack of finite presentation, and $\pi_i(\cO_X)$ is a coherent $\pi_0(\cO_X)$-module for every $i$.
\end{defin}

\begin{lem} \label{lem:analytification_truncated_spaces}
	Let $X = (\cX, \cO_X)$ be a derived algebraic \DM stack locally almost of finite presentation over $k$.
	Let $X\an = (\cX\an, \cO_{X\an})$ be its analytification.
	Then $\trunc(X\an) = (\cX\an, \pi_0(\cO_{X\an}))$ is an underived analytic \DM stack.
\end{lem}

\begin{proof}
	By \cite[Lemma 2.1.3]{DAG-V}, the question is local on $X$.
	So we can assume $X$ to be affine.
	Furthermore, using \cref{lem:analytification_formal_properties}(\ref{item:truncation_functors}), we see that there is a canonical equivalence
	\[ \trunc(X\an) \simeq \Psi_0(\trunc(X)) . \]
	Since $X$ is an affine scheme, we can find an \emph{underived} pullback diagram of the following form:
	\[ \begin{tikzcd}
		\trunc(X) \arrow{r} \arrow{d} & \mathbb A^n_k \arrow{d} \\
		\Spec(k) \arrow{r}{0} & \mathbb A^m_k .
	\end{tikzcd} \]
	Let $Y$ denote the derived pullback of the above diagram.
	Then $\trunc(Y) \simeq \trunc(X)$.
	Unramifiedness of $\cTetk$ implies that
	\[ \begin{tikzcd}
		Y \arrow{r} \arrow{d} & \mathbb A^n_k \arrow{d} \\ \Spec(k) \arrow{r}{0} & \mathbb A^m_k
	\end{tikzcd} \]
	remains a pullback diagram when viewed in $\RTop(\cTetk)$.
	Since $(-)\an$ is a right adjoint, it follows that
	\[ \begin{tikzcd}
		Y\an \arrow{r} \arrow{d} & (\mathbb A^n_k)\an \arrow{d} \\ (\Spec(k))\an \arrow{r} & (\mathbb A^m_k)\an
	\end{tikzcd} \]
	is a pullback diagram in $\RHTop(\cTank)$.
	Using \cref{lem:analytification_formal_properties}(\ref{item:analytification_affines}), we see that $(\Spec(k))\an\allowbreak\simeq \Sp(k)$, $(\mathbb A^n_k)\an \simeq \mathbf A^n_k$ and $(\mathbb A^m_k)\an \simeq \mathbf A^m_k$.
	Moreover, \cref{cor:analytification_closed_immersions} implies that the the morphism $\Sp(k) \to \mathbf A^m_k$ is again a closed immersion.
	Since $\dAnk$ is closed in $\RHTop(\cTank)$ under pullback by closed immersions by \cite[Proposition 6.2]{Porta_Yu_Derived_non-archimedean_analytic_spaces} and \cite[Proposition 12.10]{DAG-IX}, we conclude that $Y\an$ is a derived analytic space.
	So it follows from \cite[Corollary 3.24]{Porta_Yu_Derived_non-archimedean_analytic_spaces} that $\trunc(Y\an)$ is an analytic space.
	Finally, using the chain of equivalences provided by \cref{lem:analytification_formal_properties}(\ref{item:analytification_truncation_adjointable})
	\[ \trunc(Y\an) \simeq \Psi_0(\trunc(Y)) \simeq \Psi_0(\trunc(X)) \simeq \trunc(X\an) , \]
	we conclude that $\trunc(X\an)$ is an analytic space.
\end{proof}

\begin{cor} \label{cor:analytification_truncated_spaces_II}
	Let $X$ be an underived algebraic \DM stack locally of finite presentation over $k$.
	Then $X\an$ is a derived analytic space and it is equivalent to the classical analytification of $X$.
\end{cor}

\begin{proof}
	The question is local on $X$ and we can therefore assume that $X$ is affine.
	Using \cref{lem:analytification_formal_properties}(\ref{item:analytification_truncation_adjointable}), we see that the structure sheaf of $X\an$ is discrete.
	Thus, $X\an \simeq \trunc(X\an)$ is an analytic space in virtue of \cref{lem:analytification_truncated_spaces}.
	Moreover, \cref{lem:analytification_formal_properties}(\ref{item:truncation_functors}) shows that $\trunc(X\an) \simeq \Psi_0(\trunc(X)) \simeq \Psi_0(X)$.
	Using the universal property of $\Psi_0$ and the fact that $X\an$ is an analytic space, we see that for every analytic space $Y$, there is an equivalence
	\[ \Map_{\Ank}(Y, X\an) \simeq \Map_{\RTop(\cTetk)}(Y\alg, X) . \]
	This shows that $X\an$ can be identified with the classical analytification of $X$.
\end{proof}

\begin{cor} \label{cor:analytification_closed_immersions_II}
	Let $j \colon X \to Y$ be a closed immersion of derived algebraic \DM stacks locally almost of finite presentation over $k$.
	Then $j\an \colon X\an \to Y\an$ is a closed immersion in $\RHTop(\cTank)$.
\end{cor}

\begin{proof}
	It is enough to prove that $\trunc(j\an) \colon \trunc(X\an) \to \trunc(Y\an)$ is a closed immersion.
	Since $\trunc(j\an) \simeq \Psi_0(\trunc(j))$, the statement is now a direct consequence of \cref{cor:analytification_truncated_spaces_II}.
\end{proof}

We are now ready to state and prove the main result of this section:

\begin{prop} \label{prop:analytification}
	Let $X = (\cX, \cO_X) \in \RTop(\cTetk)$ be a derived algebraic \DM stack locally almost of finite presentation over $k$.
	Then $X\an$ is a derived analytic space.
\end{prop}

\begin{proof}
	Using \cite[Lemma 2.1.3]{DAG-V}, we can reason \'etale locally on $X$ and therefore assume that $X$ is affine.
	Let $\dAff_k\afp$ denote the \infcat of derived affine $k$-schemes almost of finite presentation.
	Let $\cC$ be the full subcategory of $\dAff_k\afp$ spanned by those derived affines $X$ such that $X\an \in \dAnk$.
	Let us remark that $\cC$ has the following properties:
	\begin{enumerate}
		\item $\cC$ contains $\cTetk$ in virtue of \cref{lem:analytification_formal_properties}(\ref{item:analytification_affines}).
		\item $\cC$ is closed under pullbacks along closed immersions. Indeed, if
		\[ \begin{tikzcd}
			W \arrow[hook]{r} \arrow{d} & Z \arrow{d} \\
			Y \arrow[hook]{r}{j} & X
		\end{tikzcd} \]
		is a pullback diagram in $\dAff_k\afp$ and $j$ is a closed immersion, then the unramifiedness of $\cTetk$ implies that the image of this diagram in $\RTop(\cTetk)$ is again a pullback square.
		Since $(-)\an$ is a right adjoint, we see that
		\[ \begin{tikzcd}
			W\an \arrow{r} \arrow{d} & Z\an \arrow{d} \\ Y\an \arrow{r}{j\an} & X\an
		\end{tikzcd} \]
		is a pullback square in $\RHTop(\cTank)$.
		Using \cref{cor:analytification_closed_immersions_II}, we see that $j\an$ is a closed immersion.
		Since $\dAnk$ is closed under pullback along closed immersions in $\RHTop(\cTank)$ (see \cite[Proposition 6.2]{Porta_Yu_Derived_non-archimedean_analytic_spaces} for the non-archimedean case and \cite[Proposition 12.10]{DAG-IX} for the complex case), we conclude that if $X, Y, Z \in \cC$, then $W \in \cC$ as well.
		\item $\cC$ is closed under finite limits. Indeed, it follows from \cite[\S 6]{Porta_Yu_Derived_non-archimedean_analytic_spaces} that general pullbacks can be constructed in terms of products of affine spaces and pullbacks along closed immersions. Since $(-)\an$ takes $\mathbb A^n_k$ to $\mathbf A^n_k$ by \cref{lem:analytification_formal_properties}(\ref{item:analytification_affines}), we see that $(-)\an \colon \cC \to \dAnk$ commutes with products of affine spaces.
		Since $\cC$ is furthermore closed under pullbacks along closed immersions by the previous point, the conclusion follows.

		\item \label{item:analytification_retraction} $\cC$ is closed under retractions. Indeed, let $X \in \cC$ and let $Y \xrightarrow{j} X \xrightarrow{p} Y$ be a retraction diagram in $\dAff_k\afp$.
		By assumption, $X\an \in \dAnk$ and \cref{lem:analytification_truncated_spaces} shows that $\trunc(Y\an) \in \dAnk$.
		It is therefore sufficient to show that $\pi_i(\cO_{Y\an})$ is a coherent sheaf over $\pi_0(\cO_{Y\an})$.
		Nevertheless, $\pi_i(\cO_{Y\an})$ is a retract of $j\inv \pi_i(\cO_{X\an})$, which is locally of finite presentation over $j\inv \pi_0(\cO_{X\an})$.
		It follows that $\pi_i(\cO_{Y\an})$ is locally of finite presentation over $j\inv \pi_0(\cO_{X\an})$.
		Since $\pi_0(\cO_{Y\an})$ is a retract of $\pi_0(\cO_{X\an})$ and $\pi_i(\cO_{Y\an})$ has a canonical $\pi_0(\cO_{Y\an})$-structure, we conclude that $\pi_i(\cO_{Y\an})$ is of finite presentation over $\pi_0(\cO_{Y\an})$ as well.
		The conclusion now follows from the fact that $\pi_0(\cO_{Y\an})$ is coherent.
			\end{enumerate}
	
	Let now $X \in \dAff_k\afp$ and write $X \simeq \Spec(A)$ for a simplicial commutative $k$-algebra $A$ almost of finite presentation.
	We want to prove that $X \in \cC$.
	Since \cref{lem:analytification_truncated_spaces} guarantees that $\trunc(X\an)$ is an analytic space, we only have to show that $\pi_i(\cO_{X\an})$ is a coherent sheaf of $\pi_0(\cO_{X\an})$-modules.
	
	In particular, for every $n \ge 0$ the algebra $\tau_{\le n}(A)$ is a compact object in the $\infty$-category $\CRing_k^{\le n}$ of $n$-truncated simplicial commutative $k$-algebras.
	It follows that there exists a \emph{finite} diagram of free simplicial commutative $k$-algebras
	\[ F \colon I \to \CRing_k \]
	such that $\tau_{\le n} A$ is a retraction of $\tau_{\le n}(B)$, where
	\[ B \coloneqq \colim_I F \in \CRing_k . \]
	
	Since $\cC$ is closed under finite limits, we see that $\Spec(B) \in \cC$.
	Now, using \cref{lem:analytification_formal_properties}(\ref{item:truncation_functors}) we conclude that
	\[ (\mathrm t_{\le n}(\Spec(A)))\an \simeq \Psi_n( \mathrm t_{\le n}(\Spec(A)) ) \]
	is a retract of
	\[ \Psi_n(\mathrm t_{\le n} (\Spec(B))) \simeq \mathrm t_{\le n}(\Spec(B)\an) . \]
	Property (\ref{item:analytification_truncation_adjointable}) implies that this is a derived analytic space.
	Therefore, it follows from (\ref{item:analytification_retraction}) that $( \mathrm t_{\le n}( \Spec(A)) )\an$ is a derived analytic space as well.
	
	Since we further have that
	\[ \Psi_n(\mathrm t_{\le n}(\Spec(A))) \simeq \mathrm t_{\le n}( \Spec(A)\an ) , \]
	we conclude that $\pi_i(\cO_{\Spec(A)\an})$ is a coherent sheaf of $\pi_0(\cO_{\Spec(A)\an})$-modules for all $0 \le i \le n$.
	Repeating the same reasoning for every $n$, we now conclude that $\Spec(A)\an$ is a derived analytic space.
	The proof is thus complete.
\end{proof}

\section{Analytic modules} \label{sec:analytic_modules}

In this section we study modules over derived analytic rings.
The main result is \cref{thm:equivalence_of_modules}.
We refer to \cref{sec:introduction} for motivations of this result and for a sketch of the proof.

Let us introduce a few notations before stating the main theorem.

Let $\cX$ be an \inftopos.
In virtue of \cite[Example 3.1.6, Remark 4.1.2]{DAG-V}, we have an equivalence of \infcats induced by the evaluation on the affine line
\[\Strloc_{\cTdisck}(\cX)\xrightarrow{\sim}\Sh_{\CRing_k}(\cX),\]
where $\Sh_{\CRing_k}(\cX)$ denotes the \infcat of sheaves on $\cX$ with values in the $\infty$-category of simplicial commutative $k$-algebras.

This motivates the following definition:

\begin{defin}
	Let $\cX$ be an $\infty$-topos.
	We denote $\CRing_k(\cX)\coloneqq\Strloc_{\cTdisck}(\cX)$, and call it the \infcat of \emph{sheaves of simplicial commutative $k$-algebras on $\cX$}.
	We denote $\AnRing_k(\cX)\coloneqq\Strloc_{\cTank}(\cX)$, and call it the \infcat of \emph{sheaves of derived \kanal rings on $\cX$}.
	We have an algebraization functor
	\[(-)\alg\colon\AnRing_k(\cX)\to\CRing_k(\cX)\]
	induced by the analytification functor $\cTdisck\to\cTank$.
\end{defin}

\begin{defin} \label{def:abelian_group_objects}
	Let $\Ab$ be the $1$-category of abelian groups.
	Let $\cT_\Ab$ denote the opposite of the full subcategory of $\Ab$ spanned by free abelian groups of finite rank.
	Let $\cC$ be an $\infty$-category with finite products.
	The \emph{$\infty$-category of abelian group objects in $\cC$} is by definition the $\infty$-category
	\[ \Ab(\cC) \coloneqq \Fun^\times(\cT_\Ab, \cC) , \]
	where the right hand side denotes the full subcategory of $\Fun(\cT_\Ab, \cC)$ spanned by product-preserving functors.
\end{defin}

\begin{defin} \label{def:modules}
	For a $\cTdisck$-structured topos $X=(\cX,\cO_X)$, we define $\cO_X\Mod$ to be the \infcat of left $\cO_X$-module objects of $\Sh_{\DAb}(\cX)$, where $\cD(\Ab)$ denotes the derived \infcat of abelian groups.
	\end{defin}

\begin{defin} \label{def:analytic_modules}
	For a $\cTank$-structured topos $X=(\cX, \cO_X)$, we define $\cO_X\Mod\coloneqq \cO\alg_X\Mod$.
	In particular, an $\cO_X$-module is by definition an $\cO_X\alg$-module.
\end{defin}

The goal of this section is to prove the following result:

\begin{thm} \label{thm:equivalence_of_modules}
	Let $X = (\cX, \cO_X)$ be a derived analytic space.
	We have an equivalence of stable $\infty$-categories
	\[\cO_X\Mod \simeq \Sp(\Ab(\AnRing_k(\cX)_{/\cO_X})),\]
	where $\Sp(-)$ denotes the \infcat of spectrum objects in a given \infcat.
\end{thm}

We split the proof into several steps.

\subsection{Construction of the functor}

Let $X = (\cX, \cO_X)$ be a derived analytic space.

The transformation of pregeometries
\[ (-)\an \colon \cTdisck \to \cTank \]
induces a functor
\[ \Phi \colon \AnRing_k(\cX)_{/\cO_X} \to \CRing_k(\cX)_{/\cO_X\alg} . \]

Note that the following diagram is commutative by construction:
\begin{equation} \label{eq:Phi}
\begin{tikzcd}
	\AnRing_k(\cX)_{/\cO_X} \arrow{r}{\Phi} \arrow{d} & \CRing_k(\cX)_{/\cO_X\alg} \arrow{d} \\
	\AnRing_k(\cX) \arrow{r}{(-)\alg} & \CRing_k(\cX).
\end{tikzcd}
\end{equation}

\begin{lem}
	The functor $\Phi$ has the following properties:
	\begin{enumerate}
		\item It is conservative;
		\item It commutes with limits and with sifted colimits.
	\end{enumerate}
\end{lem}
\begin{proof}
	The first property follows from \cite[Proposition 11.9]{DAG-IX} in the complex analytic case and from \cite[Lemma 3.13]{Porta_Yu_Derived_non-archimedean_analytic_spaces} in the non-archimedean case.
	The second property is a consequence of \cite[Proposition 1.17]{Porta_DCAGI}.
\end{proof}

\begin{lem} \label{lem:abelianization_of_functors}
	Let $f \colon \cC \to \cD$ be a functor between $\infty$-categories with finite products.
	If $f$ preserves finite products, then it induces a well-defined functor $\Ab(f) \colon \Ab(\cC) \to \Ab(\cD)$.
	Furthermore, if $f$ has one among the following properties:
	\begin{enumerate}
		\item $f$ is conservative;
		\item $f$ commutes with limits;
		\item $f$ commutes with sifted colimits;
	\end{enumerate}
	then $\Ab(f)$ has the same property.
\end{lem}

\begin{proof}
	Unraveling the definitions we see that composition with $f$ restricts to a well-defined functor
	\[ \Ab(\cC) = \Fun^\times(\cT_\Ab, \cC) \to \Fun^\times(\cT_\Ab, \cD) = \Ab(\cD) . \]
	This functor fits into a commutative diagram
	\[ \begin{tikzcd}
	\Fun^\times(\cT_\Ab, \cC) \arrow{r}{\Ab(f)} \arrow[hook]{d} & \Fun^\times(\cT_\Ab, \cD) \arrow[hook]{d} \\
	\Fun(\cT_\Ab, \cC) \arrow{r}{f_*} & \Fun(\cT_\Ab, \cD) .
	\end{tikzcd} \]
	The vertical morphisms are fully faithful and furthermore they commute with limits and with sifted colimits.
	Observe now that if $f$ has one of the listed properties, then $f_*$ shares the same property for formal reasons.
	The commutativity of the above diagram allows then to deduce that also $\Ab(f)$ inherits these properties.
\end{proof}

Since $\Phi$ commutes with limits, \cref{lem:abelianization_of_functors} implies that $\Phi$ induces a well-defined functor
\[ \Ab(\Phi) \colon \Ab\left( \AnRing_k(\cX)_{/\cO_X} \right) \to \Ab \left( \CRing_k(\cX)_{/\cO_X\alg} \right) . \]
Moreover, the functor $\Ab(\Phi)$ is conservative and commutes with limits and sifted colimits.

\begin{cor}
	The functor $\Ab(\Phi)$ induces a well-defined functor of stable $\infty$-categories
	\begin{equation} \label{eq:abelian_derivative_algebraization}
		\partial_\Ab(\Phi) \colon \Sp\left(\Ab\left(\AnRing_k(\cX)_{/\cO_X}\right)\right) \to \Sp\left(\Ab\left(\CRing_k(\cX)_{/\cO_X\alg}\right)\right) .
	\end{equation}
\end{cor}

\begin{proof}
	Recall from \cite[1.4.2.8]{Lurie_Higher_algebra} that given an $\infty$-category $\cC$, the $\infty$-category of spectra in $\cC$ is equivalent to
	\[ \Sp(\cC) \simeq \mathrm{Exc}_*(\cS^{\mathrm{fin}}_*, \cC) , \]
	where $\cS^{\mathrm{fin}}_*$ denotes the $\infty$-category of pointed finite spaces and $\mathrm{Exc}_*$ denotes the $\infty$-category of strongly excisive functors from $\cS^{\mathrm{fin}}_*$ to $\cC$, that is those functor $f \colon \cS^{\mathrm{fin}}_* \to \cC$ satisfying the following two conditions:
	\begin{enumerate}[(i)]
		\item $f$ takes final objects to final objects;
		\item $f$ takes pushout diagrams to pullback diagrams.
	\end{enumerate}
	Since $\Ab(\Phi)$ commutes with limits, it is clear that composition with $\Ab(\Phi)$ induces the functor \eqref{eq:abelian_derivative_algebraization}.
\end{proof}

By \cref{cor:modules_in_positive_characteristic}, we have an equivalence of stable \infcats
\[\cO_X\Mod\simeq\Sp(\Ab(\CRing_k(\cX)_{/A})).\]
Therefore, we can reduce \cref{thm:equivalence_of_modules} to the following theorem:

\begin{thm} \label{thm:equivalence_of_modules_precise}
	The functor
	\[ \partial_\Ab(\Phi) \colon \Sp(\Ab(\AnRing_k(\cX))_{/\cO_X} \to \Sp(\Ab(\CRing_k(\cX))_{/\cO\alg_X} \]
	is an equivalence of stable $\infty$-categories.
\end{thm}

\subsection{Reduction to connected objects} \label{sec:reduction_to_connected_objects}

By the construction of $\partial_\Ab(\Phi)$, in order to prove \cref{thm:equivalence_of_modules_precise}, it would be enough to prove that $\Ab(\Phi)$ is an equivalence.
In fact, it is sufficient to prove that $\Ab(\Phi)$ is an equivalence up to a finite number of suspensions.
Let us explain this reduction step precisely.

Observe that the functor $\Phi \colon \AnRing_k(\cX)_{/\cO_X} \to \CRing_k(\cX)_{/\cO_X\alg}$ induces a well-defined functor
\[ \Phi_*\colon\AnRing_k(\cX)_{\underover{\cO_X}} \to \CRing_k(\cX)_{\underover{\cO_X\alg}}.\]

\begin{lem} \label{lem:pointed_objects}
	\begin{enumerate}
		\item Let $\cC$ be an $\infty$-category with finite limits and let $*_\cC$ denote a final object for $\cC$.
		Write $\cC_* \coloneqq \cC_{*_\cC /}$.
		Then the forgetful functor $\cC_* \to \cC$ induces equivalences
		\[ \Ab(\cC_*) \to \Ab(\cC) \quad \textrm{and} \quad \Sp(\cC_*) \to \Sp(\cC) . \]
		\item Let $f \colon \cC \to \cD$ be a functor between $\infty$-categories with finite limits.
		Suppose that $f$ commutes with finite limits.
		Then $f$ induces a well-defined functor $f_* \colon \cC_* \to \cD_*$.
		Moreover, the diagrams
		\[ \begin{tikzcd}
			\Ab(\cC_*) \arrow{r}{\Ab(f_*)} \arrow{d} & \Ab(\cD_*) \arrow{d} \\
			\Ab(\cC) \arrow{r}{\Ab(f)} & \Ab(\cD)
		\end{tikzcd} \quad \textrm{and} \quad \begin{tikzcd}
			\Sp(\cC_*) \arrow{r}{\partial(f_*)} \arrow{d} & \Sp(\cD_*) \arrow{d} \\
			\Sp(\cC) \arrow{r}{\partial(f)} & \Sp(\cD)
		\end{tikzcd} \]
		commute. In particular, $\Ab(f)$ (resp.\ $\partial(f)$) is an equivalence if and only if $\Ab(f_*)$ (resp.\ $\partial(f_*)$) is one.
	\end{enumerate}
\end{lem}

\begin{proof}
	The forgetful functor $\cC_* \to \cC$ commutes with limits.
	Therefore, the existence of $f_*$ is a consequence of \cref{lem:abelianization_of_functors}.
	Knowing this, the second statement is a direct consequence of the first one.
	
	We now prove the first statement. The case of spectra has been discussed in \cite[1.4.2.25]{Lurie_Higher_algebra}.
	As for abelian groups, let $F \colon \cT_\Ab \to \cC$ be an $\infty$-functor that preserves products.
	Since $\cT_\Ab$ has a zero object, we see that $F$ factors canonically as
	\[ \widetilde{F} \colon \cT_\Ab \to \cC_* . \]
	This produces a functor $\Ab(\cC) \to \Ab(\cC_*)$ that is readily checked to be the inverse to the canonical functor $\Ab(\cC_*) \to \Ab(\cC)$.
\end{proof}

We need a digression on connected objects in $\infty$-categories.
We refer to \cite[5.5.6.18]{HTT} for the definition of truncation functors $\tau_{\le n}$ in a presentable $\infty$-category.

\begin{defin} \label{def:n-connected}
	Let $\cC$ be a presentable $\infty$-category.
	For any $n \ge 1$, we say that an object $X \in \cC$ is \emph{$n$-connected} if the canonical map $X \to *_\cC$ induces an equivalence
	\[ \tau_{\le n - 1} X \xrightarrow{\sim} *_\cC . \]
	We denote by $\cC^{\ge n}$ the full subcategory of $\cC$ spanned by $n$-connected objects.
\end{defin}

\begin{lem} \label{lem:connected_objects_abelian_groups}
	Let $\cC$ be a presentable $\infty$-category.
	Suppose that there exists an $\infty$-topos $\cX$ and a functor $F \colon \cC \to \cX$ such that:
	\begin{enumerate}
		\item $F$ is conservative;
		\item $F$ commutes with finite limits;
		\item $F$ commutes with the truncation functors.
	\end{enumerate}
	Then:
	\begin{enumerate}
		\item $\cC^{\ge n}$ is closed under finite products in $\cC$;
		\item there is a canonical equivalence of $\infty$-categories $\Ab(\cC^{\ge n}) \simeq \Ab(\cC)^{\ge n}$.
	\end{enumerate}
\end{lem}

\begin{proof}
	Recall from \cite[6.5.1.2]{HTT} that the truncation functor $\tau_{\le n} \colon \cX \to \cX$ commutes with finite products.
	The hypotheses on $F$ guarantee that the same goes for the truncation functor $\tau_{\le n} \colon \cC \to \cC$.
	At this point, the first statement follows immediately.
	
	Let us now prove the second statement.
	We start by recalling that there is an equivalence
	\[ \Ab(\cC) \simeq \Fun^\times( \cTAb, \cC) , \]
	where $\cTAb$ is the opposite category of free abelian groups of finite rank.
	We denote the free abelian group of rank $n$ by $A^n$.
	
	We claim that an object $F \in \Ab(\cC)$ belongs to $\Ab(\cC)^{\ge n}$ if and only if its image in $\cC$ belongs to $\cC^{\ge n}$.
	To see this, let $0 \in \Ab(\cC)$ denote the constant functor associated to $*_\cC$.
	Let furthermore $F \colon \cTAb \to \cC$ be a product preserving functor.
	Since $\tau_{\le n}$ commutes with finite products, $\tau_{\le n} \circ F$ is again a product preserving functor.
	It follows that the morphism $\tau_{\le n} \circ F \to 0$ is an equivalence if and only if it is an equivalence when evaluated on $A^1 \in \cTAb$.
	Since the forgetful functor $\Ab(\cC) \to \cC$ coincides (by definition) with the evaluation at $A^1$, this completes the proof of the claim.
	
	Now we remark that statement (1) implies that the inclusion
	\[ i \colon \cC^{\ge n} \hookrightarrow \cC \]
	commutes with finite products.
	Using \cite[Lemma 5.2]{Gepner_Lax_2015}, we see that the induced functor
	\[ \Fun(\cTAb, \cC^{\ge n}) \to \Fun(\cTAb, \cC) \]
	is fully faithful.
	It follows that the induced functor
	\[ \Ab(i) \colon \Ab(\cC^{\ge n}) \to \Ab(\cC) \]
	is fully faithful as well.
	Moreover, the diagram
	\[ \begin{tikzcd}
		\Ab(\cC^{\ge n}) \arrow{r} \arrow{d} & \Ab(\cC) \arrow{d} \\
		\cC^{\ge n} \arrow{r}{i} & \cC
	\end{tikzcd} \]
	commutes.
	It follows that $\Ab(i)$ factors through
	\[ j \colon \Ab(\cC^{\ge n}) \to \Ab(\cC)^{\ge n} , \]
	and that also $j$ is fully faithful.
	We are left to prove that $j$ is essentially surjective.
	Let $F \in \Ab(\cC)^{\ge n}$.
	Then by the above claim, the image of $F$ in $\cC$ belongs to $\cC^{\ge n}$.
	We can therefore see $F$ as an element in $\Ab(\cC^{\ge n})$, thus completing the proof.
\end{proof}

Since the functor
\[ \Phi_* \colon \AnRing_k(\cX)_{\underover{\cO_X}} \to \CRing_k(\cX)_{\underover{\cO_X\alg}} \]
commutes with limits and sifted colimits, it admits a left adjoint
\[ \Psi_* \colon \CRing_k(\cX)_{\underover{\cO_X\alg}} \to \AnRing_k(\cX)_{\underover{\cO_X}} . \]

\begin{lem}
	The functor $\Psi_*$ takes $\CRing_k(\cX)_{\underover{\cO_X\alg}}^{\ge 1}$ to $\AnRing_k(\cX)_{\underover{\cO_X}}^{\ge 1}$, where $(-)^{\ge 1}$ is in the sense of \cref{def:n-connected}.
\end{lem}

\begin{proof}
	It is enough to remark that the functor
	\[ \pi_0 \circ \Psi_* \colon \CRing_k(\cX)_{\underover{\cO_X\alg}} \to \AnRing_k^{\le 0}(\cX)_{\underover{\pi_0(\cO_X)}} \]
	is naturally equivalent to the functor
	\[ \pi_0 \circ \Psi_* \circ \pi_0 \colon \CRing_k^{\le 0}(\cX)_{\underover{\pi_0(\cO_X\alg)}} \to \AnRing_k^{\le 0}(\cX)_{\underover{\pi_0(\cO_X)}},\]
	where $(-)^{\le 0}$ denotes the full subcategory spanned by 0-truncated objects (cf.\ \cite[5.5.6.1]{HTT}).
\end{proof}

In particular, $\Psi_*$ induces a functor
\[ \Psi_*^{\ge 1} \colon \CRing_k(\cX)_{\underover{\cO_X\alg}}^{\ge 1} \to \AnRing_k(\cX)_{\underover{\cO_X}}^{\ge 1} , \]
and moreover $\Psi_*^{\ge 1}$ is a left adjoint to $\Phi_*^{\ge 1}$.

The main goal of this subsection is to reduce the proof of \cref{thm:equivalence_of_modules_precise} to the following statement:

\begin{thm} \label{thm:connected_closed_immersion_are_algebraic}
	The adjoint functors
	\[ \Phi_*^{\ge 1} \colon \AnRing_k(\cX)_{\underover{\cO_X}}^{\ge 1} \leftrightarrows \CRing_k(\cX)_{\underover{\cO_X\alg}}^{\ge 1} \colon \Psi_*^{\ge 1} \]
	form an equivalence.
\end{thm}

The next two subsections will be devoted to the proof of \cref{thm:connected_closed_immersion_are_algebraic}.
Now let us explain how to deduce \cref{thm:equivalence_of_modules_precise} from \cref{thm:connected_closed_immersion_are_algebraic}:

\begin{proof}[Proof of \cref{thm:equivalence_of_modules_precise} assuming \cref{thm:connected_closed_immersion_are_algebraic}]
	Since $\Phi_*^{\ge 1}$ is an equivalence, the same goes for
	\[ \Ab(\Phi_*^{\ge 1}) \colon \Ab\left(\AnRing_k(\cX)_{\underover{\cO_X}}^{\ge 1}\right) \to \Ab\left( \CRing_k(\cX)_{\underover{\cO_X\alg}}^{\ge 1} \right) . \]
	Notice that \cref{thm:connected_closed_immersion_are_algebraic} guarantees, in particular, that $\Psi_*^{\ge 1}$ commutes with finite limits.
	In particular, composition with $\Psi_*^{\ge 1}$ induces a well-defined functor
	\[ \Ab(\Psi_*^{\ge 1}) \colon \Ab\left( \CRing_k(\cX)_{\underover{\cO_X\alg}}^{\ge 1} \right) \to \Ab\left(\AnRing_k(\cX)_{\underover{\cO_X}}^{\ge 1}\right) \]
	which is left adjoint to $\Ab(\Phi_*^{\ge 1})$.

	In order to prove that
	\[ \partial_\Ab(\Phi_*) \colon \Sp\left(\Ab\left(\AnRing_k(\cX)_{\underover{\cO_X}}\right)\right) \to \Sp\left(\Ab\left(\CRing_k(\cX)_{\underover{\cO_X\alg}}\right)\right) \]
	is an equivalence, it is enough to prove that for any
	\[ M \in \Ab\left(\CRing_k(\cX)_{\underover{\cO_X\alg}}\right) \]
	the canonical map
	\[ \Sigma(M) \to \Ab(\Phi_*)( \Ab(\Psi_*)(\Sigma(M))) \]
	is an equivalence.
	Here $\Sigma$ denotes the suspension functor (see the discussion around \cite[1.1.2.6]{Lurie_Higher_algebra}).
	
	Notice that the natural inclusion
	\[ \CRing_k(\cX)_{\underover{\cO_X\alg}} \hookrightarrow \Fun(\cTdisck, \cX)_{/\cO_X} \]
	is conservative, commutes with limits and with truncations.
	In particular, we can apply \cref{lem:connected_objects_abelian_groups} to deduce the equivalence
	\[ \Ab\left( \CRing_k(\cX)_{\underover{\cO_X\alg}} \right)^{\ge 1} \simeq \Ab\left( \CRing_k(\cX)_{\underover{\cO_X\alg}}^{\ge 1} \right) . \]
	Observe now that
	\[ \Sigma(M) \in \Ab\left( \CRing_k(\cX)_{\underover{\cO_X\alg}} \right)^{\ge 1} . \]
		In particular
	\[ \Ab(\Psi_*)(\Sigma(M)) \simeq \Ab(\Psi_*^{\ge 1})(\Sigma(M)) \in \Ab\left( \AnRing_k(\cX)_{\underover{\cO_X\alg}}^{\ge 1} \right). \]
	As a consequence,
	\[ \Ab(\Phi_*)( \Ab(\Psi_*)(\Sigma(M)) \simeq \Ab(\Phi_*^{\ge 1})( \Ab(\Psi_*^{\ge 1})(\Sigma(M) ) . \]
	Since $\Ab(\Phi_*^{\ge 1})$ is an equivalence and $\Ab(\Psi_*^{\ge 1})$ is its left adjoint, the conclusion follows.
\end{proof}

\subsection{Reduction to the case of spaces} \label{sec:reduction_to_spaces}

Here we explain how to reduce the proof of \cref{thm:connected_closed_immersion_are_algebraic} to the case where $\cX$ is the $\infty$-category of spaces $\cS$.

In order to prove \cref{thm:connected_closed_immersion_are_algebraic}, it is enough to prove that the pair of functors $(\Psi_*^{\ge 1}, \Phi_*^{\ge 1})$ form an equivalence of categories.
Fix a geometric point $x\inv \colon \cX \leftrightarrows \cS \colon x_*$.
Invoking \cite[Theorem 1.12]{Porta_DCAGII} we conclude that the induced diagram
\[ \begin{tikzcd}
	\AnRing_k(\cX)_{\underover{\cO_X}}^{\ge 1} \arrow{d}{\Phi_*^{\ge 1}} \arrow{r}{x\inv} & \AnRing_k(\cS)_{\underover{x\inv \cO_X}}^{\ge 1} \arrow{d}{\Phi_*^{\ge 1}} \\
	\CRing_k(\cX)_{\underover{\cO_X\alg}}^{\ge 1} \arrow{r}{x\inv} & \CRing_k(\cS)_{\underover{x\inv\cO_X\alg}}^{\ge 1}
\end{tikzcd} \]
commutes and it is left adjointable.
Since $\cX$ has enough points (see \cite[Remark 3.3]{Porta_Yu_Derived_non-archimedean_analytic_spaces}), we see that it is enough to check that the adjunction
\[ \Phi_*^{\ge 1} \colon \AnRing_k(\cS)_{\underover{x\inv \cO_X}}^{\ge 1} \leftrightarrows \CRing_k(\cS)_{\underover{x\inv \cO_X\alg}}^{\ge 1} \colon \Psi_*^{\ge 1} \]
is an equivalence.
We can therefore take $\cX = \cS$.
To ease the notations, we set
\[ A \coloneqq x\inv \cO_X . \]
Furthermore, we write $\AnRing_k$ instead of $\AnRing_k(\cS)$, and similarly we write $\CRing_k$ for $\CRing_k(\cS)$.

\subsection{Flatness} \label{sec:flatness}

Here we will achieve the proof of \cref{thm:connected_closed_immersion_are_algebraic}, i.e.\ the functor
\[ \Phi_*^{\ge 1} \colon \AnRing_{\underover{A}}^{\ge 1} \to \CRing_{\underover{A\alg}}^{\ge 1} \]
is an equivalence.
We already observed that $\Phi_*^{\ge 1}$ has a left adjoint $\Psi_*^{\ge 1}$.
Furthermore, we know that $\Phi_*$ is conservative, and hence so is $\Phi_*^{\ge 1}$.
Therefore, it is enough to prove that for every $B \in \CRing_{\underover{A\alg}}^{\ge 1}$, the unit transformation
\[ \eta \colon B \to \Phi_*^{\ge 1}( \Psi_*^{\ge 1}(B) ) \]
is an equivalence.
Notice that
\[ \pi_0(B) \simeq \pi_0(A\alg) \simeq \pi_0 ( \Phi_*^{\ge 1}( \Psi_*^{\ge 1}(B) ) ) . \]
In particular, $\pi_0(\eta)$ is an isomorphism.
In order to complete the proof of \cref{thm:connected_closed_immersion_are_algebraic}, it is therefore sufficient to prove the following result:

\begin{prop} \label{prop:flatness}
	For every $B \in \CRing_{\underover{A\alg}}$, the canonical map
	\[ \eta \colon B \to \Phi_*( \Psi_*(B) ) \]
	is a flat map of simplicial commutative rings.
	\end{prop}

\begin{notation}
	In order to ease notation, in virtue of the commutative diagram \eqref{eq:Phi}, let us denote from now on $\Phi_*$ by $(-)\alg$.
	Moreover, let us denote $\Psi_*$ by $(-)\an_A$ and call it the \emph{functor of analytfication relative to} $A$.
\end{notation}

\begin{rem}
	In the complex case, a proof of the above result already appeared in \cite[Appendix A]{Porta_DCAGII}.
	In this section, we expand the proof given in loc.\ cit.\ and we introduce slight modifications in order to obtain a uniform proof that works both in the non-archimedean and in the complex case.
\end{rem}

The proof of \cref{prop:flatness} occupies the remaining of this subsection.
We start by introducing the full subcategory $\cC_A$ of $\CRing_{\underover{A\alg}}$ spanned by those $B \in \CRing_{\underover{A\alg}}$ such that the canonical map
\[ B \to (B\an_A)\alg \]
is flat.
We observe that $\cC_A$ is closed under various operations:

\begin{lem} \label{lem:sorite_analytification_flat_algebras}
	The full subcategory $\cC_A$ enjoys the following properties:
	\begin{enumerate}
		\item $A\alg \in \cC_A$;
		\item $\cC_A$ is closed under retracts;
		\item $\cC_A$ is closed under filtered colimits;
		\item Let $R \to T$ is an effective epimorphism in $\CRing_{\underover{A\alg}}$ such that the square
		\[ \begin{tikzcd}
			R \arrow{r} \arrow{d} & T \arrow{d} \\
			(R\an_A)\alg \arrow{r} & (T\an_A)\alg
		\end{tikzcd} \]
		is a pushout.
		Let $R \to B$ be any map in $\CRing_{\underover{A\alg}}$.
		If $B$, $R$ and $T$ belong to $\cC_A$, then so does the pushout $B \otimes_R T$.
	\end{enumerate}
\end{lem}

\begin{proof}
	Statement (1) follows directly from the fact that $(A\alg)\an_A \simeq A$.
	Statement (2) follows because flat maps are stable under retracts.
	Statement (3) is a consequence of the following two facts: on one side, flat maps are stable under filtered colimits and, on the other side, the functors $(-)\alg$ and $(-)\an_A$ commute with filtered colimits.
	We now prove statement (4).
	Set $C \coloneqq B \otimes_R T$ and consider the commutative cube
	\[ \begin{tikzcd}[row sep=small,column sep=small]
		R \arrow{dd} \arrow{dr} \arrow{rr} & & T \arrow{dd} \arrow{dr} \\
		{} & (R\an_A)\alg \arrow[crossing over]{rr} & & (T\an_A)\alg \arrow{dd} \\
		B \arrow{rr} \arrow{dr} & & C \arrow{dr} \\
		{} & (B\an_A)\alg \arrow[leftarrow, crossing over]{uu} \arrow{rr} & & (C\an_A)\alg .
	\end{tikzcd} \]
	Since $(-)\an_A$ is a left adjoint, we see that the diagram
	\[ \begin{tikzcd}
		R\an_A \arrow{r} \arrow{d} & T\an_A \arrow{d} \\
		B\an_A \arrow{r} & C\an_A
	\end{tikzcd} \]
	is a pushout diagram in $\AnRing_{\underover{A}}$.
	Moreover, since the top square in the above cube is a pushout by assumption, we see that the map $R\an_A \to T\an_A$ is an effective epimorphism.
	Therefore, the unramifiedness of $\cTank$ implies that the front square in the above cube is a pushout as well (cf.\ \cite[Corollary 3.11 and Proposition 3.17]{Porta_Yu_Derived_non-archimedean_analytic_spaces}).
	It follows that the outer square in the diagram
	\[ \begin{tikzcd}
		R \arrow{r} \arrow{d} & B \arrow{d} \arrow{r} & (B\an_A)\alg \arrow{d} \\
		T \arrow{r} & C \arrow{r} & (C\an_A)\alg
	\end{tikzcd} \]
	is a pushout.
	Since the square on the left is a pushout by construction, we conclude that the right square is a pushout as well.
	Since flat maps are stable under base change and $B \to (B\an_A)\alg$ is flat, we deduce that the same goes for $C \to (C\an_A)\alg$.
	In other words, $C \in \cC_A$.
\end{proof}

Motivated by statement (4) in the above lemma, we introduce the following temporary definition:

\begin{defin}
	Let $p \colon R \to T$ be an effective epimorphism in $\CRing_{\underover{A\alg}}$.
	We say that $p$ has the property ($P_A$) if the diagram
	\[ \begin{tikzcd}
		R \arrow{r} \arrow{d} & T \arrow{d} \\
		(R\an_A)\alg \arrow{r} & (R\an_A)\alg
	\end{tikzcd} \]
	is a pushout.
\end{defin}

With this terminology, \cref{lem:sorite_analytification_flat_algebras} immediately implies the following:

\begin{cor} \label{cor:flatness_reduction_to_generators}
	Suppose that there exists a collection of objects $S = \{B_\alpha\}_{\alpha \in I}$ in $\CRing_{\underover{A\alg}}$ such that:
	\begin{enumerate}
		\item every object in $\CRing_{\underover{A\alg}}$ is a retract of an $S$-cell complex;
		\item the structural morphisms $B_\alpha \to A\alg$ have the property ($P_A$);
		\item each $B_\alpha$ belongs to $\cC_A$.
	\end{enumerate}
	Then $\cC_A = \CRing_{\underover{A\alg}}$.
\end{cor}

We are therefore reduced to find a set $S$ of objects in $\CRing_{\underover{A\alg}}$ with the above properties.
In order to achieve this goal, we need a further reduction step: we want to replace $A\alg$ with the ring of germs of holomorphic functions at any geometric point of $\bD^n_k$ in the non-archimedean case, and of $\bA^n_{\mathbb C}$ in the complex case.

We start by observing that the collection of $A\alg$-linear morphisms
\[ A\alg[X_1, \ldots, X_m] \to A\alg \]
for various $m$ form a set $S_A$ of elements in $\CRing_{\underover{A\alg}}$ with the property that every other object is a retract of an $S$-cell complex.
The following is the key reduction step:

\begin{lem} \label{lem:flatness_smooth_reduction}
	Let $f \colon R \to A$ be an effective epimorphism in $\AnRing_k$.
	The following holds:
	\begin{enumerate}
		\item If $B \in \CRing_{\underover{R\alg}}$ belongs to $\cC_R$, then $B \otimes_{R\alg} A\alg \in \CRing_{\underover{A\alg}}$ belongs to $\cC_A$;
		\item If $B \to C$ is an effective epimorphism in $\CRing_{\underover{R\alg}}$ that satisfies the property ($P_R$), then the induced morphism
		\[ B \otimes_{R\alg} A\alg \to C \otimes_{R\alg} A\alg \in \CRing_{\underover{A\alg}} \]
		satisfies the property ($P_A$).
		\item The base change functor
		\[ - \otimes_{R\alg} A\alg \colon \CRing_{\underover{R\alg}} \to \CRing_{\underover{A\alg}} \]
		takes $S_R$ to $S_A$.
		Furthermore every object in $S_A$ lies in the essential image of $S_R$ via this functor.
	\end{enumerate}
\end{lem}

\begin{proof}
	We start by proving (1).
	Denote by $(-)\an$ the left adjoint to the underlying algebra functor
	\[ (-)\alg \colon \AnRing_{/A} \to \CRing_{/A} . \]
	We therefore obtain the following commutative cube:
	\[ \begin{tikzcd}[row sep=small,column sep=small]
		(R\alg)\an \arrow{rr} \arrow{dd} \arrow{dr} & & B\an \arrow{dd} \arrow{dr} \\
		{} & R \arrow[crossing over]{rr} & & B\an_R \arrow{dd} \\
		(A\alg)\an \arrow{rr} \arrow{dr} & & C\an \arrow{dr} \\
		{} & A \arrow{rr} \arrow[leftarrow, crossing over]{uu} & & C\an_A .
	\end{tikzcd} \]
	The universal property of the relative analytifications $(-)\an_R$ and $(-)\an_A$ shows that the top and the bottom squares are pushout squares.
	Furthermore, since $(-)\an$ is a left adjoint, we see that the square on the back is a pushout as well.
	The transitivity property for pushouts implies that the front square is a pushout.

	Since the morphism $f \colon R \to A$ is an epimorphism, unramifiedness of $\cTank$ implies that the functor $(-)\alg$ preserves the pushout in the front.
	Consider now the following commutative diagram
	\[ \begin{tikzcd}
		R\alg \arrow{r} \arrow{d} & B \arrow{r} \arrow{d} & (B\an_R)\alg \arrow{d} \\
		A\alg \arrow{r} & C \arrow{r} & (C\an_A)\alg .
	\end{tikzcd} \]
	The square on the left is a pushout by definition, and we proved above that the outer square is also a pushout.
	It follows that the left square is a pushout as well.
	Since $B \to (B\an_R)\alg$ is flat, it follows that the same goes for $C \to (C\an_A)\alg$, thus completing the proof.
	
	We now prove statement (2).
	Consider the commutative cube
	\[ \begin{tikzcd}[row sep=small,column sep=small]
		B \arrow{rr} \arrow{dd} \arrow{dr} & & C \arrow{dd} \arrow{dr} \\
		{} & (B\an_R)\alg \arrow[crossing over]{rr} & & (C\an_R)\alg \arrow{dd} \\
		B \otimes_{R\alg} A\alg \arrow{rr} \arrow{dr} & & C \otimes_{R\alg} A\alg \arrow{dr} \\
		{} & ((B \otimes_{R\alg} A\alg)\an_A)\alg \arrow{rr} \arrow[leftarrow, crossing over]{uu} & & ((C \otimes_{R\alg} A\alg)\an_A)\alg .
	\end{tikzcd} \]
	The hypotheses guarantee that the top and the back squares are pushout.
	As a consequence, we deduce that the map $B\an_R \to C\an_R$ is an effective epimorphism.
	We claim that the front square is a pushout as well.
	Indeed, we have the following commutative diagram:
	\[ \begin{tikzcd}
		R  \arrow{r} \arrow{d} & B\an_R \arrow{r} \arrow{d} & C\an_R \arrow{d} \\
		A \arrow{r} & ( B \otimes_{R\alg} A\alg )\an_A \arrow{r} & ( C \otimes_{R\alg} A\alg )\an_A .
	\end{tikzcd} \]
	The argument given in the proof of statement (1) implies that the outer and the left squares are pushout.
	Therefore, the same goes for the square on the right.
	Since $B\an_R \to C\an_R$ is an effective epimorphism, the unramifiedness of $\cTank$ guarantees that $(-)\alg$ commutes with this pushout.
	Therefore, the front square in the previous commutative cube is a pushout as well.
	The transitivity property of pushout squares implies then that the bottom square is also a pushout.
	In other words, the map
	\[ B \otimes_{R\alg} A\alg \to C \otimes_{R\alg} A\alg \]
	has the property ($P_A$).
	
	Finally, we prove statement (3). Let
	\[ p \colon A\alg[X_1, \ldots, X_m] \to A\alg \]
	be an $A\alg$-linear morphism.
	This morphism chooses $m$ elements $a_1, \ldots, a_m \in \pi_0(A\alg)$.
	Since the map $\pi_0(p) \colon \pi_0(R\alg) \to \pi_0(A\alg)$ is surjective, we can find elements $r_1, \ldots, r_m \in \pi_0(R\alg)$ such that
	\[ \pi_0(p)( r_i ) = a_i , \]
	for $1 \le i \le m$.
	We can now choose a morphism
	\[ q \colon R\alg[X_1, \ldots, X_m] \to R\alg \]
	that selects the elements $r_1, \ldots, r_m$.
	Applying the base change functor $- \otimes_{R\alg} A\alg$ we obtain a new map
	\[ p' \colon A\alg[X_1, \ldots, X_m] \to A\alg . \]
	Observe that both $p$ and $p'$ define elements in
	\[ \pi_0 \Map_{\CRing_{A\alg}}( A\alg[X_1, \ldots, A_m], A\alg ) \simeq \pi_0(A\alg)^m . \]
	The construction reveals that $p$ and $p'$ coincide as element in the above set.
	In other words, we can find a homotopy $p \simeq p'$ in $\CRing_{A\alg}$.
	This completes the proof.
	\end{proof}

Combining \cref{lem:flatness_smooth_reduction} and \cref{cor:flatness_reduction_to_generators}, we deduce that whenever $R \to A$ is an effective epimorphism in $\AnRing_k$, if $\cC_R = \CRing_{\underover{R\alg}}$ holds, then $\cC_A = \CRing_{\underover{A\alg}}$ holds as well.

We now use the hypothesis that $A$ is the stalk of a derived analytic space $X = (\cX, \cO_X)$ at a geometric point $x_* \colon \cS \leftrightarrows \cX \colon x\inv$.
In particular, using \cite[Lemma 6.3]{Porta_Yu_Derived_non-archimedean_analytic_spaces} in the non-archimedean case and \cite[Proposition 12.13]{DAG-IX} in the complex case, we can suppose that $X$ admits a closed embedding into a smooth analytic space:
\[ j \colon X \hookrightarrow U . \]
In the non-archimedean case, we can take $U$ to be a polydisk $\mathbf D^n_k$, while in the complex case we can take $U$ to be an affine space $\mathbf A^n_{\mathbb C}$.
In either case, let
\[ y_* \colon \cS \leftrightarrows \cX \colon y\inv \]
be the geometric point defined as the composition $y_* \coloneqq j_* \circ x_*$.
Set
\[ R \coloneqq y\inv \cO_U \]
and observe that the induced map $f \colon R \to A$ is an effective epimorphism.
The above argument allows us to replace $A$ by $R$.
In other words, we can assume from the very beginning that $A$ is of the form $x\inv \cO_U$ for some geometric point of $U$, where $U$ is a polydisk $\bD^n_k$ in the non-archimedean case and it is $\mathbf A^n_k$ in the complex case.
Using \cref{cor:flatness_reduction_to_generators}, we are therefore reduced to prove that for every $A\alg$-linear morphism
\[ f \colon A\alg[X_1, \ldots, X_m] \to A\alg \]
the following properties are verified:
\begin{enumerate}
	\item $A\alg[X_1, \ldots, X_m]$ belongs to $\cC_A$.
	\item the morphism $f$ has the property ($P_A$);
\end{enumerate}
In order to prove these statements, we need a geometric characterization of the relative analytification
\[ A\alg[X_1, \ldots, X_m]\an_A \in \AnRing_{\underover{A}} . \]

The map $f \colon A\alg[X_1, \ldots, X_m] \to A\alg$ selects $m$ elements $a_1, \ldots, a_m \in A\alg$.
Since $A = x\inv \cO_U$ is the ring of germs of holomorphic functions around the point $x$, we can find an \'etale neighborhood $V$ of $x$ so that the elements $a_1, \ldots , a_m$ extend to holomorphic functions $\widetilde{a_1}, \ldots, \widetilde{a_m}$ on $V$.
In both cases, we can interpret these holomorphic functions as a section of the relative algebraic space
\[ \pi \colon V\alg \times \mathbb A^m_k \to V\alg . \]
We denote the section determined by the functions $\widetilde{a_1}, \ldots, \widetilde{a_m}$ by $s \colon V\alg \to V\alg \times \mathbb A^m_k$.
The analytification relative to $V$ takes $s$ to a section
\[ s\an_V \colon V \to V \times \mathbf A^m_k . \]
Denote by $y$ the point of $V \times \mathbf A^m_k$ which is the image of the point $x \in V$ via $s\an_V$.
Since $V \times \mathbf A^m_k$ is the analytification of $V\alg \times \mathbb A^m_k$ relative to $V$, there is a canonical map
\[ q \colon (V \times \mathbf A^m_k)\alg \to V\alg \times \bbA^m_k \]
making the following diagram commutative:
\[ \begin{tikzcd}
	{} & V\alg \arrow{dl}[swap]{(s\an_V)\alg} \arrow{dr}{s} \\
	(V \times \mathbf A^m_k)\alg \arrow{rr}{q} \arrow{dr}[swap]{(\pi\an_V)\alg} & & V\alg \times \bbA^m_k \arrow{dl}{\pi} \\
	{} & V\alg .
\end{tikzcd} \]
By passing at the stalk at $x$ the map $q$ induces a well-defined map
\[ \alpha \colon A\alg[X_1, \ldots, X_m] \to x\inv \cO_{V \times \mathbf A^m_k}\alg \to y\inv \cO_{V \times \mathbf A^m_k}\alg . \]
We can now prove the following result:

\begin{prop} \label{prop:geometric_interpretation_relative_analytification}
	The map $\alpha \colon A\alg[X_1, \ldots, X_m] \to y\inv \cO_{V \times \mathbf A^n_k}\alg$ exhibits $y\inv \cO_{V \times \mathbf A^n_k}$ as analytification of $A\alg[X_1, \ldots, X_m]$ relative to $A$.
	In particular, it induces an equivalence
	\[ A\alg[X_1, \ldots, X_m]\an_A \simeq y\inv \cO_{V \times \bA^m_k} \]
	in $\AnRing_{\underover{A}}$.
\end{prop}

\begin{proof}
	In order to ease the notations, we set $R \coloneqq A\alg[X_1, \ldots, X_m]$ and $B \coloneqq y\inv \cO_{V \times \mathbf A^m_k}$.
	Furthermore, we denote by $\Map_{\underover{A}}$ and $\Map_{\underover{A\alg}}$ the mapping spaces in the $\infty$-categories $\AnRing_{\underover{A}}$ and $\CRing_{\underover{A\alg}}$, respectively.
	
	We have to check that for any $C \in \AnRing_{\underover{A}}$, the map
	\[ \Map_{\underover{A}}(B, C) \to \Map_{\underover{A\alg}}(R, C\alg) \]
	induced by $\alpha \colon R \to B\alg$ is an equivalence.
	Let us introduce the following temporary notation: given an object $C$ in either $\AnRing_k$ or in $\CRing_k$, we denote by $\cS_C$ the structured $\infty$-topos $(\cS, C)$.
	When $C \in \AnRing_k$, we set, as usual, $\cS_C\alg \coloneqq (\cS, C\alg)$.
	Moreover, we denote by $\Map_{\underover{\cS_A}}$ and $\Map_{\underover{\cS_A\alg}}$ the mapping spaces in the $\infty$-categories $\RTop(\cTank)_{\underover{\cS_A}}$ and $\RTop(\cTdisck)_{\underover{\cS_A\alg}}$, respectively.
	The very definition of the mapping spaces in $\RTop(\cTank)$ and in $\RTop(\cTdisck)$ induce following pair of natural equivalences:
	\[ \Map_{\underover{A}}(B, C) \simeq \Map_{\underover{\cS_A}}( \cS_C, \cS_B ) \]
	and
	\[ \Map_{\underover{A\alg}}(R, C\alg) \simeq \Map_{\underover{\cS_A\alg}}(\cS_C\alg,  \cS_R) . \]
	Finally, we represent the $\cTank$-structured topoi $V$ and $V \times \bA^m_k$ as the pairs $(\cV, \cO_V)$ and $(\cY, \cO_{V \times \bA^m_k})$, respectively.
	We represent the $\cTdisck$-structured topos $V\alg \times \bbA^n_k$ as the pair $(\cZ, \cO_{V\alg \times \bbA^m_k})$.
	Form the pullbacks of topoi
	\[ \begin{tikzcd}
		\cW_2 \arrow{d}{g_{2*}} \arrow{r} & \cW_1 \arrow{r} \arrow{d}{g_{1*}} & \cS \arrow{d}{x_*} \\
		\cY \arrow{r}{q_*} & \cZ \arrow{r}{\pi_*} & \cV .
	\end{tikzcd} \]
	Using \cite[Lemma 2.1.3]{DAG-V}, we see that $W_2 \coloneqq (\cW_2, g_2\inv \cO_{V \times \bA^m_k})$ is the analytification of $W_1 \coloneqq (\cW_1, g_1\inv \cO_{V\alg \times \bbA^m_k})$ relative to $\cS_A$.
	In particular, for every $C  \in \AnRing_{A//A}$, we obtain an equivalence
	\[ \Map_{\cS_A // \cS_A}(\cS_C, W_2) \simeq \Map_{\cS_A\alg // \cS_A\alg}(\cS_C\alg, W_1) . \]
	In order to complete the proof, it is now sufficient to show that there are equivalences
	\[ \Map_{\cS_A\alg // \cS_A\alg}(\cS_C\alg, W_1) \simeq \Map_{\cS_A\alg // \cS_A\alg}(\cS_C\alg, \cS_R) \]
	and
	\[ \Map_{\cS_A // \cS_A}(\cS_C, W_2) \simeq \Map_{\cS_A // \cS_A}(\cS_C, \cS_B) . \]
	We argue for the first one.
	The map
	\[ s \colon V\alg \to V\alg \times \bbA^m_k \]
	induces a map
	\[ s_1 \colon \cS_A \to W_1 , \]
	and there is a canonical equivalence
	\[ R \simeq s_1\inv g_1\inv \cO_{V\alg \times \bbA^m_k} . \]
	Consider the natural fiber sequence
	\[ \Map_{\underover{A\alg}}(s_1\inv g_1\inv \cO_{V\alg \times \bbA^n_k}, C\alg) \to \Map_{\underover{\cS_A\alg}}(\cS_C, W_1) \to \Map_{\RTop_{\underover{\cS}}}(\cS, \cW_1) . \]
	Since $\Map_{\RTop_{\underover{\cS}}}(\cS, \cW_1) \simeq *$, we conclude that the first map is an equivalence.
	The second equivalence is proved in a similar way.
	The proof is now complete.
\end{proof}

Now we move to the next step of the proof of \cref{prop:flatness}:

\begin{cor} \label{cor:flatness_smooth_algebras}
	For every $A\alg$-linear map
	\[ f \colon A\alg[X_1, \ldots, X_m] \to A\alg , \]
	the canonical map $\eta\colon A\alg[X_1, \ldots, X_m] \to A\alg[X_1, \ldots, X_m]\an_A$ is flat.
\end{cor}

\begin{proof}
	Using \cref{prop:geometric_interpretation_relative_analytification}, we can describe $A\alg[X_1, \ldots, X_m]\an_A$ as the ring of germs of holomorphic functions $y\inv \cO_{V \times \bA^n_k}$.
	
	Let us treat the non-archimedean case first.
	In this case, we have
	\[ A\alg \simeq k\langle T_1, \ldots, T_n \rangle_x , \]
	and
	\[ y\inv \cO_{V \times \bA^n_k} \simeq k\langle T_1, \ldots, T_n, X_1, \ldots, X_m \rangle_y . \]
	We have to prove that the canonical map
	\[ k\langle T_1, \ldots, T_n \rangle_x[X_1, \ldots, X_m] \to k\langle T_1, \ldots, T_n, X_1, \ldots, X_m \rangle_y \]
	is flat.
	Since the passage to germs preserves flatness, it is enough to prove that the map of commutative rings
	\[ i \colon k \langle T_1, \ldots, T_n \rangle[X_1, \ldots, X_m] \to k\langle T_1, \ldots, T_n, X_1, \ldots, X_m \rangle \]
	is flat.
	Since both rings are noetherian, it is enough to check flatness after passing to the formal completion at every maximal ideal of $k\langle T_1, \ldots, T_n, X_1, \ldots, X_m \rangle$.
	If $\mathfrak m$ is such a maximal ideal, then we have equivalences
	\[ (k\langle T_1, \ldots, T_n \rangle[X_1, \ldots, X_m])^{\wedge}_{i\inv(\mathfrak m)} \simeq \kappa(\mathfrak m)\llb T_1, \ldots, T_n, X_1, \ldots, X_m \rrb \]
	and
	\[ ( k\langle T_1, \ldots, T_n, X_1, \ldots, X_m \rangle )^{\wedge}_{\mathfrak m} \simeq \kappa(\mathfrak m)\llb T_1, \ldots, T_n, X_1, \ldots, X_m \rrb, \]
	where $\kappa(\mathfrak m)$ denotes the residue field.
	It follows that $i$ induces an isomorphism on the formal completions, and therefore that $i$ is flat.
	
	Let us now deal with the complex case.
	In this case, we have
	\[ A\alg \simeq \mathbb C\{ T_1, \ldots, T_n\}_x \]
	and
	\[ y\inv \cO_{V \times \bA^n_k} \simeq \mathbb C\{T_1, \ldots, T_n, X_1, \ldots, X_m\}_y \]
	where the right hand side denote the rings of germs of holomorphic functions on $V$ at $x$ and of $V \times \bA^n_k$ at $y$, respectively.
	Thus, we have to prove that the natural map
	\begin{equation} \label{eq:flatness_analytification_smooth_algebras_complex_case_I}
		\mathbb C\{ T_1, \ldots, T_n \}_x[X_1, \ldots, X_m] \to \mathbb C\{T_1, \ldots, T_n, X_1, \ldots, X_m\}_y
	\end{equation}
	is flat.
	Consider the map
	\[ f \colon \mathbb C\{T_1, \ldots, T_n\}_x[X_1, \ldots, X_m] \to \mathbb C\{T_1, \ldots, T_n\}_x , \]
	and let $\mathfrak m$ denote the maximal ideal of $\mathbb C\{T_1, \ldots, T_n\}_x$.
	Since $f$ is $\mathbb C\{T_1, \ldots, T_n\}_x$-linear, we see that $f\inv(\mathfrak m)$ is again a maximal ideal of $\mathbb C\{T_1, \ldots, T_n\}_x[X_1, \ldots, X_m]$ and that the map \eqref{eq:flatness_analytification_smooth_algebras_complex_case_I} induces a canonical map
	\begin{equation} \label{eq:flatness_analytification_smooth_algebras_complex_case_II}
		( \mathbb C\{ T_1, \ldots, T_n \}_x[X_1, \ldots, X_m] )_{f\inv(\mathfrak m)} \to \mathbb C\{T_1, \ldots, T_n, X_1, \ldots, X_m\}_y .
	\end{equation}
	Since the localization map 
	\[ \mathbb C\{T_1, \ldots, T_n \}_x[X_1, \ldots, X_m] \to ( \mathbb C\{ T_1, \ldots, T_n \}_x[X_1, \ldots, X_m] )_{f\inv(\mathfrak m)} \]
	is flat, it is enough to prove that also the map \eqref{eq:flatness_analytification_smooth_algebras_complex_case_II} is flat.
	Observe that both the source and the target of that map are noetherian local rings.
	In particular, it is enough to check that \eqref{eq:flatness_analytification_smooth_algebras_complex_case_II} becomes flat after passing at the formal completion at the maximal ideals.
	Since we can identify both formal completions with the ring of formal power series $\mathbb C\llb T_1, \ldots, T_n, X_1, \ldots, X_m \rrb$, the conclusion follows.
\end{proof}

The last step of the proof of \cref{prop:flatness} is provided by the following:

\begin{cor}
	Every $A\alg$-linear map
	\[ f \colon A\alg[X_1, \ldots, X_m] \to A\alg \]
	has the property ($P_A$).
\end{cor}

\begin{proof}
	Unraveling the definitions, we see that we have to prove that the square
	\[ \begin{tikzcd}
		A\alg[X_1, \ldots, X_m] \arrow{r}{f} \arrow{d}{\eta} & A\alg \arrow{d}{\id} \\
		A\alg[X_1, \ldots, X_m]\an_A \arrow{r} & A\alg
	\end{tikzcd} \]
	is a pushout in $\CRing_k$.
	Using \cref{prop:geometric_interpretation_relative_analytification}, $A\alg[X_1, \ldots, X_m]\an_A$ can be described as $y\inv \cO_{V \times \bA^m_k}$, where the notations are those introduced right before \cref{prop:geometric_interpretation_relative_analytification}.
	Therefore, the square above is a pushout in the category of (underived) rings.
	By \cref{cor:flatness_smooth_algebras}, the map $\eta$ is flat.
	We conclude that the square above is a pushout in $\CRing_k$.
\end{proof}

\section{Analytic cotangent complex}\label{sec:cotangent_complex}

In this section we introduce the analytic cotangent complex and we establish its basic properties.
In the first subsection, we work in the general framework of structured topoi for a given pregeometry.
The main tool we employ is Lurie's formalism of tangent category.
However, an adaptation is needed due to our framework of analytic modules in \cref{sec:analytic_modules}.
In Subsection \ref{subsec:analytic_cotangent_complex}, we specialize the general formalism to the setting of derived analytic geometry.
The remaining subsections concern various properties of the analytic cotangent complex.

\subsection{The cotangent complex formalism} \label{subsec:cotangent_formalism}

Let $\Cat_\infty$ denote the \infcat of \infcats.
Let $\Catlex$ denote the subcategory of $\Cat_\infty$ spanned by those $\infty$-categories having finite limits and by those functors that preserve them.
Let $\cTAb$ be the Lawvere theory of abelian groups (cf.\ \cref{def:abelian_group_objects}).
For $n \ge 0$, we denote by $A^n$ the free abelian group on $n$ elements seen as an element in $\cTAb$.

Using \cref{lem:abelianization_of_functors}, we see that the the assignment $\cC \mapsto \Ab(\cC)$ can be promoted to an $\infty$-functor
\[ \Ab(-) \coloneqq \Fun^\times(\cTAb, -) \colon \Catlex \to \Catlex . \]
We call this functor the \emph{abelianization functor}.

Let $\cC$ be an $\infty$-category with finite limits and consider the Cartesian fibration
\[ p \colon \Fun(\Delta^1, \cC) \to \Fun(\{1\}, \cC) \simeq \cC . \]
Observe that the associated $\infty$-functor $\cC\op \to \Catinf$ factors through $\Catlex$.
Let $\cC_\Ab$ be the full subcategory of $\Fun(\Delta^1 \times \cTAb, \cC)$ spanned by those functors
\[ F \colon \Delta^1 \times \cTAb \to \cC \]
satisfying the following conditions:
\begin{enumerate}
	\item the restriction $\left. F \right|_{\{0\} \times \cTAb}$ commutes with fiber products;
	\item the canonical map $F(0, A^0) \to F(1, A^0)$ is an equivalence;
	\item for every $A^n \in \cTAb$, the canonical map $F(1, A^n) \to F(1, A^0)$ is an equivalence.
\end{enumerate}
Let $e \colon \Delta^1 \to \Delta^1 \times \cTAb$ be the functor selecting the morphism
\[ (0, A^1) \to (0, A^0) . \]
Finally, we consider the composition
\[ q \colon \cC_\Ab \hookrightarrow \Fun(\Delta^1 \times \cTAb, \cC) \xrightarrow{e_*} \Fun(\Delta^1, \cC) \xrightarrow{p} \cC , \]
where $e_*$ is given by precomposition with $e$.

\begin{lem} \label{lem:abelianized_tangent_bundle}
	The functor $q \colon \cC_\Ab \to \cC$ is a Cartesian fibration.
	Furthermore:
	\begin{enumerate}
		\item a morphism $f$ in $\cC_\Ab$ is $q$-Cartesian if and only if $e_*(f)$ is $p$-Cartesian in $\Fun(\Delta^1, \cC)$;
		\item for any $x \in \cC$, the fiber $(\cC_\Ab)_x$ is equivalent to $\Ab(\cC_{/x})$;
		\item a diagram $g \colon K^\vartriangleleft \to \cC_\Ab$ is a (co)limit diagram if and only if $g$ is a $q$-(co)limit diagram and $q \circ g$ is a (co)limit diagram in $\cC$.
	\end{enumerate}
\end{lem}

\begin{proof}
	We first remark that if $\cD$ is an $\infty$-category with final object $*_\cD$ then evaluation at $*_\cD$ induces a Cartesian fibration
	\[ \Fun(\cD, \cC) \to \cC , \]
	and moreover a natural transformation $f \colon F \to G$ in $\Fun(\cD, \cC)$ is a Cartesian edge if and only if for every object $x \in \cD$, the square
	\[ \begin{tikzcd}
		F(x) \arrow{r}{f} \arrow{d} & G(x) \arrow{d} \\
		F(*_\cD) \arrow{r}{f} & G(*_\cD)
	\end{tikzcd} \]
	is a pullback square in $\cC$.
	It follows that evaluation at $(1, A^0) \in \Delta^1 \times \cTAb$ induces a Cartesian fibration
	\[ \Fun(\Delta^1 \times \cTAb, \cC) \to \cC , \]
	and that moreover
	\[ e_* \colon \Fun(\Delta^1 \times \cTAb, \cC) \to \Fun(\Delta^1, \cC) \]
	preserves Cartesian edges.
	
	Let now $G \in \cC_\Ab$ and suppose that $f \colon F \to G$ is a Cartesian edge in $\Fun(\Delta^1 \times \cTAb, \cC)$.
	We claim that $F \in \cC_\Ab$ as well.
	Indeed, observe that the morphism $(1, A^n) \to (1, A^0)$ induces a pullback square
	\[ \begin{tikzcd}
		F(1, A^n) \arrow{r}{f} \arrow{d} & G(1, A^n) \arrow{d} \\
		F(1, A^0) \arrow{r}{f} & G(1, A^0) .
	\end{tikzcd} \]
	Since the vertical morphism on the right is an equivalence, the same goes for the one on the left.
	The same reasoning applied to the morphism $(0, A^0)  \to (1, A^0)$ shows that
	\[ F(0, A^0) \to F(1, A^0) \]
	is an equivalence.
	We are left to prove that $F(0, A^{n+m}) \simeq F(0, A^n) \times F(0, A^m)$.
	Consider the diagram
	\[ \begin{tikzcd}
		F(0, A^{n+m}) \arrow{r} \arrow{d} & F(0, A^n) \arrow{d} \arrow{r} & F(0, A^0) \arrow{d} \\
		G(0,A^{n+m}) \arrow{r} & G(0, A^n) \arrow{r} & G(0, A^0) .
	\end{tikzcd} \]
	Since $f$ is a Cartesian edge, we see that the outer square and the one on the right are pullback.
	It follows that the same goes for the one on the left.
	Since $G(0, A^{n+m}) \simeq G(0, A^n) \times G(0, A^m)$, we conclude that $F(0, A^{n+m}) \simeq F(0, A^n) \times F(0, A^m)$ as well.
		
	Recall now that for objects $F \in \cC_\Ab$ the canonical morphism $F(0,A^0) \to F(1, A^0)$ is an equivalence.
	We therefore deduce that the functor $q \colon \cC_\Ab \to \cC$ is a Cartesian fibration and that the composition
	\[ \cC_\Ab \hookrightarrow \Fun( \Delta^1 \times \cTAb, \cC ) \xrightarrow{e_*} \Fun( \Delta^1, \cC ) \]
	preserves Cartesian edges.
	Let now $f \colon F \to G$ be a morphism in $\cC_\Ab$ and suppose that $e_*(f)$ is $p$-Cartesian.
	Since both $F$ and $G$ belong to $\cC_\Ab$, it is enough to check that the squares of the form
	\[ \begin{tikzcd}
		F(0, A^n) \arrow{d} \arrow{r}{f} & G(0, A^n) \arrow{d} \\
		F(1, A^0) \arrow{r}{f} & G(1, A^0)
	\end{tikzcd} \]
	are pullback diagrams.
	When $n = 0$, this is true because both vertical maps are equivalences, and when $n = 1$ this follows from the hypothesis that $e_*(f)$ is $p$-Cartesian.
	The general case follows by induction, using the fact that $F(0,A^{n+1}) \simeq F(0, A^n) \times F(0, A^1)$ and $G(0, A^{n+1}) \simeq G(0, A^n) \times G(0,A^1)$.
	This completes the proof of (1).
	
	We now turn to statement (2).
	Recall that
	\[ \Fun(\cTAb, \cC_{/x}) \simeq \Fun_x(\cTAb^\vartriangleright, \cC) . \]
	We can identify $\cTAb^\vartriangleright$ with the full subcategory of $\Delta^1 \times \cTAb$ spanned by $\{0\} \times \cTAb$ and the object $(1, A^0)$.
	Using \cite[4.3.2.15]{HTT} twice, we see that the restriction along $\cTAb^\vartriangleright \hookrightarrow \Delta^1 \times \cTAb$ induces an equivalence
	\[ \cC_\Ab \simeq  \Fun^\times(\cTAb^\vartriangleright, \cC) . \]
	Passing to the fiber at $x \in \cC$, we obtain the equivalence
	\[ (\cC_\Ab)_x \simeq \Ab(\cC_{/x}) \]
	we were looking for.
	
	As for statement (3), the same proof of \cite[7.3.1.12]{Lurie_Higher_algebra} applies.
\end{proof}

\begin{defin}
	Let $\cC$ be a presentable $\infty$-category.
	The \emph{abelianized tangent bundle} of $\cC$ is by definition the stabilization of the Cartesian fibration
	\[ q \colon \cC_\Ab \to \cC \]
	constructed above.
	It is denoted by $\mathrm{T}_{\Ab}(\cC)$.
\end{defin}

Using \cref{lem:abelianized_tangent_bundle}, we see that the abelianized tangent bundle to $\cC$ is a Cartesian fibration
\[ \pi \colon \rT_\Ab(\cC) \to \cC , \]
whose fiber at $x \in \cC$ is equivalent to $\Sp(\Ab(\cC_{/x}))$.

Now let us explain how to use the language of the abelianized tangent bundle to introduce the analytic cotangent complex.
We have:

\begin{lem} \label{lem:abelianized_tangent_bundle_II}
	Let $\cC$ be a presentable $\infty$-category.
	Then:
	\begin{enumerate}
		\item $\rT_\Ab(\cC)$ is a presentable $\infty$-category;
		\item The canonical map $q \colon \rT_\Ab(\cC) \to \cC$ commutes with limits and filtered colimits.
	\end{enumerate}
\end{lem}

\begin{proof}
	It follows from the proof of \cref{lem:abelianized_tangent_bundle} that $\cC_\Ab$ can be realized as an accessible localization of $\Fun(\cTAb \times \Delta^1, \cC)$.
	In particular, $\cC_\Ab$ is presentable.
	Moreover, \cref{lem:abelianized_tangent_bundle}(3) implies that the map
	\[ q \colon \cC_\Ab \to \cC \]
	preserves both limits and colimits.
	We are therefore reduced to prove the following statements.
	Let $p \colon \cX \to S$ be a presentable fibration which preserves limits and filtered colimits and where $\cX$ is presentable\footnote{This last condition is redundant. See \cite[Theorem 10.3]{Gepner_Lax_2015}.}. Then:
	\begin{enumerate}
		\item the $\infty$-category $\mathrm{Stab}(p)$ is presentable;
		\item the functor $\pi \colon \mathrm{Stab}(p)$ commutes with limits and filtered colimits;
		\item the functor $\pi \colon \mathrm{Stab}(p) \to S$ is a presentable fibration.
	\end{enumerate}
	Condition (3) follows from the definition of $\mathrm{Stab}(p)$ \cite[7.3.1.1, 7.3.1.7]{Lurie_Higher_algebra}.
	The first two statements follow from the fact that $\mathrm{Stab}(p)$ can be realized as an accessible localization of $\Fun(\cS_*^{\mathrm{fin}}, \cX)$.
	Indeed, let $\cE$ be the full subcategory of $\Fun(\cS_*^{\mathrm{fin}}, \cX)$ spanned by those functor
	\[ g \colon \cS_*^{\mathrm{fin}} \to \cX \]
	such that:
	\begin{enumerate}
		\item $g$ is excisive;
		\item if $s = p(g(*)) \in S$, then $g(*)$ is a final object for $\cX_s$;
		\item the composition $p \circ g \colon \cS_*^{\mathrm{fin}}$ factors through $S^\simeq$, the maximal $\infty$-groupoid contained inside $S$.
	\end{enumerate}
	Observe that the inclusion
	\[ \cE \hookrightarrow \Fun(\cS_*^{\mathrm{fin}}, \cX) \]
	commutes with limits and filtered colimits.
	It follows that $\cE$ is an accessible localization of the presentable $\infty$-category $\Fun(\cS_*^{\mathrm{fin}}, \cX)$ and that the projection
	\[ \cE \hookrightarrow \Fun(\cS_*^{\mathrm{fin}}, \cX) \to \cX \xrightarrow{p} S \]
	induced by evaluation at the object $S^0 \in \cS_*^{\mathrm{fin}}$ commutes with limits and filtered colimits.
	
	We are only left to identify $\cE$ with $\mathrm{Stab}(p)$.
	Reasoning as in the proof of \cref{lem:abelianized_tangent_bundle}, we see that the map $\cE \to \cX$ takes Cartesian edges to Cartesian edges.
	Furthermore, the fiber at $s \in S$ can be canonically identified with the full subcategory of
	\[ \mathrm{Exc}(\cS_*^{\mathrm{fin}}, \cX_s) \]
	spanned by those functor that take final objects to final objects.
	In other words, $\cE_s \simeq \Sp( \cX_s )$.
	This completes the proof.
\end{proof}

Now let $\cT$ be any pregeometry and let $X \coloneqq (\cX, \cO_X)$ be a $\cT$-structured topos.
Recall from \cite[Proposition 1.15]{Porta_DCAGI} that the $\infty$-category $\Strloc_{\cT}(\cX)_{/\cO_X}$ is presentable.
Let $\cA \in \Strloc_{\cT}(\cX)_{/\cO_X}$ be any $\cT$-structure equipped with a local morphism to $\cO_X$.
Then the $\infty$-category
\[ \cT_{X,\cA} \coloneqq \Strloc_\cT(\cX)_{\cA // \cO_X} \]
is again presentable.
As a consequence, we can apply the above results to see that
\[ \pi \colon \rT_\Ab(\cT_{X, \cA}) \to \cT_{X, \cA} \]
is a functor between presentable categories that preserves limits and colimits.
It fits in a commutative triangle
\[ \begin{tikzcd}
	\rT_\Ab(\cT_{X, \cA}) \arrow{rr}{G} \arrow{dr}{\pi} & & \Fun(\Delta^1, \cT_{X, \cA}) \arrow{dl} \\
	{} & \cT_{X, \cA} ,
\end{tikzcd} \]
where $G$ is the natural functor.
Observe that the fiber of G at an object $\cB \in \cT_{X, \cA}$ can be identified with the following composition:
\[ \Sp(\Ab(\Strloc_{\cT}(\cX)_{\cA // \cB})) \xrightarrow{\Omega^\infty} \Ab(\Strloc_{\cT}(\cX)_{\cA // \cB}) \xrightarrow{U} \Strloc_{\cT}(\cX)_{\cA // \cB} , \]
where $U$ denotes the forgetful functor.
Let us denote by $\Omega^\infty_\Ab$ the composition $U\circ\Omega^\infty$.
In particular, it admits a left adjoint, that we denote $\Sigma^\infty_\Ab$.
We can therefore combine \cref{lem:abelianized_tangent_bundle} and \cite[7.2.3.11]{Lurie_Higher_algebra} to conclude that $G$ admits a left adjoint relative to $\cT_{X, \cA}$ (in the sense of \cite[7.3.2.2]{Lurie_Higher_algebra}).
We denote this left adjoint by $F$.
Finally, we let
\[ s \colon \cT_{X, \cA} \to \Fun(\Delta^1, \cT_{X, \cA}) \]
the functor defined informally by sending $\cA \xrightarrow{f} \cB \xrightarrow{g} \cX$ to the diagram
\[ \begin{tikzcd}
	\cA \arrow{r}{\mathrm{id}_\cA} \arrow{d}{\mathrm{id}_\cA} & \cA \arrow{r} \arrow{d}{f} & \cO_X \arrow{d}{\mathrm{id}_{\cO_X}} \\
	\cA \arrow{r}{f} & \cB \arrow{r}{g} & \cO_X .
\end{tikzcd} \]
Notice that the existence of the functor $s$ is a direct application of \cite[4.3.2.15]{HTT}.

\begin{defin} \label{def:relative_cotangent_complex_algebras}
	Let $X \coloneqq (\cX, \cO_X)$ be a $\cT$-structured topos and let $\cA \in \Strloc_{\cT}(\cX)_{/\cO_X}$.
	The \emph{$\cT$-theoretic cotangent complex functor relative to $X$ and $\cA$} is the composition
	\[ \mathbb L^{\cT}_{X, \cA} \colon \cT_{X, \cA} \xrightarrow{s} \Fun(\Delta^1, \cT_{X, \cA}) \xrightarrow{F} \rT_\Ab(\cT_{X, \cA}) . \]
	Let $\cB \in \Strloc_{\cT}(\cX)_{/\cO_X}$ and let $\varphi \colon \cA \to \cB$ be a morphism.
	The \emph{relative $\cT$-theoretic cotangent complex of $\varphi$}, denoted by $\mathbb L^\cT_{\varphi}$, or by $\mathbb L^\cT_{\cB / \cA}$ when the morphism is clear from the context, is the object
	\[ \mathbb L^{\cT}_{X, \cA}( \cB ) \in \Sp( \Ab (\Strloc_{\cT}(\cX)_{\cA // \cB}) ) . \]
	When $\cA$ is an initial object of $\Strloc_\cT(\cX)_{/\cO_X}$ we refer to $\mathbb L^{\cT}_{X, \cA}$ as the \emph{absolute cotangent complex} and we omit $\cA$ from the above notations.
\end{defin}

Let $\cT$ be a pregeometry, $\cX$ an $\infty$-topos and $\cO$ a $\cT$-structure on $\cX$.
Since $\Strloc_{\cT}(\cX)_{/\cO}$ is presentable, it admits pushouts.
We denote by $\cB_1 \otimes_{\cA}^{\cT} \cB_2$ the pushout of the diagram
\[ \cB_1 \leftarrow \cA \to \cB_2 \]
in $\Strloc_{\cT}(\cX)_{/\cO}$. 
Furthermore, we can rewrite the $\cT$-theoretic cotangent complex of $\varphi \colon \cA \to \cB$ in $\Strloc_{\cT}(\cX)$ as
\[ \mathbb L^\cT_{\cB / \cA} \simeq \Sigma^\infty_\Ab(\cB \otimes_{\cA}^\cT \cB) . \]

\begin{defin} \label{def:relative_cotangent_complex_spaces}
	Let $\cT$ be a pregeometry and let $X = (\cX, \cO_X)$ and $Y = (\cY, \cO_Y)$ be $\cT$-structured topoi.
	Let $f = (f_*, f^\sharp) \colon (\cX, \cO_X) \to (\cY, \cO_Y)$ be a morphism in $\RTop(\cT)$.
	The \emph{relative $\cT$-theoretic cotangent complex of $f$}, denoted by $\mathbb L_f^{\cT}$, is defined to be the relative $\cT$-theoretic cotangent complex of $f^\sharp \colon f\inv \cO_Y \to \cO_X$ in the sense of \cref{def:relative_cotangent_complex_algebras}.
	We will denote $\mathbb L_f^{\cT}$ by $\mathbb L^{\cT}_{X / Y}$ when the morphism $f$ is clear from the context.
\end{defin}

We now deduce some basic properties of the cotangent complex using the formal properties in \cite[§7.3.3]{Lurie_Higher_algebra}.
We start by fixing some notations.

Let $X \coloneqq (\cX, \cO_X)$, $Y \coloneqq (\cY, \cO_Y)$ be $\cT$-structured topoi and let $f \colon X \to Y$ be a morphism between them.
We denote the underlying geometric morphism of $\infty$-topoi by
\[ f_* \colon \cX \leftrightarrows \cY \colon f\inv, \]
and the underlying local morphism of $\cT$-structures by
\[ f^\sharp \colon f\inv \cO_Y \to \cO_X . \]

Since the functor $f\inv$ commutes with finite limits, composition with it induces a well-defined functor
\begin{equation} \label{eq:topological_pullback}
	\Strloc_\cT(\cY)_{/\cO_Y} \to \Strloc_\cT(\cX)_{/f\inv \cO_Y} .
\end{equation}
Observe that this functor commutes again with limits and sifted colimits.
In particular, it induces a functor
\[ \Sp( \Ab ( \Strloc_\cT(\cY)_{/\cO_Y} ) ) \to \Sp( \Ab( \Strloc_\cT(\cX)_{/f\inv \cO_Y} ) ) , \]
which we still denote by $f\inv$.

On the other hand, composition with $f^\sharp$ induces a functor
\begin{equation} \label{eq:tensor_product}
	f^\sharp_! \colon \Strloc_\cT(\cX)_{/f\inv \cO_Y} \to \Strloc_{\cT}(\cX)_{/ \cO_X} .
\end{equation}
Although this functor does not commute with finite limits, pullback along $f^\sharp \colon f\inv \cO_Y \to \cO_X$ provides a right adjoint to $f^\sharp_!$, that we denote by $f^\sharp_*$.
Notice that $f^\sharp_*$ commutes with filtered colimits.
Composition with $f^\sharp_*$ induces a functor
\[ \Sp( \Ab( \Strloc_\cT(\cX)_{/ \cO_X} ) ) \to \Sp ( \Ab ( \Strloc_{\cT}(\cX)_{/f\inv \cO_Y}) ) \]
that commutes with limits and filtered colimits.
The adjoint functor theorem guarantees then the existence of a left adjoint, which we denote by
\[ f^{\sharp *} \colon \Sp ( \Ab ( \Strloc_{\cT}(\cX)_{/f\inv \cO_Y}) ) \to \Sp ( \Ab ( \Strloc_{\cT}(\cX)_{/\cO_X}) ) \]

Finally, composing $f^{\sharp*}$ and $f\inv$ provides a functor
\[ f^* \colon \Sp(\Ab (\Strloc_{\cT}(\cY)_{/\cO_Y}) ) \xrightarrow{f\inv} \Sp(\Ab(\Strloc_{\cT}(\cX)_{/f\inv \cO_Y})) \xrightarrow{f^{\sharp*}} \Sp ( \Ab ( \Strloc_{\cT}(\cX)_{/\cO_X}) ) . \]

\begin{lem} \label{lem:transitivity_cotangent_complex}
	Let $f \colon X \to Y$ be a morphism of $\cT$-structured topoi. Then the diagram
	\[ \begin{tikzcd}
	\Sp(\Ab(\Strloc_{\cT}(\cY)_{/ \cO_Y})) \arrow{rr}{f\inv} & & \Sp(\Ab(\Strloc_{\cT}(\cX)_{/ f\inv \cO_Y})) \\
	\Strloc_{\cT}(\cY)_{/ \cO_Y} \arrow{rr}{f\inv} \arrow{u}{\Sigma^\infty_\Ab} & & \Strloc_{\cT}(\cX)_{/ f\inv \cO_Y} \arrow{u}{\Sigma^\infty_\Ab}
	\end{tikzcd} \]
	commutes. In particular, $f\inv( \anL_Y) \simeq \anL_{f\inv \cO_Y}$.
\end{lem}

\begin{proof}
		Introduce the $\infty$-category $\Str_\cT'(\cX)$ whose objects are functors
	\[ \cO \colon \cT \to \cX \]
	that commute with products and admissible pullbacks, and whose morphisms are natural transformations $\varphi \colon \cO \to \cO'$ such that for every admissible morphism $j \colon U \to V$ in $\cT$ the square
	\[ \begin{tikzcd}
	\cO(U) \arrow{r} \arrow{d} & \cO(U) \arrow{d} \\
	\cO(V) \arrow{r} & \cO'(V)
	\end{tikzcd} \]
	is a pullback square.
	Then the natural functor $\Strloc_{\cT}(\cX)_{/\cO_X} \to \Str_\cT'(\cX)_{/\cO_X}$ is fully faithful.
	Let $\cO \in \Str_\cT'(\cX)$ and let $\varphi \colon \cO \to \cO_X$ be a morphism.
	Let $\{U_i \to U\}$ be an admissible cover in $\cT$.
	Then the diagram
	\[ \begin{tikzcd}
	\coprod \cO(U_i) \arrow{r} \arrow{d} & \cO(U) \arrow{d} \\
	\coprod \cO_X(U_i) \arrow{r} & \cO_X(U)
	\end{tikzcd} \]
	is a pullback. Since the bottom horizontal morphism is an effective epimorphism, the same goes for the top horizontal one.
	In other words, $\cO \in \Strloc_\cT(\cX)_{/\cO_X}$.
	This shows that there is an canonical equivalence
	\begin{equation} \label{eq:non_local_structures}
	\Str_\cT'(\cX)_{/\cO_X} \simeq \Strloc_{\cT}(\cX)_{/\cO_X} .
	\end{equation}
	
	We can now argue as follows.
	Composition with
	\[ f_* \colon \cX \to \cY \]
	induces a well-defined functor
	\[ f_* \colon \Str_\cT'(\cX)_{/f\inv \cO_Y} \to \Str_\cT'(\cY)_{/f_* f\inv \cO_Y} . \]
	Moreover, pullback along the natural transformation $\cO_Y \to f_* f\inv \cO_Y$, we obtain a functor
		\[ \Str_\cT'(\cY)_{/f_* f\inv \cO_Y} \to \Str_\cT'(\cY)_{/ \cO_Y} . \]
	Composing these two functors and using the equivalence \eqref{eq:non_local_structures} we obtain a functor
	\[ f_* \colon \Strloc_\cT(\cX)_{/f\inv \cO_Y} \simeq \Str_\cT'(\cX)_{/ f\inv \cO_Y} \to \Str_\cT'(\cY)_{/ \cO_Y} \simeq \Strloc_{\cT}(\cY)_{/\cO_Y} . \]
	This functor is the right adjoint for the functor
	\[ f\inv \colon \Strloc_\cT(\cY)_{/ \cO_Y} \longrightarrow \Strloc_\cT(\cX)_{/f\inv \cO_Y} . \]
		It follows that composition with $f_*$ induces a functor
	\[ f_* \colon \Sp(\Ab(\Str_\cT'(\cY)_{/\cO_Y})) \to \Sp(\Ab(\Str_{\cT}'(\cY)_{/\cO_Y})) \]
	that is right adjoint to the functor $f\inv$ constructed above.
	It is now enough to check that the diagram of right adjoints
	\[ \begin{tikzcd}
	\Sp\big(\Ab\big(\Str_\cT'(\cY)_{/\cO_Y}\big)\big) \arrow{d}{\Omega^\infty_\Ab} & \Sp\big(\Ab\big(\Str_{\cT}'(\cX)_{/f\inv \cO_Y}\big)\big) \arrow{l}[swap]{f_*} \arrow{d}{\Omega^\infty_\Ab} \\
	\Str_\cT'(\cY)_{/\cO_Y} & \Str_\cT'(\cX)_{/f\inv \cO_Y} \arrow{l}[swap]{f_*}
	\end{tikzcd} \]
	commutes.
	This follows because, given $F \in \Sp(\Ab( \Str_\cT'(\cX)_{/f\inv \cO_Y} ))$, we have natural identifications
	\[ f_*( \Omega^\infty_\Ab(F) ) \simeq f_* \circ F(S^0, A^1) \simeq (f_* \circ F)(S^0, A^1) \simeq \Omega^\infty_\Ab( f_* F) . \]
		The proof is therefore complete.
\end{proof}

\begin{prop} \label{prop:transitivity_cotangent_complex}
	Let $\cT$ be a pregeometry and let $f \colon X \to Y$ and $g \colon Y \to Z$ be morphisms of $\cT$-structured topoi, where $X = (\cX, \cO_X)$.
	Then there is a fiber sequence
	\[ f^* \mathbb L^\cT_{Y/Z} \to \mathbb L^\cT_{X / Z} \to \mathbb L^\cT_{X / Y} \]
	in $\Sp(\Ab(\Strloc_{\cT}(\cX)_{/\cO_X}))$.
\end{prop}

\begin{proof}
	Using Lemmas \ref{lem:abelianized_tangent_bundle_II} and \ref{lem:transitivity_cotangent_complex}, the same proof of \cite[7.3.3.6]{Lurie_Higher_algebra} applies.
\end{proof}

\begin{cor} \label{cor:cotangent_complex_etale}
	Let $\cT$ be a pregeometry.
	If $f \colon X \to Y$ is an \'etale morphism of $\cT$-structured topoi (cf.\ \cite[Definition 2.3.1]{DAG-V}), then $\mathbb L_{Y/X}^\cT \simeq 0$.
\end{cor}

\begin{proof}
	It follows from the transitivity sequence of \cref{prop:transitivity_cotangent_complex} by taking $Z$ to be a point and localising on $X$.
\end{proof}

\begin{prop} \label{prop:base_change_cotangent_complex}
	Suppose given a pullback diagram
	\[ \begin{tikzcd}
	X' \arrow{r} \arrow{d}{g} & Y' \arrow{d}{f} \\
	X \arrow{r} & Y
	\end{tikzcd} \]
	in the category $\RTop(\cT)$.
	Then the natural morphism
	\[ g^*(\mathbb L^\cT_{X / Y}) \to \mathbb L^\cT_{X' / Y'} \]
	is an equivalence.
\end{prop}

\begin{proof}
	Using Lemmas \ref{lem:abelianized_tangent_bundle_II} and \ref{lem:transitivity_cotangent_complex}, the same proof of \cite[7.3.3.7]{Lurie_Higher_algebra} applies.
\end{proof}

\begin{rem} \label{rem:warning_cotangent_complex_pullback}
	The above proposition works for any pregeometry $\cT$.
	Nevertheless, we are seldom interested in working with the full $\infty$-category $\RTop(\cT)$.
	For example, when $\cT = \cTetk$ is the \'etale pregeometry, we are only interested in working with the full subcategory of $\RTop(\cTetk)$ spanned by derived algebraic \DM stacks.
	Similarly, when $\cT = \cTank$, we are interested in working with the full subcategory of $\RTop(\cTank)$ spanned by derived analytic spaces.
	In general, the inclusion of these full subcategories does not commute with pullbacks.
	In other words, the \cref{prop:base_change_cotangent_complex} has to be proven again in the cases of interests.
	
	The complex analytic case is an exception.
	Indeed, \cite[Proposition 12.12]{DAG-IX} guarantees that the inclusion $\dAn_{\mathbb C} \hookrightarrow \RTop(\cTanc)$ commutes with pullbacks.
	The case of derived algebraic \DM stacks can also be dealt with easily: the question being local, one can reduce to the affine case, where the result follows directly from \cite[7.3.3.7]{Lurie_Higher_algebra}.
	However, the non-archimedean analytic case is trickier and requires techniques that will be introduced in the next subsection.
	We refer to \cref{prop:base_change_cotangent_complex_analytic} for the proof.
\end{rem}

\subsection{The analytic cotangent complex} \label{subsec:analytic_cotangent_complex}

From this point on, we will specialize to the pregeometry $\cTank$.
If $f \colon X \to Y$ is a morphism in $\RTop(\cTank)$, we write $\anL_{X/Y}$ instead of $\mathbb L^{\cTank}_{X / Y}$.
It is an element in $\Sp(\Ab( \AnRing_k(\cX)_{/\cO_X} ))$.
Nonetheless, using the equivalence
\[ \Sp(\Ab(\AnRing_k(\cX)_{/\cO_X})) \simeq \cO_X\Mod . \]
provided by \cref{thm:equivalence_of_modules}, we consider $\anL_{X/Y}$ as an element in $\cO_X\Mod$.
Since this stable $\infty$-category has a canonical t-structure (cf.\ \cite[1.7]{DAG-VII}), we have the cohomology sheaves $\pi_i(\anL_{X/Y})$.

As in the algebraic setting, the analytic cotangent complex is closely related to analytic derivations.

\begin{defin} \label{def:analytic_split_square_zero_extension}
	Let $X = (\cX, \cO_X)$ be a derived analytic space and let $\cF \in \cO_X\Mod^{\ge 0}$.
	The \emph{analytic split square-zero extension of $\cO_X$ by $\cF$} is the derived analytic ring
	\[ \cO_X \oplus \cF \coloneqq \Omega^\infty_\Ab(\cF) \in \AnRing_k(\cX)_{/\cO_X} . \]
\end{defin}

This definition is motivated by \cite[7.3.4.15]{Lurie_Higher_algebra}.
Let us show that the notion of analytic split square-zero extension is compatible with the underlying algebra:

\begin{lem} \label{lem:underlying_algebra_split_square_zero_extension}
	Let $\varphi \colon \cT' \to \cT$ be a transformation of pregeometries and let $X = (\cX, \cO_X)$ be a $\cT$-structured topos.
	Then the functor
	\[ \varphi_* \colon \Strloc_{\cT}(\cX)_{/\cO_X} \to \Strloc_{\cT'}(\cX)_{/\cO_X\circ \varphi} \]
	given by precomposition with $\varphi$ induces a commutative square
	\[ \begin{tikzcd}
		\Sp(\Ab( \Strloc_{\cT'}(\cX)_{/\cO_X \circ \varphi} )) \arrow{d}{\Omega^\infty_\Ab} & \Sp(\Ab( \Strloc_{\cT}(\cX)_{/\cO_X} )) \arrow{l}[swap]{\varphi_*} \arrow{d}{\Omega^\infty_\Ab} \\
		\Strloc_{\cT'}(\cX)_{/\cO_X \circ \varphi} & \Strloc_{\cT}(\cX)_{/\cO_X} \arrow{l}[swap]{\varphi_*} .
	\end{tikzcd} \]
\end{lem}

\begin{proof}
	Since $\varphi_* \colon \Strloc_\cT(\cX)_{/\cO_X} \to \Strloc_{\cT'}(\cX)_{/\cO_X}$ commutes with limits, composition with $\varphi_*$ induces a well-defined functor
	\[ \varphi_* \colon \Sp(\Ab(\Strloc_{\cT}(\cX)_{/\cO_X})) \to \Sp(\Ab(\Strloc_{\cT'}(\cX)_{/\cO_X})) . \]
	Let $\cF \colon \cS^{\mathrm{fin}}_* \times \cTAb \to \Strloc_{\cT}(\cX)_{/\cO_X}$ be an element in
	\[ \Sp(\Ab(\Strloc_{\cT}(\cX)_{/\cO_X})) . \]
	Then
	\[ \Omega^\infty_\Ab( \varphi_*( F ) ) \simeq (\varphi_*(F))(S^0, A^1) \simeq F(S^0,A^1) \circ \varphi \simeq \varphi_*( \Omega^\infty_\Ab(F) ) . \]
	The proof is thus complete.
\end{proof}

\begin{cor} \label{cor:split_square-zero_extension}
	Let $X = (\cX, \cO_X)$ be a derived analytic space and let $\cF \in \cO_X\Mod^{\ge 0}$.
	Then $( \cO_X \oplus \cF )\alg$ is the split square-zero extension of $\cO_X\alg$ by $\cF$.
\end{cor}

\begin{proof}
	Applying \cref{lem:underlying_algebra_split_square_zero_extension} to the transformation of pregeometries
	\[ (-)\an \colon \cTetk \to \cTank , \]
	the conclusion follows directly.
\end{proof}

\begin{defin} \label{def:analytic_derivation}
	Let $X = (\cX, \cO_X)$ be a derived analytic space and $\cA \in \allowbreak \AnRing_k(\cX)_{/\cO_X}$.
	Let $\cF \in \cO_X\Mod^{\ge 0}$.
	The space of \emph{$\cA$-linear analytic derivations of $\cO_X$ into $\cF$} is the space
	\[ \DerAn_{\cA}(\cO_X, \cF) \coloneqq \Map_{\AnRing_k(\cX)_{\cA // \cO_X}}( \cO_X, \cO_X \oplus \cF ). \]
\end{defin}

With this definition, we have the following characterization of the analytic cotangent complex:

\begin{prop} \label{prop:UMP_analytic_cotangent_complex}
	Let $X = (\cX, \cO_X)$ be a derived analytic space and $\cA \in \AnRing_k(\cX)_{/\cO_X}$.
	Then for any $\cF \in \cO_X\Mod^{\ge 0}$ there is a canonical equivalence
	\[ \Map_{\cO_X\Mod}(\anL_{\cO_X / \cA}, \cF) \simeq \DerAn_{\cA}(\cO_X, \cF) . \]
\end{prop}

\begin{proof}
	We have:
	\begin{align*}
		\DerAn_{\cA}(\cO_X, \cF) & \simeq \Map_{\AnRing_k(\cX)_{\cA // \cO_X}}( \cO_X, \cO_X \oplus \cF ) \\
		& = \Map_{\AnRing_k(\cX)_{\cA // \cO_X}}( \cO_X, \Omega^\infty_\Ab(\cF) ) \\
		& \simeq \Map_{\cO_X\Mod}( \Sigma^\infty_\Ab( \cO_X {\cotimes}_\cA \cO_X ), \cF ) \\
		& \simeq \Map_{\cO_X\Mod}( \anL_{\cO_X / \cA}, \cF ) . 
	\end{align*}
	The proof is therefore complete.
\end{proof}

To conclude this section, we discuss the behavior of the equivalence
\[ \Sp(\Ab(\AnRing_k(\cX)_{/\cO_X})) \simeq \cO_X\alg \Mod \]
under pullback along morphism of derived analytic spaces.

\begin{prop} \label{prop:pullback_Goodwillie_analytic}
	Let $f \colon X \to Y$ be a morphism of derived analytic spaces.
	Let $\cA \to f\inv \cO_Y$ be a morphism in $\AnRing_k(\cX)_{/\cO_X}$.
		Denote by
	\[ (-)\an \colon \CRing_k( \cX )_{\cA\alg / / \cO_X\alg} \to \AnRing_k(\cX)_{\cA // \cO_X} \]
	the left adjoint to the underlying algebra functor.
	Then:
	\begin{enumerate}
		\item The diagram
		\[ \begin{tikzcd}
		\AnRing_k(\cY)_{\cA // \cO_Y} \arrow{r}{f\inv} & \AnRing_k(\cX)_{\cA // f\inv \cO_Y} \arrow{r}{f^\sharp_!} & \AnRing_k(\cX)_{\cA // \cO_X} \\
		\CRing_k(\cY)_{\cA\alg // \cO_Y\alg} \arrow{u}{(-)\an} \arrow{r}{f\inv} & \CRing_k(\cX)_{\cA\alg // f\inv \cO_Y\alg} \arrow{r}{f^\sharp_!} \arrow{u}{(-)\an} & \CRing_k(\cX)_{\cA\alg // \cO_X\alg} \arrow{u}{(-)\an}
		\end{tikzcd} \]
		commutes.
		\item The diagram
		\[ \begin{tikzcd}
		\Sp\left(\Ab\left(\CRing_k(\cX)_{/f\inv \cO_Y\alg}\right)\right) \arrow{rr}{f^{\sharp*}} \arrow{d}{\simeq} & & \Sp\left(\Ab\left(\CRing_k(\cX)_{/\cO_X\alg}\right)\right) \arrow{d}{\simeq} \\
		f\inv \cO_Y\alg \Mod \arrow{rr}{- \otimes_{f\inv \cO_Y} \cO_X} & & \cO_X\alg \Mod
		\end{tikzcd} \]
		commutes.
		\item There is a natural equivalence $f^* \anL_{Y} \simeq \anL_{f\inv \cO_Y} \otimes_{f\inv \cO_Y} \cO_X$.
	\end{enumerate}
\end{prop}

\begin{proof}
	The first two statements follow from the commutativity of the corresponding diagrams of right adjoints.
	The last statement is a formal consequence of the previous ones and \cref{lem:transitivity_cotangent_complex}.
\end{proof}

\subsection{Cotangent complex and analytification} \label{subsec:cotangent_complex_analytification}

The goal of this subsection is to show that cotangent complex is compatible with analytification.
This result allows us to compute the first examples of analytic cotangent complexes (cf.\ \cref{cor:cotangent_complex_An}).
Finally, we will use these computations in order to prove the base change property of analytic cotangent complex in the non-archimedean setting (cf.\ \cref{prop:base_change_cotangent_complex_analytic}).

Let $X = (\cX, \cO_X)$ be a derived algebraic \DM stack locally almost of finite presentation over $k$.
Recall from \cref{sec:analytification} that the analytification functor
\[ (-)\an \colon \RTop(\cTetk) \longrightarrow \RHTop(\cTank) \]
is right adjoint to the algebraization functor
\[ \RHTop(\cTank) \longrightarrow \RTop(\cTetk) . \]
The counit of the adjunction produces a canonical map
\[ p \colon (\cX\an, \cO_{X\an}\alg) \to (\cX, \cO_X) .  \]

\begin{defin} \label{defin:analytification_functor_modules}
	We refer to the induced functor
	\[ p^* \colon \cO_X\Mod \longrightarrow \cO_{X\an} \Mod \]
	as the \emph{analytification functor}, and we denote it by $(-)\an$.
\end{defin}

\begin{thm} \label{thm:analytification_cotangent_complex}
	Let $X = (\cX, \cO_X)$ be a derived algebraic \DM stack locally almost of finite presentation over $k$.
	There is a canonical morphism
	\[ \varphi \colon \mathbb L_{X\an}\an \to (\mathbb L_X)\an \]
	in $\cO_{X\an}\Mod$.
	Moreover, $\varphi$ is an equivalence.
\end{thm}

\begin{proof}
	Applying \cref{lem:transitivity_cotangent_complex} with $\cT = \cTetk$ to the morphism $p \colon (\cX\an, \cO_{X\an}\alg) \to (\cX, \cO_X)$, we see that
	\[ p\inv \mathbb L_X \simeq \mathbb L_{p\inv \cO_X} , \]
	where we wrote $\mathbb L_{p \inv \cO_X}$ instead of $\mathbb L_{p\inv \cO_X}^{\cTetk}$.
	On the other hand, pulling back along the morphism $p^\sharp \colon p\inv \cO_X \to \cO_{X\an}\alg$ induces the following commutative diagram
	\begin{equation} \label{eq:analytification_cotangent_complex_I}
		\begin{tikzcd}
			\Sp\big(\Ab\big(\CRing_k(\cX\an)_{/p\inv \cO_X}\big)\big) \arrow{d}{\Omega^\infty_\Ab} & \Sp\big(\Ab\big(\CRing_k(\cX\an)_{/\cO_{X\an}\alg}\big)\big) \arrow{d}{\Omega^\infty_\Ab} \arrow{l}[swap]{p^\sharp_*} \\
			\CRing_k(\cX\an)_{/p\inv \cO_X} & \CRing_k(\cX\an)_{/\cO_{X\an}\alg} \arrow{l}[swap]{p^\sharp_*} .
		\end{tikzcd}
	\end{equation}
	Passing to the left adjoints and applying \cref{prop:pullback_Goodwillie_analytic}(2), we obtain
	\[ (\mathbb L_X)\an \simeq \mathbb L_{p\inv \cO_X} \otimes_{p\inv \cO_X} \cO_{X\an}\alg \simeq p^{\sharp *}( \mathbb L_{p\inv \cO_X} ) . \]

	Now we apply \cref{lem:underlying_algebra_split_square_zero_extension} to the canonical transformation of pregeometries
	\[ (-)\an \colon \cTetk \to \cTank \]
	to deduce that the square
	\begin{equation} \label{eq:analytification_cotangent_complex_II}
		\begin{tikzcd}
			\Sp\big(\Ab\big( \CRing_k(\cX\an)_{/\cO_{X\an}\alg} \big)\big) \arrow{d}{\Omega^\infty_\Ab} & \Sp\big(\Ab\big(\AnRing_k(\cX\an)_{/\cO_X}\big)\big) \arrow{l}[swap]{\sim} \arrow{d}{\Omega^\infty_\Ab} \\
			\CRing_k(\cX\an)_{/\cO_{X\an}\alg} & \AnRing_k(\cX\an)_{/\cO_X\an} \arrow{l}[swap]{(-)\alg}
		\end{tikzcd}
	\end{equation}
	commutes, where the top morphism is an equivalence in virtue of \cref{thm:equivalence_of_modules_precise}.
	
	Combining diagrams \eqref{eq:analytification_cotangent_complex_I} and \eqref{eq:analytification_cotangent_complex_II}, we obtain the commutativity of the following diagram:
	\[ \begin{tikzcd}
		\Sp \big( \Ab\big( \CRing_k(\cX\an)_{/p\inv \cO_X} \big)\big) \arrow{d}{\Omega^\infty_\Ab} & \Sp\big(\Ab\big( \AnRing_k(\cX\an)_{/\cO_{X\an}} \big)\big) \arrow{l} \arrow{d}{\Omega^\infty_\Ab} \\
		\CRing_k(\cX\an)_{/p\inv \cO_X} & \AnRing_k(\cX\an)_{/\cO_{X\an}} , \arrow{l}[swap]{\Phi}
	\end{tikzcd} \]
	where $\Phi$ is the composition $p^\sharp_* \circ (-)\alg$.
	Since both $(-)\alg$ and $p^\sharp_*$ are right adjoint, $\Phi$ has a left adjoint, that we denote
	\[ \Psi \colon \CRing_k(\cX\an)_{/p\inv \cO_X} \to \AnRing_k(\cX\an)_{/\cO_{X\an}} . \]
	
	To complete the proof, it is then enough to prove that $\Psi(p\inv \cO_X) \simeq \cO_{X\an}$.
	Let us denote by
	\[ (-)\an \colon \CRing_k(\cX\an)_{/\cO_{X\an}\alg} \to \AnRing_k(\cX\an)_{/\cO_{X\an}} \]
	the left adjoint to $(-)\alg$.
	Unraveling the definitions, we see that it is enough to prove that
	\[ ( p\inv \cO_X )\an \simeq \cO_{X\an} . \]
	This amounts to prove that for every $\cO \in \AnRing_k(\cX\an)_{/\cO_{X\an}}$, the map $p^\sharp \colon p\inv \cO_X \to \cO_{X\an}\alg$ induces an equivalence
	\begin{equation} \label{eq:analytification_cotangent_complex_III}
		\Map_{/ \cO_{X\an}}( \cO_{X\an}, \cO ) \simeq \Map_{/\cO_{X\an}\alg}( p\inv \cO_X, \cO\alg ) .
	\end{equation}
	Consider the commutative square
	\[ \begin{tikzcd}
		\Map_{X\an /}( (\cX\an, \cO), (\cX\an, \cO_{X\an}) ) \arrow{r} \arrow{d}{\alpha} & \Map_{\RTop}(\cX\an, \cX\an) \arrow{d}{\id} \\
		\Map_{(\cX\an, \cO_{X\an}\alg)/}( (\cX\an, \cO\alg), (\cX\an, p\inv \cO_X) ) \arrow{r} \arrow{d}{\beta} & \Map_{\RTop}(\cX\an, \cX\an) \arrow{d} \\
		\Map_{(\cX\an, \cO_{X\an}\alg)/} ( (\cX\an, \cO\alg), (\cX, \cO_X) ) \arrow{r} & \Map_{\RTop}( \cX\an, \cX ) .
	\end{tikzcd} \]
	The fiber of the top (resp.\ middle) horizontal morphism at the identity of $\cX\an$ is canonically equivalent to the left (resp.\ right) hand side of \eqref{eq:analytification_cotangent_complex_III}.
	It is therefore enough to prove that the map $\alpha$ becomes an equivalence after passing to the fiber at $p_* \colon \cX\an \to \cX$.
	The conclusion now follows from the following two observations: on one side, the composition $\beta \circ \alpha$ is an equivalence in virtue of the universal property of the analytification; on the other side, $\beta$ becomes an equivalence after passing to the fiber at $p_*$.
	Thus, the proof is complete.
\end{proof}

\begin{cor} \label{cor:analytification_relative_cotangent_complex}
	Let $f \colon X \to Y$ be a morphism of derived algebraic \DM stacks locally almost of finite presentation over $k$.
	Then there is a canonical morphism $\varphi \colon \anL_{X\an / Y\an} \to \mathbb (\mathbb L_{X / Y})\an$ and moreover $\varphi$ is an equivalence.
\end{cor}

\begin{proof}
	Both statements follow at once combining \cref{thm:analytification_cotangent_complex}, \cref{prop:transitivity_cotangent_complex} and \cref{prop:pullback_Goodwillie_analytic}.
\end{proof}

\begin{cor} \label{cor:cotangent_complex_An}
	The analytic cotangent complex of $\mathbf A^n_k$ is free of rank $n$.
	In particular, it is perfect and in tor-amplitude $0$.
\end{cor}

\begin{proof}
	Since $\mathbf A^n_k \simeq (\mathbb A^n_k)\an$, the statement is an immediate consequence of \cref{thm:analytification_cotangent_complex}.
\end{proof}

\begin{prop} \label{prop:base_change_cotangent_complex_analytic}
	For any pullback square
	\begin{equation} \label{eq:pullback_contangent_complex}
	\begin{tikzcd}
	X' \arrow{r} \arrow{d}{g} & Y' \arrow{d}{f} \\
	X \arrow{r}{u} & Y
	\end{tikzcd}
	\end{equation}
	in $\dAnk$, we have a canonical equivalence
	\[ g^* \anL_{X' / Y'} \xrightarrow{\sim} \anL_{X / Y} . \]
\end{prop}

\begin{proof}
	In the complex case, this is an immediate consequence of \cref{prop:base_change_cotangent_complex} and of \cref{rem:warning_cotangent_complex_pullback}.
	Let us now turn to the non-archimedean case.
	Using the transitivity fiber sequence, we see that there is a canonical map
	\[ g^* \anL_{X' / Y'} \to \anL_{X / Y} , \]
	and we claim that this map is an equivalence.
	This question is local on $X$ and on $Y$, and we can therefore suppose that $u \colon X \to Y$ factors as
	\[ \begin{tikzcd}
	X \arrow[hook]{r}{j} & Y \times \mathbf D^n_k \arrow{r}{p} & Y ,
	\end{tikzcd} \]
	where $j$ is a closed immersion and $p$ is the projection.
	We therefore get the following commutative diagram
	\[ \begin{tikzcd}
	X' \arrow{r}{i} \arrow{d}{g} & Y' \times \mathbf D^n_k \arrow{r}{q} \arrow{d}{h} & Y' \arrow{d}{f} \\
	X \arrow{r}{j} & Y \times \mathbf D^n_k \arrow{r}{p} & Y .
	\end{tikzcd} \]
	This diagram induces a morphism of fiber sequences
	\[ \begin{tikzcd}
	g^* j^* \anL_{Y \times \mathbf D^n_k / Y} \arrow{r} \arrow{d} &g^*  \anL_{X / Y} \arrow{r} \arrow{d} & g^* \anL_{X / Y \times \mathbf D^n_k} \arrow{d} \\
	i^* \anL_{Y' \times \mathbf D^n_k / Y'} \arrow{r} & \anL_{X' / Y'} \arrow{r} & \anL_{X' / Y' \times \mathbf D^n_k} .
	\end{tikzcd} \]
	Since $g^* j^* \simeq i^* h^*$, we are reduced to prove the following statements:
	\begin{enumerate}
		\item the morphism $h^* \anL_{Y \times \mathbf D^n_k / Y} \to \anL_{Y' \times \mathbf D^n_k / X'}$ is an equivalence;
		\item the morphism $g^* \anL_{X / Y \times \mathbf D^n_k} \to \anL_{X' / Y' \times \mathbf D^n_k}$ is an equivalence.
	\end{enumerate}
	In other words, we are reduced to prove the proposition in the special case where $u$ is either a closed immersion or a projection of the form $Y \times \mathbf D^n_k \to Y$.
	
	We first deal with the case of the closed immersion.
	Using \cite[Proposition 6.2]{Porta_Yu_Derived_non-archimedean_analytic_spaces}, we see that the above pullback square remains a pullback when considered in $\RTop(\cTank)$.
	We can therefore conclude by \cref{prop:base_change_cotangent_complex}.
	
	Let us now deal with the case of the projection $p \colon Y \times \mathbf D^n_k \to Y$.
	Consider the following ladder of pullback squares
	\[ \begin{tikzcd}
	Y' \times \mathbf D^n_k \arrow{r} \arrow{d} & Y' \arrow{d} \\
	Y \times \mathbf D^n_k \arrow{r} \arrow{d} & Y \arrow{d} \\
	\mathbf D^n_k \arrow{r} & \Sp(k) .
	\end{tikzcd} \]
	Reasoning as before, it is enough to prove that the proposition holds true for the outer square and the bottom one.
	By symmetry, it is sufficient to prove that the proposition holds for the bottom square.
	Since the question is local on $Y$, we can choose a closed immersion
	\[ j \colon Y \hookrightarrow \mathbf D^m_k . \]
	We can therefore further decompose the bottom square as
	\[ \begin{tikzcd}
	Y \times \mathbf D^n_k \arrow[hook]{r} \arrow{d} & \mathbf D^{n+m}_k \arrow{r} \arrow{d} & \mathbf D^n_k \arrow{d} \\
	Y \arrow[hook]{r}{j} & \mathbf D^m_k \arrow{r} & \Sp(k) . 
	\end{tikzcd} \]
	Once again, it is sufficient to prove the proposition for the square on the left and the one on the right.
	Since $j$ is a closed immersion, we already know that the proposition holds true for the square on the left.
	We are thus reduced to deal with the square on the right.
	Since the maps $\mathbf D^{n+m}_k \to \mathbf D^m_k$ and $\mathbf D^{n+m}_k \to \mathbf D^n_k$ are the projections, we see that they are the restriction of maps
	\[ \mathbf A^{n+m}_k \to \mathbf A^m_k , \quad \mathbf A^{n+m}_k \to \mathbf A^n_k . \]
	Furthermore, the inclusions $\mathbf D^l_k \to \mathbf A^l_k$ are \'etale.
	As a consequence, we can replace the polydisks by affine spaces.
	In this case, the proposition is a direct consequence of \cref{cor:analytification_relative_cotangent_complex}.
\end{proof}

\subsection{The analytic cotangent complex of a closed immersion} \label{subsec:cotangent_complex_closed_immersion}

The main result of this subsection asserts that the analytic cotangent complex of a closed immersion can be computed as the algebraic cotangent complex after forgetting the analytic structures.
We will then deduce from this result the connectivity estimates on the analytic cotangent complex.

Here is the precise statement:

\begin{thm} \label{thm:algebraic_vs_analytic_cotangent_complex}
	Let $\cX$ be an $\infty$-topos and let $f \colon \cA \to \cB$ be a morphism in $\AnRing_k(\cX)$.
	If $f$ is an effective epimorphism, then there is a canonical equivalence
	\[ (\anL_{\cB / \cA})\alg \simeq \mathbb L_{\cB\alg / \cA\alg} \]
	in $\cB\alg \Mod$, where $(\anL_{\cB / \cA})\alg$ denotes the image of $\anL_{\cB / \cA}$ under the functor
	\[ (-)\alg \colon \Sp\big( \Ab\big( \AnRing_k(\cX)_{\cA // \cB} \big)\big) \to \Sp\big( \Ab\big( \CRing_k(\cX)_{\cA\alg // \cB\alg} \big)\big) . \]
\end{thm}

The proof of the above theorem relies on the following lemma:

\begin{lem} \label{lem:effective_epimorphisms_infinite_suspension}
	Let $\cX$ be an $\infty$-topos and let $f \colon \cA \to \cB$ be a morphism in $\AnRing_k(\cX)$.
	Suppose that $f$ is an effective epimorphism.
	Then the commutative diagram
	\begin{equation} \label{eq:left_adjointable_closed_immersion_I}
		\begin{tikzcd}
			\Sp\big( \Ab\big( \CRing_k(\cX)_{\cA\alg // \cB\alg} \big) \big) \arrow{d}{\Omega^\infty_\Ab} & \Sp\big(\Ab\big( \AnRing_k(\cX)_{\cA // \cB} \big)\big) \arrow{l} \arrow{d}{\Omega^\infty_\Ab} \\
			\CRing_k(\cX)_{\cA\alg // \cB\alg} & \AnRing_k(\cX)_{\cA // \cB} \arrow{l}[swap]{(-)\alg} .
		\end{tikzcd}
	\end{equation}
	is left adjointable.
\end{lem}

\begin{proof}
	Using the canonical equivalences
	\begin{gather*}
		\Sp\big(\Ab\big(\CRing_k(\cX)_{\cA\alg // \cB\alg}\big)\big) \simeq \Sp\big(\Ab\big(\CRing_k(\cX)_{\cB\alg // \cB\alg}\big)\big), \\
		\Sp\big(\Ab\big(\AnRing_k(\cX)_{\cA // \cB}\big)\big) \simeq \Sp\big(\Ab\big(\AnRing_k(\cX)_{\cB // \cB}\big)\big) ,
	\end{gather*}
	we can decompose the square \eqref{eq:left_adjointable_closed_immersion_I} as
	\begin{equation} \label{eq:left_adjointable_closed_immersion_II}
		\begin{tikzcd}
			\Sp\big( \Ab\big( \CRing_k(\cX)_{\cA\alg // \cB\alg} \big) \big) \arrow{d}{\Omega^\infty_\Ab} & \Sp\big(\Ab\big( \AnRing_k(\cX)_{\cA // \cB} \big)\big) \arrow{l} \arrow{d}{\Omega^\infty_\Ab} \\
			\CRing_k(\cX)_{\cB\alg // \cB\alg} \arrow{d}{f\alg_!} & \AnRing_k(\cX)_{\cB // \cB} \arrow{l}[swap]{(-)\alg} \arrow{d}{f_!} \\
			\CRing_k(\cX)_{\cA\alg // \cB\alg} & \AnRing_k(\cX)_{\cA // \cB} \arrow{l}[swap]{(-)\alg} .
		\end{tikzcd}
	\end{equation}
	It is then enough to prove that both the upper and the lower squares are left adjointable.

	For the lower one, the statement is a consequence of the unramifiedness of the pregeometry $\cTank$: see \cite[Proposition 11.12]{DAG-IX} for the complex case and \cite[Proposition 3.17(iii)]{Porta_Yu_Derived_non-archimedean_analytic_spaces} for the non-archimedean case.
	Indeed, the left adjoints of $f_!$ and of $f\alg_!$ can respectively be described as the functors
	\[ \AnRing_k(\cX)_{\cA // \cB} \ni \cO \mapsto \cO {\cotimes}_\cA \cB \in \AnRing_k(\cX)_{\cB // \cB} \]
	and
	\[ \CRing_k(\cX)_{\cA\alg // \cB\alg} \ni \cO \mapsto \cO \otimes_{\cA\alg} \cB\alg \in \AnRing_k(\cX)_{\cB\alg // \cB\alg} . \]
	Since $f \colon \cA \to \cB$ is an effective epimorphism, unramifiedness of $\cTank$ implies that
	\[ ( \cO {\cotimes}_{\cA} \cB )\alg \simeq \cO\alg \otimes_{\cA\alg} \cB\alg . \]
	
	As for the upper one, it is enough to observe that given $\cO \in \AnRing_k(\cX)_{\cB // \cB}$, the canonical map  $\cO \to \cB$ has a section and it is therefore an effective epimorphism.
	In particular, using the unramifiedness of $\cTank$ once again, we obtain:
	\[ (\Sigma(\cO))\alg \simeq ( \cB {\cotimes}_\cO \cB )\alg \simeq \cB\alg \otimes_{\cO\alg} \cB\alg \simeq \Sigma( \cO\alg ). \]
	It follows that the upper square of \eqref{eq:left_adjointable_closed_immersion_II} is left adjointable as well.
\end{proof}

\begin{proof}[Proof of \cref{thm:algebraic_vs_analytic_cotangent_complex}]
	Applying \cref{lem:effective_epimorphisms_infinite_suspension} to the morphism $f$, we see that the square
	\[ \begin{tikzcd}
		\Sp\big(\Ab\big( \CRing_k(\cX)_{\cA\alg // \cB\alg} \big)\big) & \Sp\big( \Ab\big( \AnRing_k(\cX)_{\cA // \cB} \big)\big) \arrow{l}[swap]{(-)\alg} \\
		\CRing_k(\cX)_{\cA\alg // \cB\alg} \arrow{u}{\Sigma^\infty_\Ab} & \AnRing_k(\cX)_{\cA // \cB} \arrow{u}[swap]{\Sigma^\infty_\Ab} \arrow{l}[swap]{(-)\alg}
	\end{tikzcd} \]
	is commutative.
	Since $\cB$ is sent to $\cB\alg$ by the lower horizontal morphism, we conclude that
	\[ \mathbb L_{\cB\alg / \cA\alg} \simeq \Sigma^\infty_\Ab(\cB\alg) \simeq (\Sigma^\infty_\Ab(\cB))\alg \simeq (\anL_{\cB / \cA})\alg . \]
	The proof is therefore complete.
\end{proof}

\begin{cor} \label{cor:cotangent_complex_closed_immersion}
	Let $X = (\cX, \cO_X)$ and $Y = (\cY, \cO_Y)$ be derived analytic spaces and let $f \colon X \to Y$ be a closed immersion.
	There is a canonical equivalence $\mathbb L_{X\alg / Y\alg} \simeq \anL_{X/Y}$, where $X\alg$ and $Y\alg$ denote the $\cTetk$-structured topoi $(\cX, \cO_X\alg)$ and $(\cY, \cO_Y\alg)$ respectively.
\end{cor}

\begin{proof}
	The analytic cotangent complex $\anL_{X /Y}$ is by definition the analytic cotangent complex of the morphism $f\inv \cO_Y \to \cO_X$ in $\AnRing_k(\cX)$.
	Since $f$ is a closed immersion, this morphism is an effective epimorphism.
	The statement now follows from \cref{thm:algebraic_vs_analytic_cotangent_complex}.
\end{proof}

An important consequence of this fact is the connectivity estimates on the analytic cotangent complex.

\begin{prop} \label{prop:connectivity_estimates}
	Let $\cX$ be an $\infty$-topos and let $f \colon \cA \to \cB$ be a morphism in $\AnRing_k(\cX)$.
	Let $\mathrm{cofib}(f)$ denote the cofiber of the underlying map of $\DAb$-valued sheaves.
	If $\mathrm{cofib}(f)$ is $n$-connective for $n \ge 1$, then there is a canonical $(2n)$-connective map
	\[ \varepsilon_f \colon \mathrm{cofib}(f) \otimes_{A\alg} B\alg \to \mathbb L_{B / A}\an . \]
\end{prop}

\begin{proof}
	Since $\pi_0(\mathrm{cofib}(f)) = 0$, we see that $f$ is an effective epimorphism.
	Therefore, \cref{thm:algebraic_vs_analytic_cotangent_complex} implies that $\anL_{\cB / \cA} \simeq \mathbb L_{\cB\alg / \cA\alg}$.
	At this point, the statement follows immediately from \cite[7.4.3.1]{Lurie_Higher_algebra}.
\end{proof}

\begin{cor} \label{cor:connectivity_estimates}
	Let $\cX$ be an $\infty$-topos and let $f \colon \cA \to \cB$ be a morphism in $\AnRing_k(\cX)$.
	Assume that $\mathrm{cofib}(f)$ is $n$-connective for some $n \ge 1$.
	Then $\anL_{B / A}$ is $n$-connective.
	The converse holds provided that $f$ induces an isomorphism $\pi_0(A) \to \pi_0(B)$.
\end{cor}

\begin{proof}
	It follows from \cref{thm:algebraic_vs_analytic_cotangent_complex} and \cite[7.4.3.2]{Lurie_Higher_algebra}.
\end{proof}

\begin{lem}
	Let $f \colon A \to B$ be a morphism in $\AnRing_k(\cS)$.
	Then $\mathbb L_{B / A}\an$ is connective.
\end{lem}

\begin{proof}
	Let $M \in B\Mod$.
	Then
	\[ \Omega^\infty_B(M) \simeq \Omega^\infty_B(\tau_{\ge 0} M) . \]
	In particular, we obtain
	\begin{align*}
	\Map_{B\Mod}(\mathbb L_{B / A}\an, M) & \simeq \Map_{A // B}(B, \Omega^\infty_B(M)) \\
	& \simeq \Map_{A // B}(B, \Omega^\infty_B(\tau_{\ge 0} M)) \\
	& \simeq \Map_{B \Mod}(\mathbb L_{B / A}\an, \tau_{\ge 0} M)
	\end{align*}
	We conclude that for all $M\in B\Mod$, we have
	\[\Map_{B\Mod}(\anL_{B/A},\tau_{\le -1}M)\simeq 0.\]
	So $\anL_{B/A}$ is connective.
\end{proof}

\Cref{cor:connectivity_estimates} has the following important consequence:

\begin{cor} \label{cor:etale_characterised_by_cotangent_complex}
	Let $f \colon X \to Y$ be a morphism of derived analytic spaces.
	Then $f$ is \'etale if and only if $\trunc(f)$ is \'etale and $\mathbb L_{X / Y}\an \simeq 0$.
\end{cor}

\begin{proof}
	If $f$ is \'etale then \cref{cor:cotangent_complex_etale} shows that $\mathbb L_{X / Y}\an \simeq 0$.
	In this case, we also have $\trunc(X) \simeq \trunc(Y) \times_Y X$ and therefore $\trunc(f)$ is \'etale.
	Vice-versa, if $\trunc(f)$ is \'etale, we see that the underlying morphism of $\infty$-topoi is \'etale.
	Moreover, the morphism $f^\sharp \colon f\inv \cO_Y \to \cO_X$ induces an equivalence on $\pi_0$ by hypothesis, and its cotangent complex vanishes.
	It follows from \cref{cor:connectivity_estimates} that is an equivalence, completing the proof.
\end{proof}

Using the results obtained so far, we can also prove the following important property of the analytic cotangent complex:

\begin{prop} \label{prop:cotangent_complex_pullback}
	Let
	\[ \begin{tikzcd}
	X' \arrow{r}{q} \arrow{d}{p} & X \arrow{d}{f} \\
	Y' \arrow{r}{g} & Y
	\end{tikzcd} \]
	be a pullback square in $\dAnk$.
	Then the canonical diagram
	\[ \begin{tikzcd}
	q^* f^* \anL_Y \arrow{r} \arrow{d} & q^* \anL_X \arrow{d} \\
	p^* \anL_{Y'} \arrow{r} & \anL_{X'}
	\end{tikzcd} \]
	is a pushout square in $\cO_{X'}\Mod$.
\end{prop}

\begin{proof}
	Notice that if both $f$ and $g$ are closed immersion, the statement is a direct consequence of \cite[Proposition 3.17]{Porta_Yu_Derived_non-archimedean_analytic_spaces}, of \cref{cor:cotangent_complex_closed_immersion} and of \cite[7.3.2.18]{Lurie_Higher_algebra}.
	Furthermore, the question is local on $X$, $Y$ and $Y'$.
	We can therefore suppose that $f$ and $g$ factor respectively as
	\[ \begin{tikzcd}
		X \arrow[hook]{r}{i} & Y \times \bD^n_k \arrow{r}{\pi} & Y , & Y' \arrow[hook]{r}{j} & Y \times \bD^n_k \arrow{r}{\pi'} & Y ,
	\end{tikzcd} \]
	where $i$ and $j$ are closed immersions and $\pi$, $\pi'$ are the canonical projections.
	Since we already dealt with the case where both morphisms are closed immersions, we are reduced to prove the result for the following pullback square:
	\[ \begin{tikzcd}
		Y \times \bD^{n+m}_k \arrow{r} \arrow{d} & Y \times \bD^n_k \arrow{d} \\
		Y \times \bD^m_k \arrow{r} & Y .
	\end{tikzcd} \]
	Since the canonical inclusions $\bD^l_k \hookrightarrow \bA^l_k$ are \'etale, \cref{cor:etale_characterised_by_cotangent_complex} implies that we can replace the disks by the analytic affine spaces.
	The result is now a direct consequence of \cref{thm:analytification_cotangent_complex} and \cref{prop:base_change_cotangent_complex_analytic}.
\end{proof}

We conclude this subsection by the proving a finiteness result for the analytic cotangent complex.

\begin{defin}
	Let $X = (\cX, \cO_X)$ be a derived analytic space.
	The stable \infcat $\cO_X\Mod$ is naturally equipped with a $t$-structure (cf.\ \cite[2.1.3]{DAG-VIII}).
	We define the stable $\infty$-category $\Coh(X)$ of \emph{coherent sheaves} on $X$ to be the full subcategory of $\cO_X\Mod$ spanned by $\cF \in \cO_X\Mod$ such that $\pi_i(\cF)$ is a coherent sheaf of $\pi_0(\cO_X\alg)$-modules for every $i$.
	Furthermore, for every $n \in \mathbb Z$, we set
	\begin{gather*}
		\Coh^{\ge n}(X) \coloneqq \Coh(X) \cap \cO_X\Mod^{\ge n}, \quad \Coh^{\le n}(X) \coloneqq \Coh(X) \cap \cO_X\Mod^{\le n} , \\
		\Coh^+(X) \coloneqq \Coh(X) \cap \cO_X\Mod^+, \quad \Coh^-(X) \coloneqq \Coh(X) \cap \cO_X\Mod^- .
	\end{gather*}
\end{defin}

\begin{cor} \label{cor:finiteness_cotangent_complex}
	Let $f \colon X \to Y$ be a morphism of derived analytic spaces.
	Then $\anL_{X / Y}$ belongs to $\Coh^{\ge 0}(X)$.
\end{cor}

\begin{proof}
	Using \cref{prop:transitivity_cotangent_complex}, we see that it is enough to prove the statement in the absolute case.
	Moreover, notice that the question is local on $X$.
	
	We first deal with the non-archimedean case.
	Since we are working locally on $X$, we can use \cite[Lemma 6.3]{Porta_Yu_Derived_non-archimedean_analytic_spaces} to guarantee the existence of a closed immersion $j \colon X \hookrightarrow \mathbf D^n_k$.
	\cref{cor:cotangent_complex_closed_immersion} guarantees that $\anL_{X / \mathbf D^n_k}$ belongs to $\Coh^{\ge 0}(X)$.
	Using the transitivity fiber sequence
	\[ j^* \anL_{\mathbf D^n_k} \to \anL_X \to \anL_{X / \mathbf D^n_k} , \]
	we are therefore reduced to prove that the same thing holds true for $j^* \anL_{\mathbf D^n_k}$, and hence for $\anL_{\mathbf D^n_k}$.
	For the latter statement, we observe that there is a canonical morphism $\mathbf D^n_k \hookrightarrow \mathbf A^n_k$ which is an affinoid domain and in particular it is \'etale.
	As a consequence, it is enough to prove that $\anL_{\mathbf A^n_k} \in \Coh^{\ge 0}(\mathbf A^n_k)$.
	This is a consequence of \cref{cor:cotangent_complex_An}.
	
	In the complex analytic situation, the same proof works. We simply notice that we can always find, locally on $X$, a closed embedding $X \hookrightarrow \mathbf A^n_{\mathbb C}$ (cf.\ \cite[Lemma 12.13]{DAG-IX}).
\end{proof}

\subsection{Postnikov towers}

An invaluable tool in derived algebraic geometry is the Postnikov tower associated to a derived scheme.
More precisely, the fact that the transition maps in this tower are square-zero extensions allows to translate many problems in derived geometry into deformation theoretic questions.
This technique is extremely useful also in derived analytic geometry and we will use it repeatedly in the rest of this paper.

\begin{defin}
	Let $X \coloneqq (\cX, \cO_X)$ be a $\cTank$-structured topos and let $\cF \in \cO_X\Mod^{\ge 1}$ be an $\cO_X$-module.
	An \emph{analytic square-zero extension of $X$ by $\cF$} is a structured topos $X' \coloneqq (\cX, \cO)$ equipped with a morphism $f \colon X \to X'$ satisfying the following conditions:
	\begin{enumerate}
		\item The underlying geometric morphism of $\infty$-topoi is equivalent to the identity of $\cX$;
		\item There exists an analytic derivation $d \colon \anL_X \to \cF[1]$ such that the square
		\[ \begin{tikzcd}
		\cO \arrow{r} \arrow{d}{f^\sharp} & \cO_X \arrow{d}{\eta_d} \\
		\cO_X \arrow{r}{\eta_0} & \cO_X \oplus \cF[1]
		\end{tikzcd} \] 
		is a pullback square in $\AnRing_k(\cX)$.
	\end{enumerate}
\end{defin}

\begin{notation} \label{notation:analytic_square_zero_derivation}
	Let $X \coloneqq (\cX, \cO_X)$ be a derived analytic space, $\cF \in \Coh^{\ge 1}(X)$ be a coherent sheaf and $d \colon \anL_X \to \cF$ an analytic derivation.
	We denote by $\cO_X \oplus_d \cF$ the pullback
	\[ \begin{tikzcd}
		\cO_X \oplus_d \cF \arrow{r} \arrow{d} & \cO_X \arrow{d}{\eta_d} \\
		\cO_X \arrow{r}{\eta_0} & \cO_X \oplus \cF .
	\end{tikzcd} \]
	We denote by $X_d[\cF]$ the $\cTank$-structured topos $(\cX, \cO_X \oplus_d \cF)$.
	Notice that when $d$ is the zero derivation, $\cO_X \oplus_d \cF$ coincides with the split square-zero extension $\cO_X \oplus \cF[-1]$.
	We denote $X[\cF] \coloneqq X_0[\cF[1]]$, and call it the \emph{split square-zero extension} of $X$ by $\cF$.
\end{notation}

Recall that if $X \coloneqq (\cX, \cO_X)$ is a derived analytic space, then
\[ \tau_{\le n} \cO_X \colon \cTank \to \cX \]
is again a $\cTank$-structure (see \cite[Theorem 3.23]{Porta_Yu_Derived_non-archimedean_analytic_spaces} for the non-archimedean case and \cite[Proposition 11.4]{DAG-IX} for the complex case).
In particular, the $n$-th truncation
\[ \mathrm t_{\le n}(X) \coloneqq (\cX, \tau_{\le n} \cO_X) \]
is again a derived analytic space.
The main goal of this subsection is to prove that the canonical morphisms $\mathrm t_{\le n}(X) \hookrightarrow \mathrm t_{\le n +1}(X)$ are analytic square-zero extensions.
We will deduce it from the following more general result:

\begin{thm} \label{thm:lifting_derivations}
	Let $\cX$ be an $\infty$-topos and let $f \colon \cB \to \cA$ be an effective epimorphism in $\AnRing_k(\cX)$.
	Let $n$ be a non-negative integer and suppose that $f\alg \colon \cB\alg \to \cA\alg$ is an $n$-small extension in the sense of \cite[7.4.1.18]{Lurie_Higher_algebra}.
		Then $f$ is an analytic square-zero extension.
\end{thm}

\begin{proof}
	Consider the analytic derivation
	\[ d \colon \anL_\cA \to \anL_{\cA / \cB} \to \tau_{\le 2n} \anL_{\cA / \cB} \]
	and introduce the associated analytic square-zero extension
	\[ \begin{tikzcd}
	\cB' \arrow{r} \arrow{d} & \cA \arrow{d}{\eta_d} \\
	\cA \arrow{r}{\eta_0} & \cA \oplus \tau_{\le 2n} \anL_{\cA / \cB} .
	\end{tikzcd} \]
	We claim that the diagram
	\[ \begin{tikzcd}
	\cB \arrow{r}{f} \arrow{d}{f} & \cA \arrow{d}{\eta_d} \\
	\cA \arrow{r}{\eta_0} & \cA \oplus \tau_{\le 2n} \anL_{\cA / \cB}
	\end{tikzcd} \]
	is commutative.
	Indeed, the space of morphisms in $\AnRing_k(\cX)_{/\cA}$ from $\cB$ to $\cA \oplus \tau_{\le 2n} \anL_{\cA / \cB}$ is equivalent to the space
	\[ \Map_{\cB \Mod}(\anL_\cB, \tau_{\le 2n} \anL_{\cA / \cB}) . \]
	The composition $\eta_d \circ f$ corresponds to the composition
	\[ \anL_\cB \to \anL_\cB \otimes_{\cB\alg} \cA\alg \to \anL_\cA \xrightarrow{d} \tau_{\le 2n} \anL_{\cA / \cB} , \]
	and it is therefore homotopic to zero.
	This produces a canonical map
	\[ g \colon \cB \to \cB' . \]
	We claim that $g$ is an equivalence.
	
	Recall that the functor $(-)\alg$ is conservative (see \cite[Lemma 3.13]{Porta_Yu_Derived_non-archimedean_analytic_spaces} for the non-archimedean case and \cite[Proposition 11.9]{DAG-IX} for the complex case).
	In particular, it is enough to check that $g\alg$ is an equivalence.
	Using \cref{cor:split_square-zero_extension}, we can identify $(\cA \oplus \tau_{\le 2n} \anL_{\cB / \cA})\alg$ with the split square-zero extension
	\[ \cA\alg \oplus \tau_{\le 2n} \anL_{\cB / \cA} . \]
	As a consequence, $\eta_d\alg$ corresponds to the algebraic derivation
	\[ \mathbb L_\cA \to \anL_\cA \to \tau_{\le 2n} \anL_{\cA / \cB} . \]
	Since $f$ is an effective epimorphism, we can apply \cref{thm:algebraic_vs_analytic_cotangent_complex} to deduce that
	\[ \anL_{\cA / \cB} \simeq \mathbb L_{\cA\alg / \cB\alg} . \]
	Using \cite[7.4.1.26]{Lurie_Higher_algebra}, we conclude that the canonical morphism
	\[ g\alg \colon \cB\alg \to (\cB')\alg \]
	is an equivalence.
	This completes the proof.
\end{proof}

\begin{cor} \label{cor:Postnikov_tower_analytic_square_zero}
	For any derived analytic space $X$, every $n \ge 0$, the canonical map $\mathrm t_{\le n}(X) \hookrightarrow \mathrm t_{\le n+1}(X)$ is an analytic square-zero extension.
\end{cor}

\begin{proof}
	Using \cite[Theorem 3.23]{Porta_Yu_Derived_non-archimedean_analytic_spaces} in the non-archimedean case and \cite[Proposition 11.4]{DAG-IX} in the complex case, we deduce that there are natural equivalences
	\[ (\tau_{\le n} \cO_X)\alg \simeq \tau_{\le n}(\cO_X\alg) . \]
	The result is then a direct consequence of \cref{thm:lifting_derivations}.
\end{proof}

\subsection{The cotangent complex of a smooth morphism} \label{subsec:cotangent_complex_smooth}

As an application of the results we have obtained so far, we prove in this subsection that the cotangent complex of a smooth morphism of derived analytic spaces is perfect and in tor-amplitude 0.

\begin{defin}
	Let $\cX$ be an \inftopos and let $f \colon \cA \to \cB$ be a morphism in $\AnRing_k(\cX)$.
	We say that $f$ is \emph{strong} if the morphism $f\alg \colon \cA\alg \to \cB\alg$ is strong, i.e.\ if for every $i \ge 0$, it induces an equivalence
	\[ \pi_i(\cA\alg) \otimes_{\pi_0(\cB\alg)} \pi_0(\cA\alg) \xrightarrow{\sim} \pi_i(\cB\alg) . \]
\end{defin}

\begin{defin}
	Let $f \colon X \to Y$ be a morphism of derived analytic spaces.
	We say that $f$ is \emph{smooth} if it satisfies the following two conditions:
	\begin{enumerate}
		\item Locally on both $X$ and $Y$, $\trunc(f)$ is a smooth morphism of ordinary analytic spaces;
		\item The morphism $f\inv \cO_Y \to \cO_X$ is strong.
	\end{enumerate}
\end{defin}

\begin{lem} \label{lem:extension_strong_morphisms}
	Let $f \colon X \to Y$ and $g \colon Y \to Z$ be morphisms of derived analytic spaces.
	If $g$ and $g \circ f$ are strong, then the same goes for $f$.
\end{lem}

\begin{proof}
	Since $f\inv$ commutes with homotopy groups and it is monoidal, we see that it preserves strong morphisms.
	Therefore we are reduced to prove the following statement: if $A, B, C$ are sheaves of connective $\mathbb E_\infty$-rings on $\cX$ and $\alpha \colon A \to B$ and $\beta \colon B \to C$ are such that $\alpha$ and $\beta \circ \alpha$ are strong, then the same goes for $\beta$.
	Since $\cX$ has enough points, we are immediately reduced to the analogous statement for connective $\mathbb E_\infty$-rings.
	In this case, we only need to remark that:
	\[ \pi_i(B) \otimes_{\pi_0(B)} \pi_0(C) \simeq \pi_i(A) \otimes_{\pi_0(A)} \pi_0(B) \otimes_{\pi_0(B)} \pi_0(C) \simeq \pi_i(B) , \]
	so that the statement follows.
\end{proof}

The following lemma is a generalization of \cite[Lemma 6.3]{Porta_Yu_Derived_non-archimedean_analytic_spaces} and of \cite[Lemma 12.13]{DAG-IX}:

\begin{lem} \label{lem:vanishing_obstruction_theory}
	Let $X = (\cX, \cO_X)$ be a derived affinoid (resp.\ Stein) space.
	Suppose that $Y = (\cY, \cO_Y)$ is discrete and that $\anL_Y$ is perfect and in tor-amplitude $0$.
	Then any map $f \colon \trunc(X) \to Y$ admits an extension $\widetilde{f} \colon X \to Y$.
\end{lem}

\begin{proof}
	We proceed by induction on the Postnikov tower of $X$.
	In other words, we will construct a sequence of maps making the following diagram commutative:
	\[ \begin{tikzcd}
		\trunc(X) \arrow[hook]{r}{j_0} \arrow[bend right = 5]{drrrrr}[swap, near start]{f_0} & \mathrm t_{\le 1}(X) \arrow[hook]{r}{j_1} \arrow[bend right = 2]{drrrr}{f_1} & \cdots \arrow[hook]{r} & \mathrm t_{\le n}(X) \arrow[bend right = 1]{drr}{f_n} \arrow[hook]{r}{j_n} & \cdots \\
		& & & & & Y .
	\end{tikzcd} \]
	Note that the morphisms $j_n$ induce the identity on the underlying $\infty$-topos.
	In particular, all the maps $f_n$ are forced to have the same underlying geometric morphism of $\infty$-topoi, which we simply denote by
	\[ f\inv \colon \cY \leftrightarrows \cX \colon f_* . \]
	
	For the base step, we simply set $f_0 = f$.
	Suppose now that $f_n$ has been constructed.
	Recall that there is a fiber sequence
	\[ \Map_{\AnRing_k(\cX)}(f\inv \cO_Y, \tau_{\le n+1} \cO_X) \to \Map_{\dAnk}(\mathrm t_{\le n+1}(X), Y) \to \Map_{\RTop}(\cX, \cY) , \]
	the fiber being taken at the geometric morphism $(f\inv, f_*)$.
	Denote by $\varphi_n \colon f\inv \cO_Y \to \tau_{\le n} \cO_X$ the morphism induced by $f_n$.
	We are therefore reduced to solve the following lifting problem:
	\begin{equation} \label{eq:vanishing_obstruction_theory}
		\begin{tikzcd}
			{} & \tau_{\le n+1} \cO_X \arrow{d} \\
			f\inv \cO_Y \arrow{r}{\varphi_n} \arrow[dashed]{ur}{\varphi_{n+1}} & \tau_{\le n}(\cO_X) .
		\end{tikzcd}
	\end{equation}

	Set $\cF \coloneqq \pi_{n+1}(\cO_X) [n+2]$.
	Using \cref{cor:Postnikov_tower_analytic_square_zero}, we see that there exists an analytic derivation $d \colon \anL_{\mathrm t_{\le n} X} \to \cF$ such that the square
	\[ \begin{tikzcd}
		\tau_{\le n + 1} \cO_X \arrow{r} \arrow{d} & \tau_{\le n} \cO_X \arrow{d}{\eta_d} \\
		\tau_{\le n} \cO_X \arrow{r}{\eta_0} & \tau_{\le n} \cO_X \oplus \cF
	\end{tikzcd} \]
	is a pullback square in $\AnRing_k(\cX)$, where $\eta_0$ and $\eta_d$ correspond to the zero derivation and to $d$, respectively.
	This shows that the obstruction to solve the problem \eqref{eq:vanishing_obstruction_theory} lives in
	\[ \pi_0 \Map_{f\inv \cO_Y \Mod}( f\inv \anL_Y, \cF ) \simeq \pi_0 \Map_{\Coh^+(\mathrm t_{\le n}X)}( f_n^* \anL_Y, \cF ) . \]
		It is then enough to prove that the above mapping space vanishes.
	Since $X$ is a derived affinoid (resp.\ Stein), it is enough to check that
	\[ \cHom_{\Coh^+(X)}(f_n^* \anL_Y, \cF) \in \Coh^{\ge 1}(X) . \]
	We can therefore reason locally on $\mathrm t_{\le n}(X)$.
	As a consequence, we can assume $f_n^* \anL_Y$ to be retract of a free sheaf of $\cO_X$-modules.
	In this case, the statement follows because $\cF \in \Coh^{\ge 1}(X)$.
	Therefore, the obstruction to the lifting vanishes and we obtain the map $f_{n+1} \colon \mathrm t_{\le n+1}(X) \to Y$ we were looking for.
\end{proof}

\begin{prop} \label{prop:characterization_smooth_morphisms}
	Let $f \colon X \to Y$ be a morphism of derived analytic spaces.
	The following conditions are equivalent:
	\begin{enumerate}
		\item \label{item:strong_smooth} $f$ is smooth;
		\item \label{item:differential_smooth} $\trunc(f)$ is smooth and $\anL_{X / Y}$ is perfect and in tor-amplitude $0$;
		\item \label{item:factorization_smooth} Locally on both $X$ and $Y$, $f$ can be factored as
		\[ \begin{tikzcd}
		X \arrow{r}{g} & Y \times \mathbf A^n_k \arrow{r}{p} & Y ,
		\end{tikzcd} \]
		where $g$ is \'etale and $p$ is the canonical projection.
	\end{enumerate}
\end{prop}

\begin{proof}
	Let us start by proving the equivalence of (\ref{item:strong_smooth}) and (\ref{item:factorization_smooth}).
	The projection $p \colon Y \times \mathbf A^n_k \to Y$ is a smooth morphism, and every étale morphism is smooth.
	Therefore, if locally on $X$ and $Y$ we can exhibit such a factorization, we can deduce that $f$ is smooth.	
	Let us prove the converse.
	By definition of smooth morphism and up to localizing on $X$ and $Y$, we can suppose that we are already given a factorization of $\trunc(f)$ as
	\[ \begin{tikzcd}
	\trunc(X) \arrow{r}{g_0} & Y \times \mathbf A^n_k \arrow{r}{p} & Y .
	\end{tikzcd} \]
	Let $q \colon Y \times \mathbf A^n_k \to \mathbf A^n_k$ be the second projection.
	It follows from \cref{cor:cotangent_complex_An} and \cref{lem:vanishing_obstruction_theory} that we can extend $q \circ g_0$ to a morphism $h \colon X \to \mathbf A^n_k$.
	This determines a map $g \coloneqq f \times h \colon X \to Y \times \mathbf A^n_k$, which clearly extends $g_0$.
	By construction, $p \circ g \simeq f$. In particular, \cref{lem:extension_strong_morphisms} implies that $g$ is strong.
	This means that the canonical morphism
	\[ g^\sharp \colon g\inv \cO_{Y \times \mathbf A^n_k} \to \cO_X \]
	is strong. It is moreover an equivalence on $\pi_0$. It follows that $g^\sharp$ is an equivalence. In particular, $g$ is an \'etale morphism.
	
	We now prove the equivalence of (\ref{item:strong_smooth}) and (\ref{item:differential_smooth}).
	Assume first that (\ref{item:strong_smooth}) holds.
	Then $\trunc(f)$ is smooth, and thus all we have to prove is that $\anL_{X / Y}$ is perfect and in tor-amplitude $0$.
	This statement is local on both $X$ and $Y$.
	We can therefore use (\ref{item:factorization_smooth}) to factor $f$ as $p \circ g$, where $g \colon X \to Y \times \mathbf A^n_k$ is \'etale and $p \colon Y \times \mathbf A^n_k \to Y$ is the canonical projection.
	It follows from \cref{cor:etale_characterised_by_cotangent_complex} that $\anL_{X  / Y \times \mathbf A^n_k}$ vanishes.
	In particular, $\anL_{X / Y} \simeq f^* \anL_{Y \times \mathbf A^n_k / Y}$. Since $f$ is flat, it is therefore sufficient to prove the same statement for $p$.
	Applying \cref{prop:base_change_cotangent_complex_analytic} to the pullback square
	\[ \begin{tikzcd}
		Y \times \mathbf A^n_k \arrow{r}{p} \arrow{d}{q} & Y \arrow{d} \\
		\mathbf A^n_k \arrow{r} & \Sp(k) ,
	\end{tikzcd} \]
	we get a canonical equivalence
	\[ \anL_{Y \times \mathbf A^n_k / Y} \simeq q^* \anL_{\mathbf A^n_k} . \]
	The statement is therefore a consequence of \cref{cor:cotangent_complex_An}.
	
	Let us now assume that $\trunc(f)$ is smooth and that $\anL_{X / Y}$ is perfect and in tor-amplitude $0$.
	We prove that $f$ is strong.
	The question is local on both $X$ and $Y$, and therefore we can localize at a point in $X$, thus reducing to the analogous statement in $\AnRing_k \coloneqq \AnRing_k(\cS)$.
	In other words, we are given a morphism $\varphi \colon A \to B$ in $\AnRing_k$ whose analytic cotangent complex is perfect in tor-amplitute $0$, and we want to prove that $\varphi$ is strong.
	Form the pushout
	\[ \begin{tikzcd}
	A \arrow{r} \arrow{d} & B \arrow{d} \\
	\pi_0(A) \arrow{r} & C .
	\end{tikzcd} \]
	Observe that since $A \to \pi_0(A)$ is an effective epimorphism, $C\alg \simeq B\alg \otimes_{A\alg} \pi_0(A\alg)$.
	We have a canonical map $C \to \pi_0(B)$, and we claim that this is an equivalence.
	Suppose by contradiction that it is not.
	Let $i > 0$ be the smallest integer such that $\pi_i(C) \ne 0$.
	Let $C_i \coloneqq \tau_{\le i}(C)$.
	We have a fiber sequence
	\[ \anL_{C_i / \pi_0(A)} \otimes_{C_i} \pi_0(C) \to \anL_{\pi_0(C) / \pi_0(A)} \to \anL_{C_i / \pi_0(C)} . \]
	Since $\pi_0(C) \simeq \pi_0(B)$ and since by hypothesis $\trunc(f)$ is smooth, we conclude that $\anL_{\pi_0(C) / \pi_0(A)}$ is perfect and concentrated in degree $0$.
	In particular, we obtain a canonical identification
	\[ \pi_i( \anL_{C_i / \pi_0(A)} ) \simeq \pi_{i+1}( \anL_{C_i / \pi_0(C)} ) .  \]
	Note that \cref{cor:cotangent_complex_closed_immersion} and \cite[2.2.2.8]{HAG-II} imply together that
	\[ \pi_{i+1}( \anL_{C_i / \pi_0(C)}) \simeq \pi_i(C) . \]
	Using the connectivity estimates for the analytic cotangent complex provided by \cref{cor:connectivity_estimates}, we deduce that
	\[ \pi_i( \anL_{C / \pi_0(A)} \otimes_C \pi_0(C) ) \simeq \pi_i(\anL_{C_i / \pi_0(A)} \otimes_{C_i} \pi_0(C)) \simeq \pi_{i+1}( \anL_{C_i / \pi_0(C)}) \simeq \pi_i(C) \ne 0. \]
	On the other side, $\anL_{C / \pi_0(A)} \simeq \anL_{B / A} \otimes_B C$.
	In particular, it is perfect and in tor-amplitude $0$.
	Therefore, the same goes for $\anL_{C / \pi_0(A)} \otimes_C \pi_0(C)$.
	This is a contradiction, and so $C \simeq \pi_0(C)$.
	Since $\pi_0(A) \to \pi_0(B)$ is a flat map of ordinary rings, we can now apply \cite[7.2.2.13]{Lurie_Higher_algebra} to conclude that $\varphi \colon A \to B$ is strong.
	The proof is therefore complete.
\end{proof}

We conclude the subsection with the following useful lemma.

\begin{lem} \label{lem:cotangent_complex_underived_closed_immersion}
	Let $X$ and $Y$ be underived analytic spaces, and assume that $Y$ is smooth.
	Let $f \colon X \to Y$ be a closed immersion.
	Let $\cJ$ be the ideal sheaf on $Y$ defining $X$.
	Then $\tau_{\le 1}\mathbb L_X\an$ is non-canonically quasi-isomorphic to the complex
	\[ \cdots \to 0 \to \cJ / \cJ^2 \xrightarrow{\delta} f^* \Omega_Y\an \to 0 \to \cdots , \]
	where the map $\delta$ is induced by
	\[ \cJ \to \cO_Y \xrightarrow{d} \Omega_Y\an . \]
\end{lem}

\begin{proof}
	We start with some general considerations.
	Let $\cC$ be a stable $\infty$-category equipped with a left complete $t$-structure $(\cC_{\ge 0}, \cC_{\le 0})$.
	Let
	\[ M \to N \to P \]
	be a fiber sequence.
	Assume that $M \in \cC^\heartsuit$ and $N \in \cC_{\ge 0}$ and $P \in \cC_{\ge 1}$.
	Let
	\[ \delta \colon \pi_1(P) \to \pi_0(M) \]
	be the natural map.
	Write $P_1 \coloneqq \pi_1(P)$ (seen as an object in $\cC^\heartsuit$).
	As $M \in \cC^\heartsuit$, we have a canonical equivalence $M \simeq \pi_0(M)$.
	We can therefore review $\delta$ as a map $\delta \colon P_1 \to M$.
	Observe that the composition
	\[ P_1 \to M \to N \]
	induces the zero map on homotopy groups.
	Since the $t$-structure is complete, we deduce that the above composition is nullhomotopic.
	For any nullhomotopy $\alpha$, we thus obtain a canonical map
	\[ g_\alpha \colon \mathrm{cofib}( P_1 \xrightarrow{\delta} M ) \longrightarrow N . \]
	Write $Q \coloneqq \mathrm{cofib}( P_1 \xrightarrow{\delta} M)$.
	The five-lemma implies that $g_\alpha$ induces an isomorphism on $\pi_1$ and on $\pi_0$.
	We therefore obtain an equivalence (depending on $\alpha$)
	\[ h_\alpha \colon Q \simeq \tau_{\le 1} N . \]
	
	Let us apply this reasoning with $\cC = \Coh^+(X)$ and to the fiber sequence
	\[ f^* \anL_Y \to \anL_X \to \anL_{X / Y} . \]
	Notice that $f^* \anL_Y \in \Cohh(X)$ because $X$ is underived and $Y$ is smooth.
	On the other hand, since $f\inv \cO_Y \to \cO_X$ is surjective, \cref{cor:connectivity_estimates} implies that $\anL_{X / Y} \in \Coh^{\ge 1}(X)$.
	We therefore obtain a (non-canonical) quasi-isomorphism
	\[ \tau_{\le 1} \anL_X \simeq \mathrm{cofib}( \pi_1( \anL_{X / Y} ) \xrightarrow{\delta} j^* \Omega\an_Y ) . \]
	To complete the proof, we observe that there is a commutative square
	\[ \begin{tikzcd}
		\mathbb L_{X /Y} \arrow{r}{\delta\alg} \arrow{d} & j^* \Omega_Y \arrow{d} \\
		\anL_{X/Y} \arrow{r}{\delta} & j^* \Omega\an_Y .
	\end{tikzcd} \]
	Since $f\inv \cO_Y \to \cO_X$ is surjective, \cref{thm:algebraic_vs_analytic_cotangent_complex} implies that the left vertical map is an equivalence.
	Furthermore, the morphism $\delta\alg$ is obtained via the transitivity sequence for algebraic cotangent complexes for the morphism of locally ringed topoi
	\[ (\cX, \cO_X\alg) \longrightarrow (\cY, \cO_Y\alg) . \]
	We can therefore canonically identify $\delta\alg$ with the inclusion of the conormal sheaf of $f\inv \cO_Y\alg \to \cO_X\alg$ into $j^* \Omega_Y$.
	Recall now that the conormal sheaf is canonically identified with $\cJ / \cJ^2$ and the map to $j^* \Omega_Y$ is the one induced by
	\[ \cJ \to \cO_Y \xrightarrow{d\alg} \Omega_Y . \]
	Recall also that the diagram
	\[ \begin{tikzcd}
		\cO_Y \arrow{r}{d\alg} \arrow{dr} & \Omega_Y \arrow{d} \\
		{} & \Omega_Y\an
	\end{tikzcd} \]
	commutes.
	Thus $\delta$ coincides with the map induced by
	\[ \cJ \to \cO_Y \xrightarrow{d} \Omega\an_Y . \]
	This completes the proof.
\end{proof}

\section{Gluing along closed immersions}\label{sec:gluing}

In this section we prove that the $\infty$-category $\dAnk$ of derived analytic spaces is closed under pushout along closed immersions.
Using the Postnikov tower machinery provided by \cref{cor:Postnikov_tower_analytic_square_zero}, we can decompose the problem into two smaller tasks.
First, we need to know that the category of underived analytic spaces $\Ank$ is closed under pushout along closed immersions;
Second, we need to know that any analytic square-zero extension of a derived analytic space is again a derived analytic space.
This second problem is also a good testing ground for our notion of analytic derivation, hence our construction of the analytic cotangent complex.
The reason is that the square-zero extension of a derived analytic space by an arbitrary algebraic derivation is in general no longer a derived analytic space.

\begin{prop} \label{prop:analytic_square_zero_extensions_are_analytic_spaces}
	Let $X \coloneqq (\cX, \cO_X)$ be an underived analytic space.
	Let $\cF \in \Coh^\heartsuit(X)$ and let $X' \coloneqq (\cX, \cO')$ be an analytic square-zero extension of $X$ by $\cF$.
	Then $X'$ is an underived analytic space.
\end{prop}

\begin{proof}
	By definition, there exists an analytic derivation $\anL_X \to \cF[1]$ such that
	\[ \begin{tikzcd}
		\cO' \arrow{r} \arrow{d} & \cO_X \arrow{d}{\eta_d} \\
		\cO_X \arrow{r}{\eta_0} & \cO_X \oplus \cF[1]
	\end{tikzcd} \]
	is a pullback square in $\AnRing_k(\cX)$.
	Here $\eta_0$ corresponds to the zero derivation and $\eta_d$ corresponds to $d$.
	
	It follows that there is a fiber sequence
	\[ \cF \to \cO' \to \cO_X . \]
	Since both $\cO_X$ and $\cF$ are discrete, we conclude that the same goes for $\cO'$.
	We are thus left to check that $X'$ is an analytic space.
	This question is local on $X$ and we can therefore suppose that it is an affinoid (resp.\ Stein) space and admits a closed embedding $j \colon X \hookrightarrow Y$, where $Y$ is either $\mathbf D^n_k$ or $\mathbf A^n_{\mathbb C}$.
	
	Let $\cJ$ denote the sheaf of ideals defining $X$ as a closed subspace of $Y$.
	It follows from \cref{lem:cotangent_complex_underived_closed_immersion} that $\mathbb L_X\an$ satisfies the relation
	\[ \tau_{\le 1} \mathbb L_X\an \simeq ( \cdots \to 0 \to \cJ / \cJ^2 \to j^* \Omega_Y\an \to 0 \to \cdots ) .\]
	In particular, we can describe $\mathrm{Ext}^1_{\cO_X}(\mathbb L_X\an, \cF)$ as the cokernel of the map
	\[ \Hom_{\cO_X}(j^* \Omega_Y\an, \cF) \to \Hom_{\cO_X}(\cJ / \cJ^2, \cF) . \]
	Fix $\alpha \colon \cJ / \cJ^2 \to \cF$.
	We can describe the associated extension as follows.
	Let $Z$ denote the closed analytic subspace of $Y$ defined by the sheaf of ideals $\cJ^2$.
	Then we can see $\cF$ as a coherent sheaf on $Z$ and we introduce the split square-zero extension $Z[\cF]$.
	Let $\gamma \colon \cJ / \cJ^2 \to \cO_Y / \cJ^2 \simeq \cO_Z$ be the natural map and consider the morphism of $\cO_{Z[\cF]}$-modules $\beta \colon \cJ / \cJ^2 \to \cO_Z \oplus \cF$ defined by $x \mapsto (\gamma(x), \alpha(x))$.
	The image of $\beta$ is an ideal $\mathcal I$, and we have $\cO' = \cO_{Z[\cF]} / \mathcal I$.
	Since $Z[\cF]$ was a analytic space, the same goes for $X'$.
\end{proof}

\begin{prop} \label{prop:pushout_closed_immersions_underived}
	Let $i \colon X \to X'$ and $j \colon X \to Y$ be two closed immersions of underived analytic spaces.
	Then the pushout
	\[ \begin{tikzcd}
		X \arrow[hook]{r}{i} \arrow[hook]{d}{j} & X' \arrow{d} \\
		Y \arrow{r} & Y'
	\end{tikzcd} \]
	exists in $\Ank$.
	Furthermore, the forgetful functor $\Ank \to \RTop$ preserves this pushout.
	\end{prop}

\begin{proof}
	In the complex case, this follows from \cite[Théorème 3]{Cartan_Quotients_1960} and \cite[Proposition 6.4]{DAG-IX}.
	Let us now prove the non-archimedean case.
	By \cite[Theorem 3.4.1]{Temkin_Non-archimedean_pinchings}, the pushout $Y'$ exists in $\Ank$.
	On the other hand, \cite[Theorem 5.1]{DAG-IX} guarantees the existence of the pushout in $\RTop(\cTank)$, which we denote by $Y''$.
	The universal property of the pushout provides a canonical map $f \colon Y'' \to Y'$, making the following diagram commutative:
	\[ \begin{tikzcd}
		X \arrow{r}{i} \arrow{d}{j} & X' \arrow{d}{p} \arrow[bend left = 25pt]{ddr}{p'} \\
		Y \arrow{r}{q} \arrow[bend right = 25pt]{drr}{q'} & Y'' \arrow{dr}{f} \\
		& & Y'
	\end{tikzcd} \]
	We claim that $f$ is an equivalence in $\RTop(\cTank)$.
	This question is local on $Y'$, and hence also local on $X'$, $Y$ and $Y''$.
	Therefore, by \cite[Lemma 3.3.1]{Temkin_Non-archimedean_pinchings}, we can assume that
	\[ X = \Sp(A) , \quad X' = \Sp(A') , \quad Y = \Sp(B) , \quad Y'=\Sp(B')\]
	with
	\begin{equation} \label{eq:Artin-Tate}
		B' = A' \times_A B,
	\end{equation}
	for $k$-affinoid algebras $A$, $A'$, $B$ and $B'$.
	Let $\cX_A$ be the \'etale $\infty$-topos of $A$, and define similarly $\cX_{A'}$, $\cX_B$ and $\cX_{B'}$.
	Running the same proof of \cite[Corollary 6.5]{DAG-IX} (but using \cite[Proposition 3.5]{Porta_Yu_Derived_non-archimedean_analytic_spaces} instead of \cite[1.2.7]{DAG-VIII}), we deduce that
	\[ \begin{tikzcd}
		\cX_A \arrow{r}{i_*} \arrow{d}{j_*} & \cX_{A'} \arrow{d} \\
		\cX_B \arrow{r} & \cX_{B'}
	\end{tikzcd} \]
	is a pushout diagram in $\RTop$.
	In other words, if we denote by $\cX_{Y'}$ the underlying $\infty$-topos of $Y'$, the geometric morphism
	\[ f\inv \colon \cX_{B'} \leftrightarrows \cX_{Y'} \colon f_\ast \]
	is an equivalence.
	We are now left to verify that the canonical map
	\[ f\inv( \cO_{Y''} ) \longrightarrow \cO_{Y'} \]
	is an equivalence.
	However, \eqref{eq:Artin-Tate} shows that
	\[ f\inv( \cO_{Y'} ) = f\inv p'_*(\cO_{X'}) \times_{ f\inv p'_* i_*(\cO_X) } f\inv q'_*(\cO_Y) . \]
	On the other hand, the explicit construction of $Y''$ given in \cite[Theorem 5.1]{DAG-IX} shows that
	\[ \cO_{Y''} = p_*(\cO_{X'}) \times_{p_* i_*(\cO_X)} q_*(\cO_Y) . \]
	Since $p'_* \simeq f_* \circ p_*$ and $f_*$ is an equivalence, we deduce that $f\inv \circ p'_* \simeq p_*$, and similarly
	\[ f\inv \circ p'_* \circ i_* \simeq p_* \circ i_* \qquad \textrm{and} \qquad f\inv \circ q'_* \simeq q_* . \]
	This completes the proof.
	\end{proof}

We are now ready for the main theorem of this section:

\begin{thm} \label{thm:pushout_closed_immersions}
	Let
	\[ \begin{tikzcd}
		X \arrow{r}{i} \arrow{d}{j} & X' \arrow{d}{p} \\
		Y \arrow{r}{q} & Y'
	\end{tikzcd} \]
	be a pushout square in $\RTop(\cTank)$.
	Suppose that $i$ and $j$ are closed immersions and $X, X', Y$ are derived analytic spaces.
	Then $Y'$ is also a derived analytic space.
\end{thm}

Before starting the proof, we need the following technical lemma:

\begin{lem} \label{lem:closed_immersions_t_exact}
	Let $j_* \colon \cX \leftrightarrows \cY \colon j\inv$ be a closed immersion of $\infty$-topoi.
	Then $j_*$ commutes with truncations.
	In other words, there are natural equivalences
	\[ j_* \circ \tau_{\le n}^\cX \simeq \tau_{\le n}^\cY \circ j_* \]
	for every $n \ge 0$.
\end{lem}

\begin{proof}
	By definition of closed immersion, we can find a $(-1)$-truncated object $U \in \cY$ and an equivalence $\cX \simeq \cY / U$.
		The functor $j_* \colon \cY / U \to \cY$ is fully faithful and \cite[7.3.2.5]{HTT} guarantees that an object $V \in \cY$ belongs to $\cY / U$ if and only if $V \times U \simeq U$.
	Now let $V \in \cY / U$ and consider $\tau_{\le n}^\cY(V)$.
	Since $U$ is $(-1)$-truncated, we see that $\tau_{\le n}^\cY(U) \simeq U$ and therefore
	\[ \tau_{\le n}^\cY(V) \times U \simeq \tau_{\le n}^\cY(V) \times \tau_{\le n}^\cY(U) \simeq \tau_{\le n}^\cY(V \times U) \simeq \tau_{\le n}^\cY(U) \simeq U. \]
	In other words, $\tau_{\le n}^\cY(V)$ belongs to $\cY / U$.
	Since furthermore $j_*$ is fully faithful and commutes with $n$-truncated objects, we conclude that $\tau_{\le n}^\cY(V) \simeq \tau_{\le n}^\cX(V)$.
\end{proof}

\begin{proof}[Proof of \cref{thm:pushout_closed_immersions}]
	The question is local on $Y'$, so it is also local on $Y$ and on $X'$.
	We can therefore assume that $X$, $X'$ and $Y$ are derived affinoid (resp.\ Stein) spaces.
	
	Write
	\[ X = (\cX, \cO_X), \quad X' = (\cX', \cO_{X'}), \quad Y = (\cY, \cO_Y) , \quad Y' = (\cY', \cO_{Y'}) . \]
	The morphisms $i$ and $j$ induce closed immersions of the underlying $\infty$-topoi
	\[ i_* \colon \cX \leftrightarrows \cX' \colon i\inv , \qquad j_* \colon \cX \leftrightarrows \cY \colon j\inv . \]
	Using \cite[Theorem 5.1]{DAG-IX}, we can identify $\cY'$ with the pushout
	\[ \begin{tikzcd}
		\cX \arrow{r}{i_*} \arrow{d}{j_*} & \cX' \arrow{d}{p_*} \\
		\cY \arrow{r}{q_*} & \cY' 
	\end{tikzcd} \]
	computed in $\RTop$.
	Let $h \colon X \to Y'$ denote the compositions $p \circ i \simeq q \circ j$.
	We can use \cite[Theorem 5.1]{DAG-IX} once more to identify $\cO_{Y'}$ with the pullback
	\begin{equation} \label{eq:structure_sheaf_gluing}
		\begin{tikzcd}
			\cO_{Y'} \arrow{r} \arrow{d} & p_* \cO_{X'} \arrow{d} \\
			q_* \cO_Y \arrow{r} & h_* \cO_{X} .
		\end{tikzcd}
	\end{equation}
	In particular, we obtain a long exact sequence of homotopy groups
	\begin{multline} \label{eq:long_exact_sequence_pushout}
		\pi_1(p_* \cO_{X'}\alg) \oplus \pi_1(q_* \cO_Y\alg) \to \pi_1(h_* \cO_X\alg) \to \pi_0 \cO_{Y'}\alg\\ \to \pi_0 (p_*\cO_{X'}\alg) \oplus \pi_0 (q_* \cO_{Y}\alg) \to \pi_0(h_* \cO_X\alg) \to 0
	\end{multline}
	
	Now consider the truncations $\trunc(X)$, $\trunc(X')$, $\trunc(Y)$ and let $Y''$ be the pushout
	\[ \begin{tikzcd}
		\trunc(X) \arrow{r}{\trunc(i)} \arrow{d}{\trunc(j)} & \trunc(X') \arrow{d} \\
		\trunc(Y) \arrow{r} & Y''
	\end{tikzcd} \]
	in $\Ank$, whose existence is guaranteed by \cref{prop:pushout_closed_immersions_underived}.
	Furthermore, \cref{prop:pushout_closed_immersions_underived} ensures that the $\infty$-topos underlying $Y''$ coincides with $\cY'$ and that the structure sheaf $\cO_{Y''}$ fits in the following pullback diagram:
	\[ \begin{tikzcd}
		\cO_{Y''} \arrow{r} \arrow{d} & p_* \pi_0(\cO_{X'}) \arrow{d} \\
		q_* \pi_0(\cO_Y) \arrow{r} & h_* \pi_0(\cO_X) .
	\end{tikzcd} \]
	Using \cref{lem:closed_immersions_t_exact}, we deduce that there are canonical equivalences
	\[ p_* \pi_0(\cO_{X'}) \simeq \pi_0( p_* \cO_{X'} ) , \quad q_* \pi_0(\cO_X) \simeq \pi_0 ( q_* \cO_X ) , \quad h_* \pi_0(\cO_X) \simeq \pi_0( h_* \cO_X ) . \]
	We can therefore split the long exact sequence \eqref{eq:long_exact_sequence_pushout} into
	\[ 0 \to \cJ \to \pi_0( \cO_{Y'}\alg) \to \cO_{Y''}\alg \to 0 , \]
	where
	\[ \cJ \coloneqq \mathrm{coker}( \pi_1(p_* \cO_{X'}\alg) \oplus \pi_1(q_* \cO_Y\alg) \to \pi_1(h_* \cO_X\alg) ) . \]
	Using \cref{lem:closed_immersions_t_exact} once more, we deduce that there are the following natural equivalences:
	\[ \pi_1( p_* \cO_{X'} ) \simeq p_*( \pi_1 \cO_{X'} ), \quad \pi_1( q_* \cO_Y\alg ) \simeq q_*( \pi_1 \cO_Y\alg ), \quad \pi_1( h_* \cO_X\alg ) \simeq h_* \pi_1(\cO_X\alg ) . \]
	This implies that the above sheaves are coherent sheaves of $\cO_{Y''}\alg$-modules.
	As a consequence, we deduce that $\cJ$ is also a coherent sheaf of $\cO_{Y''}\alg$-modules.
	Finally, we observe that $\pi_0(\cO_{Y'}\alg)$ and $\cO_{Y''}\alg$ have the same support.
	This implies that $\cJ$ is (locally) a nilpotent sheaf of ideals of $\pi_0(\cO_{Y'}\alg)$.
	Proceeding by induction, we can therefore suppose that $\cJ^2 = 0$.
	
	We are therefore reduced to the case where $\pi_0( \cO_{Y'} )$ is a square-zero extension of $\cO_{Y''}$.
	In this case, we can invoke \cref{thm:lifting_derivations} to conclude that $\pi_0(\cO_{Y'})$ is an \emph{analytic} square-zero extension of $\cO_{Y''}$.
	Using \cref{prop:analytic_square_zero_extensions_are_analytic_spaces}, we conclude that the $\cTank$-structured topos $(\cY', \pi_0(\cO_{Y'}))$ is an analytic space.
	In order to complete the proof, we only have to prove that each $\pi_i( \cO_{Y'})$ is coherent over $\pi_0(\cO_{Y'})$.
	Observe that the morphisms
	\[ \pi_0(\cO_{Y'}) \to \pi_0(p_* \cO_{X'} ), \quad \pi_0(\cO_{Y'}) \to \pi_0(h_* \cO_X), \quad  \pi_0(\cO_{Y'}) \to \pi_0(q_* \cO_Y) \]
	are epimorphisms.
	The conclusion now follows from the long exact sequence associated to the pullback diagram \eqref{eq:structure_sheaf_gluing}.
\end{proof}

\section{The representability theorem}\label{sec:representability}

The goal of this section is to prove the main theorem of this paper, i.e.\ the representability theorem in derived analytic geometry.

Let $k$ be either the field $\C$ of complex numbers, or a complete non-archimedean field with nontrivial valuation.

Let $\Afd_k$ denote the category of $k$-affinoid spaces when $k$ is non-archimedean, and the category of Stein spaces when $k=\C$.
Let $\dAfd_k$ denote the \infcat of derived $k$-affinoid spaces when $k$ is non-archimedean, and the \infcat of derived Stein spaces when $k=\C$.

Let us first state the theorem before giving the precise definitions of the notions involved.

\begin{thm} \label{thm:representability}
	Let $F$ be a stack over the \infsite $(\dAfdk,\tauet)$.
	The followings are equivalent:
	\begin{enumerate}
		\item \label{item:F_is_representable} $F$ is an $n$-geometric stack with respect to the geometric context $(\dAfdk,\tauet,\bPsm)$;
		\item \label{item:F_satisfies_rep_conditions} $F$ is compatible with Postnikov towers, has a global analytic cotangent complex, and its truncation $\trunc(F)$ is an $n$-geometric stack with respect to the geometric context $(\Afd_k,\tauet,\bPsm)$.
	\end{enumerate}
\end{thm}

We refer to \cite[\S 2]{Porta_Yu_Higher_analytic_stacks_2014} for the notions of geometric context and geometric stack with respect to a given geometric context.
Recall that a geometric context $(\cC,\tau,\bP)$ consists of a small \infcat $\cC$ equipped with a Grothendieck topology $\tau$ and a class $\bP$ of morphisms in $\cC$, satisfying a short list of axioms.
In the statement of \cref{thm:representability}, $\tauet$ denotes the étale topology and $\bPsm$ denotes the class of smooth morphisms.

A stack over an \infsite $(\cC,\tau)$ is by definition a hypercomplete sheaf with values in spaces over the \infsite.
We denote by $\St(\cC,\tau)$ the \infcat of stacks over $(\cC,\tau)$.

Given a geometric context $(\cC,\tau,\bP)$ and an integer $n\ge -1$, the notion of $n$-geometric stack is defined by induction on the geometric level $n$.
We refer to \cite[\S 2.3]{Porta_Yu_Higher_analytic_stacks_2014} for the details.
Let us simply recall that a $(-1)$-geometric stack is by definition a representable stack.

\begin{defin}
	A \emph{derived analytic stack} is an $n$-geometric stack with respect to the geometric context $(\dAfd_k, \tauet, \bPsm)$ for some $n$.
\end{defin}

The following definitions are analytic analogues of the algebraic notions introduced in \cite{DAG-XIV,HAG-II}.

\begin{defin}
	Let $f \colon F \to G$ be a morphism in $\St(\dAfdk,\tauet)$.
	We say that $f$ is \emph{infinitesimally cartesian} if for every derived affinoid (resp.\ Stein) space $X \in \dAfd_k$, every coherent sheaf $\cF \in \Coh^{\ge 1}(X)$ and every analytic derivation $d \colon \anL_X \to \cF$, the square
	\[ \begin{tikzcd}
	F(X_d[\cF]) \arrow{r} \arrow{d} & G(X_d[\cF]) \arrow{d} \\
	F(X) \times_{F(X[\cF])} F(X) \arrow{r} & G(X) \times_{G(X[\cF])} G(X)
	\end{tikzcd} \]
	is a pullback square.
	We say that a stack $F \in \St(\dAfdk,\tauet)$ is \emph{infinitesimally cartesian} if the canonical map $F \to *$ is infinitesimally cartesian, where $*$ denotes a final object of $\St(\dAfdk,\tauet)$.
\end{defin}

\begin{defin}
	Let $f \colon F \to G$ be a morphism in $\St(\dAfdk,\tauet)$.
	We say that $f$ is \emph{convergent} (or \emph{nil-complete}) if for every derived affinoid (resp.\ Stein) space $X = (\cX, \cO_X) \in \dAfd_k$, the square
	\[ \begin{tikzcd}
	F(X) \arrow{r} \arrow{d} & \lim_n F( \mathrm t_{\le n} X) \arrow{d} \\
	G(X) \arrow{r} & \lim_n G( \mathrm t_{\le n} X )
	\end{tikzcd} \]
	is a pullback square.
	We say that a stack $F \in \St(\dAfdk,\tauet)$ is \emph{convergent} if the canonical map $F \to *$ is convergent, where $*$ denotes a final object of $\St(\dAfdk,\tauet)$.
\end{defin}

\begin{defin} \label{def:compatible_Postnikov}
	A morphism $f\colon F\to G$ is said to be \emph{compatible with Postnikov towers} if it is infinitesimally cartesian and convergent.
\end{defin}

Let $F \in \St(\dAfdk,\tauet)$.
Let $X \in \dAfd_k$ and let $x \colon X \to F$ be a morphism of sheaves.
For every coherent sheaf $\cF \in \Coh^{\ge 0}(X)$, we denote by $\DerAn_F(X, \cF)$ the fiber at $x$ of the canonical map 
\[ F(X[\cF]) \to F(X) . \]
This assignment is functorial in $\cF$ and therefore provides us a functor
\[ \DerAn_{F}(X, -) \colon \Coh^{\ge 0}(X) \to \cS . \]
If $f \colon F \to G$ is a morphism of sheaves, we obtain a natural transformation
\[ \eta \colon \DerAn_F(X, -) \to \DerAn_G(X, -) \]
for every fixed $X \in \dAfd_k$ and every fixed morphism $x \colon X \to F$.
For every $\cF \in \Coh^{\ge 0}(X)$, the space $\DerAn_G(X,\cF)$ has a distinguished element: the zero derivation.
Let us denote the fiber of $\eta_\cF$ at the zero derivation by $\DerAn_{F/G}(X,\cF)$.
It is naturally functorial in $\cF$.
We denote the corresponding functor by
\[ \DerAn_{F/G}(X, -) \colon \Coh^{\ge 0}(X) \to \cS . \]

\begin{defin}
	Let $f \colon F \to G$ be a morphism in $\St(\dAfdk,\tauet)$.
	\begin{enumerate}
		\item Let $X \in \dAfd_k$ and let $x \colon X \to F$ be a morphism. We say that \emph{$f$ has an analytic cotangent complex at $x$} if the functor
		\[ \DerAn_{F/G}(X,-) \colon \Coh^{\ge 0}(X) \to \cS \]
		is corepresentable by an object in $\Coh^+(X)$.
				In this case, we denote this object by $\anL_{F/G, x}$.
		\item We say that \emph{$f$ has a global analytic cotangent complex} if the following conditions are satisfied:
		\begin{enumerate}
			\item $f$ has an analytic cotangent complex at every morphism $x \colon X \to F$ for every $X \in \dAfd_k$;
			\item For any morphism $g \colon X \to Y$ in $\dAfd_k$, any morphism $y \colon Y \to F$, denote $x \coloneqq y \circ g$.
			Then the canonical morphism
			\[ g^* \anL_{F / G, y} \to \anL_{F / G, x} \]
			is an equivalence in $\Coh^+(X)$.
		\end{enumerate}
	\end{enumerate}
\end{defin}

For the proof of \cref{thm:representability}, we will address the implication (1)$\Rightarrow$(2) in \cref{sec:properties_of_derived_analytic_stacks}, and the implication (2)$\Rightarrow$(1) in \cref{sec:lifting_atlases}.

\subsection{Properties of derived analytic stacks} \label{sec:properties_of_derived_analytic_stacks}

In this subsection, we prove the implication (1)$\Rightarrow$(2) of \cref{thm:representability}.
We will first prove that (2) holds for derived analytic spaces.
After that, we will prove (2) for derived analytic stacks by induction on the geometric level.

\begin{lem} \label{lem:representable_satisfies_rep_conditions}
	Let $X = (\cX, \cO_X)$ be a derived analytic space and let $F_X \in \St(\dAfdk,\tauet)$ be the associated stack via the Yoneda embedding.
	Then $F_X$ is infinitesimally cartesian, convergent and it admits a global analytic cotangent complex.
\end{lem}

\begin{proof}
	Let $Y \in \dAfdk$ be a derived affinoid (resp.\ Stein) space.
	Let $\cF \in \Coh^{\ge 0}(Y)$ and let $d \colon \anL_Y \to \cF$ be an analytic derivation.
	It follows from \cref{thm:pushout_closed_immersions} that the diagram
	\[ \begin{tikzcd}
	Y[\cF] \arrow{r} \arrow{d} & Y \arrow{d} \\
	Y \arrow{r} & Y_d[\cF]
	\end{tikzcd} \]
	is a pushout square in $\dAnk$.
	As a consequence, $F_X = \Map_{\dAnk}(-, X)$ takes this diagram into a pullback square.
	In other words, $F_X$ is infinitesimally cartesian.
	
	Let $Y = (\cY, \cO_Y) \in \dAfd_k$.
	Since $\cY$ is hypercomplete, we deduce that the canonical map
	\[ \colim_n \mathrm{t}_{\le n}(Y) \to Y \]
	is an equivalence in $\dAnk$.
	In particular,
	\begin{align*}
	F_X(Y) & = \Map_{\dAnk}(Y, X) \\
	& \simeq \Map_{\dAnk}(\colim_n \mathrm t_{\le n}(Y), X) \\
	& \simeq \lim_n \Map_{\dAnk}(\mathrm t_{\le n}(Y), X) \simeq \lim_n F_X( \mathrm t_{\le n}(Y)) .
	\end{align*}
	It follows that $F_X$ is convergent.
	
	Let us now show that $F_X$ admits a global cotangent complex.
	Let $\anL_X$ be the analytic cotangent complex of $X$ introduced in \cref{subsec:analytic_cotangent_complex}.
	It follows from \cref{cor:finiteness_cotangent_complex} that $\anL_X \in \Coh^{\ge 0}(X)$.
	It will therefore be sufficient to prove that for every derived affinoid (resp.\ Stein) space $Y = (\cY, \cO_Y)$ and every map $y \colon Y \to F_X$, the object $y^* \anL_X \in \Coh^{\ge 0}(Y)$ satisfies the universal property of the analytic cotangent complex.
	Recall now that derived analytic spaces embed fully faithfully in $\St(\dAfdk, \tauet)$: in the non-archimedean case, this follows from \cite[Theorem 7.9]{Porta_Yu_Derived_non-archimedean_analytic_spaces}, while in the complex case this is a consequence of \cite[Theorem 3.7]{Porta_DCAGI}.
	Therefore the map $y$ corresponds to a unique (up to a contractible space of choices) map $f_y \colon Y \to X$ in $\dAnk$.
	Using again the fully faithfulness of the embedding of derived analytic spaces in $\St(\dAfdk, \tauet)$, we conclude that
	\[ \DerAn_F(X, \cF) = \Map_{\AnRing_k(\cY)_{/\cO_Y}}( f_y\inv \cO_X, \cO_Y \oplus \cF ) \simeq \Map_{\Coh^+(Y)}(y^* \anL_X, \cF) . \]
	This completes the proof.
\end{proof}

We will now show that the above conditions are also satisfied by derived analytic stacks.
Our arguments are similar to \cite[§1.4.3]{HAG-II}.

\begin{lem}
	Let $F \in \St(\dAfdk,\tauet)$.
	If $F$ is infinitesimally cartesian, then for every $X \in \dAfd_k$, every point $x \colon X \to F$ and every connective coherent sheaf  $\cF \in \Coh^{\ge 0}(X)$, the canonical morphism
	\[ \DerAn_F(X, \cF) \to \Omega \DerAn_F(X, \cF[1]) \]
	is an equivalence.
\end{lem}

\begin{proof}
	Let $X \in \dAfd_k$ be a derived affinoid (resp.\ Stein) space, and let $\cF \in \Coh^{\ge 0}(X)$.
	Since $F$ is infinitesimally cartesian, we have a pullback square
	\[ \begin{tikzcd}
	F(X[\cF]) \arrow{r} \arrow{d} & F(X) \arrow{d} \\
	F(X) \arrow{r} & F(X[\cF[1]])  .
	\end{tikzcd} \]
	We have a canonical map $F(X[\cF[1]]) \to F(X)$ induced by the closed immersion $X \to X[\cF[1]]$.
	Taking fibers at $x \in \pi_0(F(X))$, we obtain a pullback square
	\[ \begin{tikzcd}
	\DerAn_F(X, \cF) \arrow{r} \arrow{d} & \{*\} \arrow{d} \\
	\{*\} \arrow{r} & \DerAn_F(X, \cF[1]) .
	\end{tikzcd} \]
	Hence, we conclude that $\DerAn_F(X, \cF) \simeq \Omega \DerAn_F(X, \cF[1])$.
\end{proof}

\begin{prop} \label{prop:inf_cartesian_implies_obstruction_theory}
	Let $F \in \St(\dAfdk,\tauet)$ be an $n$-geometric stack with respect to the geometric context $(\dAfdk,\tauet,\bPsm)$.
	If $F$ is infinitesimally cartesian, then it has a global cotangent complex, which is $(-n)$-connective.
\end{prop}

\begin{proof}
	We follow closely the proof of \cite[1.4.1.11]{HAG-II}.
	We proceed by induction on $n$.
	If $n = -1$, then the statement follows from \cref{lem:representable_satisfies_rep_conditions}.
	Let therefore $n \ge 0$ and let $F$ be an $n$-geometric stack and $x \colon X \to F$ be a point, with $X \in \dAfd_k$.
	Consider the natural morphisms
	\[ \delta \colon X \to X \times X, \qquad \delta_F \colon X \to X \times_F X . \]
	By induction, both $X \times X$ and $X \times_F X$ have analytic cotangent complexes at $\delta$ and at $\delta_F$.
	Let us denote by by $\mathbb L, \mathbb L'$, respectively.
	The canonical map since $\delta$ factors through $\delta_F$, there is a canonical map $f \colon \mathbb L \to \mathbb L'$ in $\Coh^+(X)$.
	Let $\mathbb L'' \coloneqq \mathrm{cofib}(f)$.
	By definition, for any $\cF \in \Coh^{\ge 0}(X)$, the space $\Map_{\Coh^+(X)}(\mathbb L'', \cF)$ is the fiber of
	\[ \DerAn_X(X \times_F X, \cF) \to \DerAn_X(X \times X, \cF) . \]
	Now, $\DerAn_X(X \times X, \cF) \simeq \{*\}$, while
	\[ \DerAn_X(X \times_F X, \cF) \simeq \DerAn_X(X, \cF) \times_{\DerAn_F(X, \cF)} \DerAn_X(X, \cF) \simeq \Omega \DerAn_F(X,\cF) . \]
	As consequence,
	\begin{align*}
	\Map_{\Coh^+(X)} (\Omega(\mathbb L''), \cF) & \simeq \Map_{\Coh^+(X)}(\mathbb L'', \cF[1]) \\
	& \simeq \Omega \DerAn_F(X, \cF[1]) \simeq \DerAn_F(X, \cF).
	\end{align*}
	It follows that $F$ has a cotangent complex at $x$.
	Moreover, the inductive hypothesis shows that both $\mathbb L$ and $\mathbb L'$ are $(-n+1)$-connective.
	It follows that $\mathbb L''$ is $(-n+1)$-connective as well, and therefore $\Omega(\mathbb L'') = \mathbb L''[-1]$ is $(-n)$-connective.
	The same argument of \cite[1.4.1.12]{HAG-II} shows that $F$ has a global cotangent complex.
\end{proof}

Our next task is to show that any $n$-geometric stack with respect to the geometric context $(\dAfdk,\tauet,\bPsm)$ is infinitesimally cartesian.

Let us recall that the notion of smooth morphism between derived analytic spaces is local on both source and target.
Therefore, we can extend it to representable morphisms in $\St(\dAfdk, \tauet)$ (cf.\ \cite[Remark 2.10]{Porta_Yu_Higher_analytic_stacks_2014}).
More explicitly, an $n$-representable morphism $f \colon F \to G$ in $\St(\dAfdk, \tauet)$ is smooth if and only if for every $U \in \dAfdk$ and every map $U \to G$, there exists an atlas $\{V_i\}$ of $U \times_G F$ such that the compositions $V_i \to U$ are smooth morphisms of derived analytic spaces.

\begin{prop} \label{prop:geometric_implies_inf_cartesian}
	\begin{enumerate}
		\item \label{item:n-representable_obs_theory} Any $n$-representable morphism of stacks is infinitesimally cartesian.
		\item \label{item:smooth_vanishing_map} Let $f \colon F \to G$ be an $n$-representable morphism. If $f$ is smooth, then for any $X \in \dAfd_k$ and any $x \colon X \to F$ there exists an \'etale covering $x' \colon X' \to X$ such that for any $\cF \in \Coh^{\ge 1}(X')$ the canonical map
		\[ \pi_0 \Map_{\Coh^+(X')}(\anL_{X / G, x'}, \cF) \to \pi_0 \Map_{\Coh^+(X')}(\anL_{F / G, x \circ x'}, \cF) \]
		is zero.
	\end{enumerate}
\end{prop}

\begin{proof}
	We proceed by induction on $n$.
	If $n = -1$, then (\ref{item:n-representable_obs_theory}) follows from \cref{lem:representable_satisfies_rep_conditions} and (\ref{item:smooth_vanishing_map}) follows from \cref{prop:characterization_smooth_morphisms}.
	
	Let now $n \ge 0$.
	We will start by proving (\ref{item:n-representable_obs_theory}).
	It is enough to prove that if $F$ is $n$-geometric then it is infinitesimally cartesian.
		Let $X \in \dAfd_k$, $\cF \in \Coh^{\ge 1}(X)$ and $d \colon X[\cF] \to X$ be an analytic derivation.
	Let $x$ be a point in $\pi_0( F(X) \times_{F(X[\cF])} F(X) )$ with projection $x_1 \in \pi_0(F(X))$ on the first factor.
	We will prove that the fiber taken at $x$ of
	\[ F(X_d[\cF]) \to F(X) \times_{F(X[\cF])} F(X) \]
	is contractible.
	This implies that the above morphism is an equivalence and therefore that $F$ is infinitesimally cartesian.
	
	We claim that this statement is local for the \'etale topology on $X_d[\cF]$.
	Indeed, if $j' \colon U' \to X_d[\cF]$ is an \'etale map in $\dAfd_k$, let
	\[ j \colon U \coloneqq U' \times_{X_d[\cF]} X \to X \]
	be the \'etale map obtained by base change.
	Then since the formation of analytic square-zero extension is local on any structured topos, we obtain that
	\[ U' \times_{X_d[\cF]} X[\cF] \simeq U[j^* \cF], \qquad U' \simeq U_{d'}[j^* \cF] . \]
	As consequence, we are free to replace $X$ by any \'etale cover.
	
	Choose an $(n-1)$-atlas $\{U_i \to F\}_{i \in I}$ of $F$.
	Thanks to the above claim, we can assume that the point $x_1 \in \pi_0(F(X))$ lifts to a point $y_1 \in \pi_0(U_i(X))$ for some index $i \in I$.
	Write simply $U \coloneqq U_i$.
	Consider the diagram
	\[ \begin{tikzcd}
	U(X_d[\cF]) \arrow{r} \arrow{d} & U(X) \times_{U(X[\cF])} U(X) \arrow{d}{f} \\
	F(X_d[\cF]) \arrow{r} & F(X) \times_{F(X[\cF])} F(X) .
	\end{tikzcd} \]
	The induction hypothesis applied to the $(n-1)$-representable morphism $\pi \colon U \to F$ shows that the above square is a pullback.
	Moreover, the top horizontal morphism is an equivalence.
	It follows that the fibers of the bottom horizontal morphism is either empty or contractible.
	In order to complete the proof of (\ref{item:n-representable_obs_theory}), it is thus sufficient to prove that the fiber of $f$ at $x$ is non-empty.
	Consider the following diagram
	\[ \begin{tikzcd}
	\mathrm{fib}(g) \arrow{r} \arrow{d} & \mathrm{fib}(p) \arrow{r}{g} \arrow{d} & \mathrm{fib}(q) \arrow{d} \\
	\mathrm{fib}(f) \arrow{r} & U(X) \times_{U(X[\cF])} U(X) \arrow{r}{f} \arrow{d}{p} & F(X) \times_{F(X[\cF])} F(X) \arrow{d}{q} \\
	{} & U(X) \arrow{r} & F(X) ,
	\end{tikzcd} \]
	where the fiber of $q$ (resp.\ $p$) are taken at $x_1$ (resp.\ $y_1$), while the horizontal fibers are taken at $x$.
	The commutativity of the diagram shows that it is enough to prove that $\mathrm{fib}(g)$ is non-empty.
	Now, $g$ is equivalent to the canonical map
	\[ \Omega_{d,0} \DerAn_U(X, \cF) \to \Omega_{d,0} \DerAn_F(X, \cF) , \]
	and therefore $\mathrm{fib}(g) \simeq \Omega_{d,0} \Map_{\Coh^+(X)}( \anL_{U/F, y_1}, \cF)$.
	The composition $X \to U \to F$ gives rise to the following exact sequence:
	\[ \pi_0 \Map_{\Coh^+(X)}( \anL_{X/F, x_1}, \cF) \to \pi_0 \Map_{\Coh^+(X)}( \anL_{U/F, y_1}, \cF ) \to \pi_{-1} \Map_{\Coh^+(X)}(\anL_{X/U, y_1}, \cF) . \]
	Using (\ref{item:smooth_vanishing_map}) at rank $(n-1)$ for the map $\pi \colon U \to F$ and up to cover $X$ with an \'etale atlas, we can therefore suppose that the first map vanishes.
	On the other hand, the image of $d$ via the second map is zero.
		Therefore, $d$ lies in the image of $\pi_0 \Map_{\Coh^+(X)}(\anL_{X/F,x_1})$, i.e.\ $d$ is in the connected component containing $0$.
	In particular, we can find a path from $d$ to $0$ in $\Map(\anL_{F / U,y_1}, \cF)$.
	This shows that $\Omega_{d,0} \Map_{\Coh^+(X)}(\anL_{U/F,y_1}, \cF)$ is non-empty and concludes the proof of (\ref{item:n-representable_obs_theory}).
	
	We now turn to the proof of (\ref{item:smooth_vanishing_map}) for rank $n$.
	We can assume that $G$ is a final object.
	Let $U \to F$ be an $n$-atlas and let $x \colon X \to F$ be a point, with $X \in \dAfd_k$.
	Up to choosing an \'etale cover of $X$, we can suppose that $x$ factors through a point $u \colon X \to U$.
	Therefore, the map $\anL_{F,x} \to \anL_X$ factors as
	\[ \anL_{F,x} \to \anL_{U,u} \to \anL_X . \]
	Since $U$ is smooth, \cref{prop:characterization_smooth_morphisms} shows that $\anL_{U,u}$ is perfect and concentrated in degree $0$.
	Therefore, for every $\cF \in \Coh^{\ge 1}(X)$, we have
	\[ \pi_0 \Map_{\Coh^+(X)}(\anL_{U,u}, \cF) = 0 , \]
	thus completing the proof.
\end{proof}

In order to prove the convergence property of $n$-representable maps, we need a characterization of smooth morphisms in terms of infinitesimal lifting properties.

\begin{prop} \label{prop:infinitesimal_characterization_smooth}
	Let $f \colon F \to G$ be an $n$-representable morphism in $\St(\dAfdk, \tauet)$ with respect to the geometric context $(\dAfdk,\tauet,\bPsm)$.
	Then $f$ is smooth if and only if it satisfies the following conditions:
	\begin{enumerate}
		\item $\trunc(f)$ is smooth;
		\item \label{item:infinitesimal_smooth} for any derived affinoid (resp.\ Stein) space $X \in \dAfd_k$, any $\cF \in \Coh^{\ge 1}(X)$ and any $d \in \DerAn(X, \cF)$, every lifting problem
		\begin{equation} \label{eq:lifting_problem}
		\begin{tikzcd}
		X \arrow{r}{x} \arrow{d} & F \arrow{d}{f} \\
		X_d[\cF] \arrow{r} \arrow[dashed]{ur} & G
		\end{tikzcd}
		\end{equation}
		admits at least a solution.
	\end{enumerate}
\end{prop}

\begin{proof}
	First suppose that $f$ is smooth.
	Then there exists an affinoid atlas $\{U_i\}$ of $G$, affinoid atlases $\{V_{ij}\}$ of $F\times_G U_i$, such that the maps $V_{ij}\to U_i$ are smooth.
	In particular, the truncations $\trunc(V_{ij})\to \trunc(U_i)$ are smooth.
	Since $\{\trunc(U_i)\}$ constitute an atlas of $\trunc(G)$ and $\{\trunc(V_{ij})\}$ constitute an atlas of $\trunc(F)$, we deduce that the truncation $\trunc(f)$ is smooth.
	Let us now prove that the second condition is satisfied as well.
	We proceed by induction on $n$.
	Suppose first $n = -1$ and consider the lifting problem \eqref{eq:lifting_problem}.
	Set
	\[ F' \coloneqq X_d[\cF] \times_G F . \]
	Let $x' \colon X \to F'$ be the morphism induced by the universal property of the pullback.
	Then the lifting problem \eqref{eq:lifting_problem} is equivalent to the following one:
	\[ \begin{tikzcd}
	X \arrow{r}{x'} \arrow{d} & F' \arrow{d} \\
	X_d[\cF] \arrow{r}{\id} \arrow[dashed]{ur} & X_d[\cF] .
	\end{tikzcd} \]
	In other words, we can assume $G$, and hence $F$, to be $(-1)$-representable.
	Recall that, by definition, $X_d[\cF]$ is the pushout
	\[ \begin{tikzcd}
	X[\cF] \arrow{d}{d} \arrow{r}{d_0} & X \arrow{d} \\
	X \arrow{r} & X_d[\cF]
	\end{tikzcd} \]
	in the category $\dAn$.
	Since $F$ is $(-1)$-representable, to produce a solution $X_d[\cF] \to F$ of the lifting problem is equivalent to produce a path between the two morphisms
	\[ \begin{tikzcd}
	X[\cF] \arrow[shift left = 1]{r}{d} \arrow[shift right = 1]{r}[swap]{d_0} & X \arrow{r}{x} & F .
	\end{tikzcd} \]
	in the category $\dAn_{X // G}$.
		Observe that these two morphisms in $\dAn_{X // G}$ define two elements $\alpha, \beta \in \pi_0 \DerAn_{F/G}(X; \cF)$.
	In order to solve the original lifting problem, it is enough to find a path between $\alpha$ and $\beta$ in the space
	\[ \DerAn_{F/G}(X; \cF) \simeq \Map_{\Coh^+(X)}(x^* \anL_{F/G}, \cF) . \]
	It is enough to prove that
	\begin{equation} \label{eq:path_obstruction}
	\pi_0 \Map_{\Coh^+(X)}(x^* \anL_{F/G}, \cF) \simeq 0 .
	\end{equation}
	Let us first prove \cref{eq:path_obstruction} in the non-archimedean analytic case.
	By \cref{prop:characterization_smooth_morphisms}, $x^* \anL_{F/G}$ is perfect and in tor-amplitude 0.
	This implies that it is a retract of a free module of finite rank.
	In particular, $\pi_0 \Map_{\Coh^+(X)}(x^* \anL_{F/G}, \cF)$ is a retract of $\pi_0(\cF^n) \simeq 0$ for some non-negative integer $n$.
	This completes the proof of \cref{eq:path_obstruction} in the non-archimedean case.
	
	Now let us prove \cref{eq:path_obstruction} in the complex analytic case.
	Consider the internal Hom $\cHom(x^* \anL_{F/G},\allowbreak \cF)$ in $\Coh^+(X)$, and remark that
	\[ \Map_{\Coh^+(X)}(x^* \anL_{F/G}, \cF) \simeq \tau_{\ge 0} \Gamma(X, \cHom(x^* \anL_{F/G}, \cF)) . \]
		Since $X$ is Stein, Cartan's theorem B shows that it is enough to check that
	\[ \cHom(x^* \anL_{F/G}, \cF) \in \Coh^{\ge 1}(X) .  \]
	This condition is local and it can therefore be checked after shrinking $X$.
	Since $f$ is smooth, it follows from \cref{prop:characterization_smooth_morphisms} that $x^* \anL_{F/G}$ is perfect and in tor-amplitude $0$.
	Therefore, locally on $X$, we can express $x^* \anL_{F / G}$ as retract of a free module of finite rank.
	It follows that, locally on $X$, the sheaf $\cHom(x^* \anL_{F/G}, \cF)$ is a retract of $\cF^n$ for some nonnegative integer $n$.
	Since $\cF \in \Coh^{\ge 1}(X)$, this completes the proof of \cref{eq:path_obstruction}.
	
	We now assume that $n \ge 0$ and that the statement has already been proven for $m < n$.
	Base-changing to $X_d[\cF]$ we can assume once again that $G$ is representable and therefore that $F$ is $n$-geometric.
	In particular, $F$ is infinitesimally cartesian in virtue of \cref{prop:geometric_implies_inf_cartesian}.
	It will therefore be sufficient to prove that $\anL_{F / G}$ is perfect and in tor-amplitude $[0, n]$.
	This follows by induction on $n$, and the same proof of \cite[2.2.5.2]{HAG-II} applies.
	
	We now prove the converse. Assume that $\trunc(f)$ is smooth and that the lifting problem (\ref{item:infinitesimal_smooth}) always has at least one solution.
	By base change, we can assume that $G$ is itself representable and therefore that $F$ is $n$-geometric.
	Let $U \to F$ be a smooth atlas for $F$.
	Since $U \to F$ is smooth, the lifting problem \eqref{eq:lifting_problem} for this map has at least one solution.
	It follows that the composition $U \to F \to G$ has the same property.
	We are thus reduced to the case where both $F = X$ and $G = Y$ are representable.
	In virtue of \cref{prop:characterization_smooth_morphisms}(\ref{item:differential_smooth}), it will be enough to show that $\anL_{F / G}$ is perfect and concentrated in tor-amplitude $0$.
	Notice that these conditions can be checked locally on $X$.
	
	The lifting condition implies for any $\cF \in \Coh^{\ge 1}(X)$ we have
	\[ \pi_0 \Map_{\Coh^+(X)}( \anL_{X / Y}, \cF) = 0 . \]
	Using \cref{cor:finiteness_cotangent_complex}, up to shrinking $X$ in the complex analytic case, we can choose a map $\phi \colon \cO_X^n \to \anL_{X / Y}$ which is surjective on $\pi_0$.
	Let $K \coloneqq \fib(\phi)$. We therefore obtain an exact sequence
	\[ \pi_0 \Map_{\Coh^+(X)}( \anL_{X / Y}, \cO_X^n) \to \pi_0 \Map_{\Coh^+(X)}( \anL_{X / Y}, \anL_{X / Y}) \to \pi_0 \Map_{\Coh^+(X)}(\anL_{X / Y}, K[1]). \]
	Since $\pi_0 \Map_{\Coh^+(X)}(\anL_{X / Y}, K[1]) = 0$, we conclude that $\anL_{X / Y}$ is a retraction of $\cO_X^n$, and as a consequence it is perfect and in tor-amplitude $0$.
\end{proof}

We complete the proof of the implication (\ref{item:F_is_representable}) $\Rightarrow$ (\ref{item:F_satisfies_rep_conditions}) in \cref{thm:representability} by the following lemma, analogous to \cite[C.0.10]{HAG-II}.

\begin{lem} \label{lem:representable_implies_convergent}
	Let $f \colon F \to G$ be an $n$-representable morphism in $\St(\dAfdk,\tauet)$.
	Then for any $X \in \dAfd_k$, the square
	\[ \begin{tikzcd}
	F(X) \arrow{r} \arrow{d} & \lim_m F(\mathrm t_{\le m} X) \arrow{d} \\
	G(X) \arrow{r} & \lim_m G(\mathrm t_{\le m} X)
	\end{tikzcd} \]
	is a pullback.
\end{lem}

\begin{proof}
	We start by remarking that in the special case where $G = *$ and $f$ is $(-1)$-representable, the statement follows directly from the fact that
	\[ X \simeq \colim_m \mathrm t_{\le m} X \]
	in $\dAnk$.
	
	Let us now turn to the general case.
	We want to prove that the canonical map
	\[ F(X) \to G(X) \times_{\lim G(\mathrm t_{\le m} X)} \lim F(\mathrm t_{\le m} X) \]
	is an equivalence.
	For this, it is enough to prove that its fibers are contractible.
	Fix a point $x \in G(X) \times_{\lim G(\mathrm t_{\le m}X)} \lim F(\mathrm t_{\le m} X)$.
	The projection of $x$ in $G(X)$ determines a map $f \colon X \to G$.
	We can then replace $G$ by $X$ and $F$ by the fiber product $X \times_G F$.
	At this point $G$ is $(-1)$-representable and therefore the map
	\[ G(X) \to \lim_m G(\mathrm t_{\le m} X) \]
	is an equivalence.
	We are therefore reduced to prove that the map
	\[ F(X) \to \lim_m F(\mathrm t_{\le m} X) \]
	is an equivalence.
	In other words, we can assume $G = *$ and $F$ to be $n$-geometric.
	
	We proceed by induction on the geometric level $n$.
	When $n = -1$, we already proved that the statement is true.
	Suppose $n \ge 0$ and let $u \colon U \to F$ be an $n$-atlas.
	We will prove that the fibers of the morphism
	\[ F(X) \to \lim_m F(\mathrm t_{\le m} X) \]
	are contractible.
	Let $x \in \lim_m F(\mathrm t_{\le m} X)$ be a point and let $x_m \colon \mathrm t_{\le m} X \to F$ be the morphism classified by the projection of $x$ in $F(\mathrm t_{\le m} X)$.
	Since $F$ is a sheaf and limits commute with limits, we see that this statement is local on $X$.
	We can therefore suppose that $x_0$ factors as
	\[ \begin{tikzcd}
		{} & U \arrow{d}{u} \\
		\trunc(X) \arrow{r}{x_0} \arrow{ur}{y_0} & F .
	\end{tikzcd} \]
	We claim that there exists a point $y \in \lim_m U( \mathrm t_{\le m} X)$ whose image in $\lim_m F( \mathrm t_{\le m} X)$ is $x$.
	In order to see this, we construct a compatible sequence of maps $y_m \colon \mathrm t_{\le m} X \to U$ by induction on $m$.
	We already constructed $m = 0$.
	Now, observe that since $u$ is smooth and since the morphisms $\mathrm t_{\le n} X \hookrightarrow \mathrm t_{\le n+1} X$ are analytic square-zero extensions by \cref{cor:Postnikov_tower_analytic_square_zero}, \cref{prop:infinitesimal_characterization_smooth} implies that the lifting problem
	\[ \begin{tikzcd}
		\mathrm t_{\le m} X \arrow{d} \arrow{r}{y_m} & U \arrow{d}{u} \\
		\mathrm t_{\le m+1} X \arrow[dashed]{ur}{y_{m+1}} \arrow{r}{x_{m+1}} & X
	\end{tikzcd} \]
	admits at least one solution.
	This completes the proof of the claim.
	We now consider the diagram
	\[ \begin{tikzcd}
		U(X) \arrow{r} \arrow{d} & \lim_m U( \mathrm t_{\le m} X) \arrow{d} \\
		F(X) \arrow{r} & \lim_m F( \mathrm t_{\le m} X ) .
	\end{tikzcd} \]
	Since $u \colon U \to F$ is $(n-1)$-representable, the induction hypothesis implies that the above diagram is a pullback square.
	We can therefore identify the fiber at $y \in \lim_m U( \mathrm t_{\le m} X )$ of the top morphism with the fiber at $x \in \lim_m F(\mathrm t_{\le m} X)$ of the bottom morphism.
	On the other hand, since $U$ is representable, we see that the top morphism is an equivalence.
	The proof is therefore complete.
\end{proof}

\subsection{Lifting atlases} \label{sec:lifting_atlases}

In this subsection, we prove the implication (2)$\Rightarrow$(1) of \cref{thm:representability}.

\begin{lem} \label{lem:comparing_nonconnective_parts}
	Let $\cC$ be a stable $\infty$-category equipped with a t-structure.
	Let $f \colon M \to N$ be a morphism between eventually connective objects.
	Let $m$ be an integer.
	If for every $P \in \cC^\heartsuit$ the canonical map
	\[ \Map_{\cC}(N, P[m]) \to \Map_{\cC}(M, P[m]) \]
	is an equivalence, then $\tau_{\le m} M \to \tau_{\le m} N$ is an equivalence as well.
\end{lem}

\begin{proof}
	Up to replace $M$ and $N$ by $M[-m]$ and $N[-m]$, we can suppose $m = 0$.
	Moreover, since $\Map_{\cC}(\tau_{\ge 1} M, P) \simeq \Map_{\cC}(\tau_{\ge 1} N, P) \simeq \{*\}$ for every $P \in \cC^\heartsuit$, we can further replace $M$ and $N$ by $\tau_{\le 0} M$ and $\tau_{\le 0} N$, respectively.
	In other words, we can suppose that $\pi_i(M) = \pi_i(N) = 0$ for every $i > 0$.
	
	Let $n$ be the largest integer such that at least one among $\pi_{-n}(M)$ and $\pi_{-n}(N)$ is not zero.
	We proceed by induction on $n$.
	If $n = 0$, then $M, N \in \cC^\heartsuit$ and therefore the statement follows from the Yoneda lemma.
	Let now $n > 0$. Choosing $P = \pi_n(M)$, we obtain an element $\gamma \in \pi_n \Map_{\cC}(M,P)$.
	The corresponding element in $\pi_n \Map_{\cC}(N, P)$ can be represented by a morphism $g \colon N \to P[-n]$.
	Inspection reveals that $\pi_n(g) \colon \pi_n(N) \to \pi_n(M)$ is an inverse for $\pi_n(f)$.
		We now consider the morphism of fiber sequences
	\[ \begin{tikzcd}
	\tau_{\ge - n + 1}M \arrow{r} \arrow{d} & M \arrow{r} \arrow{d} & \pi_{-n}(M) \arrow{d} \\
	\tau_{\ge - n + 1} N \arrow{r} & N \arrow{r} & \pi_{-n}(N) .
	\end{tikzcd} \]
	Fix $P \in \cC^\heartsuit$.
	Applying the functor $\Map_{\cC}(-, P)$ and then taking the long exact sequence of homotopy groups, we conclude that
	\[ \Map_{\cC}(\tau_{\ge -n+1} N, P) \to \Map_{\cC}(\tau_{\ge -n+1} M, P) \]
	is an equivalence for every choice of $P$.
	We can therefore invoke the induction hypothesis to deduce that $\tau_{\ge -n+1}(f)$ is an equivalence.
	As we already argued that $\pi_n(f)$ is an equivalence, we conclude that the same goes for $f$ itself, thus completing the proof.
\end{proof}

\begin{lem} \label{lem:cotangent_complex_truncation}
	Let $F \in \St(\dAfdk,\tauet)$ be a stack satisfying the conditions in \cref{thm:representability}(\ref{item:F_satisfies_rep_conditions}).
	Let $j \colon \trunc(F) \to F$ be the canonical morphism.
	Then $\anL_{\trunc(F) / F}$ belongs to $\Coh^{\ge 2}(\trunc(F))$.
\end{lem}

\begin{proof}
	We follow closely the proof of \cite[Theorem 3.1.2]{DAG-XIV}.
	Let $\eta \colon U \to \trunc(F)$ be a smooth morphism from an affinoid (resp.\ Stein) space $U$.
	For every discrete coherent sheaf $\cF$ on $U$, the canonical map
	\begin{equation} \label{eq:2-connective}
		\Map_{\Coh^+(U)}(\eta^* \anL_{\trunc(F)}, \cF) \to \Map_{\Coh^+(U)}(\eta^* j^* \anL_F, \cF)
	\end{equation}
	is obtained by passing to vertical fibers in the commutative diagram
	\[ \begin{tikzcd}
	\trunc(F)(U[\cF]) \arrow{r} \arrow{d} & F(U[\cF]) \arrow{d} \\
	\trunc(F)(U) \arrow{r} & F(U) .
	\end{tikzcd} \]
	Since $\cF$ is discrete, $U[\cF]$ is an underived affinoid (resp.\ Stein) space.
	As consequence, the horizontal morphisms are equivalences.
	It follows that the same goes for the map \eqref{eq:2-connective}.
	Therefore, \cref{lem:comparing_nonconnective_parts} shows that $\tau_{\le 0} \eta^* j^* \anL_F \to \tau_{\le 0} \eta^* \anL_{\trunc(F)}$ is an equivalence.
	We conclude that $\anL_{\trunc(F) / F}$ is $1$-connective.
	
	We now prove that it is also $2$-connective. We have an exact sequence
	\[ \pi_{1}( j^* \anL_F) \to \pi_{1}(\anL_{\trunc(F)}) \to \pi_{1}( \anL_{\trunc(F) / F, j}) \to 0 . \]
	Let $\cF \coloneqq \pi_{1}(\anL_{\trunc(F) / F, j})$.
	If $\cF \ne 0$, then we obtain a non-zero map
	\[ \gamma \colon \anL_{\trunc(F)} \to \anL_{\trunc(F) / F} \to \cF[1] , \]
	whose restriction to $j^* \anL_F$ vanishes.
	Choose a smooth morphism $\eta \colon U \to \trunc(F)$ such that $\eta^* \cF \ne 0$.
	Then $\gamma$ determines a non-zero morphism $\eta^* \mathbb L_{\trunc(F)} \to \eta^* \cF[1]$.
	Since there is a fiber sequence
	\[ \anL_{U / \trunc(F)}[-1] \to \eta^* \anL_{\trunc(F)} \to \anL_U \]
	and since $\anL_{U / \trunc(F)}$ is perfect and in tor-amplitude 0, we conclude that the composition
	\[ \anL_{U / \trunc(F)}[-1] \to \eta^* \anL_{\trunc(F)} \to \eta^* \cF[1] \]
	vanishes.
	In other words, we obtain a non-zero analytic derivation $d \colon \anL_U \to \eta^* \cF[1]$.
	Let $U_d[\eta^* \cF]$ be the associated square-zero extension.
	We now consider the following diagram:
	\[ \begin{tikzcd}
		U[\eta^* \cF[1]]  \arrow{r}{d} \arrow{d} & U \arrow{d}{i} \arrow{r}{\eta} & \trunc(F) \arrow{d}{j} \\
		U \arrow{r} & U_d[\eta^* \cF] \arrow[dashed]{r}{\beta} \arrow[dashed]{ur}{\alpha} & F .
	\end{tikzcd} \]
	The left square is a pushout, so to produce the lifting $\alpha$ (resp.\ $\beta$) in the category $\St(\dAfdk,\tauet)_{U /}$ is equivalent to produce a path in
	\[ \Map_{\Coh^+(U)}( \eta^* \anL_{\trunc(F)}, \eta^* \cF[1] ) \qquad ( \textrm{resp.\ } \Map_{\Coh^+(U)}(\eta^* j^* \anL_{F}, \eta^* \cF[1] )) \]
	between $\eta \circ d$ and $\eta \circ d_0$ (resp.\ $j \circ \eta \circ d$ and $j \circ \eta \circ d_0$).
	It follows from \cref{prop:analytic_square_zero_extensions_are_analytic_spaces} that $U_d[\eta^* \cF]$ is an underived affinoid (resp.\ Stein) space.
	In particular, the canonical map
	\[ \trunc(F)(U_d[\eta^* \cF]) \to F(U_d[\eta^* \cF]) \]
	is a homotopy equivalence.
	As a consequence, the existence of $\alpha$ is equivalent to the existence of $\beta$.
	Nevertheless:
	\begin{enumerate}
		\item the map $\alpha$ cannot exist because $\eta \circ d_0$ is equivalent to the zero map $\eta^* \anL_{\trunc(F)} \to \eta^* \cF[1]$, while $\eta \circ d$ is non-zero by construction;
		\item the map $\beta$ exists because both $j \circ \eta \circ d_0$ and $j \circ \eta \circ d_0$ correspond to the zero map $\eta^* j^* \anL_F \to \eta^* \cF[1]$.
		This is because the composition $\eta^* j^* \anL_F \to \eta^* \anL_{\trunc(F)} \to \eta^* \cF[1]$ is zero.
	\end{enumerate}
	This is a contradiction, and the lemma is therefore proved.
\end{proof}

\begin{lem} \label{lem:lifting_etale_morphisms}
	Let $F \in \St(\dAfdk, \tauet)$ be a stack satisfying the conditions in \cref{thm:representability}(\ref{item:F_satisfies_rep_conditions}).
	Then for any $U_0\in\Afd_k$ and any \'etale morphism $u_0 \colon U_0 \to \trunc(F)$, there is $U\in\dAfd_k$, a morphism $u \colon U \to F$ satisfying $\anL_{U/F} \simeq 0$ and a pullback square
	\[ \begin{tikzcd}
	U_0 \arrow{r} \arrow{d} & \trunc(F) \arrow{d} \\
	U \arrow{r} & F.
	\end{tikzcd} \]
\end{lem}

\begin{proof}
	We follow closely the proof of \cite[Lemma C.0.11]{HAG-II}.
	We will construct by induction a sequence of derived affinoid (resp.\ Stein) spaces
	\[ U_0 \to U_1 \to \cdots \to U_n \to \cdots \to  F \]
	satisfying the following properties:
	\begin{enumerate}
		\item \label{item:U_n_n_truncated} $U_n$ is $n$-truncated;
		\item \label{item:U_n_truncation_U_n+1} the morphism $U_n \to U_{n+1}$ induces an equivalence on $n$-th truncations;
		\item \label{item:U_n_is_n_etale} the morphisms $u_n \colon U_n \to F$ are such that $\pi_{i}(\anL_{U_n / F}) \simeq 0$ for every $i \le n + 1$.
	\end{enumerate}
	Assume that the sequence has already been constructed.
	Then all the derived affinoid (resp.\ Stein) spaces $U_n$ share the same underlying $\infty$-topos $\cU$.
	Moreover, the canonical morphism $\cO_{U_n} \to \pi_0(\cO_{U_n}) \simeq \cO_{U_0}$ are local.
	It follows that
	\[ \cO_U \coloneqq \lim_n \cO_{U_n} \in \AnRing_k(\cU)_{/\cO_{U_0}} \]
	is a $\cTank$-structure satisfying $\tau_{\le n}(\cO_U) \simeq \cO_{U_n}$.
	In particular, $U \coloneqq (\cU, \cO_U)$ is a derived affinoid (resp.\ Stein) space.
	Since $F$ is convergent, we obtain a canonical morphism $u \colon U \to F$.
	Let us check that $\anL_{U/F,u} \simeq 0$.
	Fix $\cF \in \Coh^{\ge 0}(U)$.
	We have
	\[ \DerAn_F(U,\cF) \simeq \lim \DerAn_F(U_n, \tau_{\le n}(\cF)) \simeq \lim \Map_{\Coh^+(U_n)}(\anL_{U_n / F}, \tau_{\le n} \cF) \simeq 0 . \]
	
	Finally, the map $U \times_F \trunc(F) \to U$ enjoys the following universal property: for every underived $X$ the map
	\[ \Map_{\St(\dAfd, \tauet)}(X, U \times_F \trunc(F)) \to \Map_{\St(\dAfd, \tauet)}(X, U) \]
	is an equivalence.
	This allows to identify $U \times_F \trunc(F)$ with $\trunc(U) \simeq U_0$.
	
	We are left to construct the sequence $U_n$. We proceed by induction.
	If $n = 0$, we only have to prove that $\anL_{U_0 / F}$ is $2$-connective.
	Let $j \colon \trunc(F) \to F$ be the canonical map. Then we have a fiber sequence
	\[ u_0^* \anL_{\trunc(F) / F} \to \anL_{U_0 / F} \to \anL_{U_0 / \trunc(F)} .  \]
	Since $u_0$ is \'etale, $\anL_{U_0 / \trunc(F), u_0} \simeq 0$.
	Therefore, the statement follows from the fact that $\anL_{\trunc(F) / F, j}$ is $2$-connective, which is the content of \cref{lem:cotangent_complex_truncation}.
	
	Assume now that $U_n$ has been constructed. Let $u_n \colon U_n \to F$ be the given morphism.
	Consider the composite map
	\[ d \colon \anL_{U_n} \to \anL_{U_n / F} \to \tau_{\le n+2} \anL_{U_n / F} \simeq \pi_{n+2}( \anL_{U_n / F} )[n + 2] . \]
	This is an analytic derivation and thus it defines an analytic square-zero extension of $U_n$ by $\pi_{n+2}( \anL_{U_n / F} )[n + 2]$.
	Let us denote it by $U_{n+1}$. It follows from \cref{prop:analytic_square_zero_extensions_are_analytic_spaces} that $U_{n+1}$ is a derived affinoid (resp.\ Stein) space.
	Moreover, since $F$ is infinitesimally cartesian, we see that there is a canonical map $u_{n+1} \colon U_{n+1} \to F$.
		
	Then conditions (\ref{item:U_n_n_truncated}) and (\ref{item:U_n_truncation_U_n+1}) are met by construction.
	Let us prove that condition (\ref{item:U_n_is_n_etale}) is satisfied as well.
	Let $j_n \colon U_n \to U_{n+1}$ denote the canonical morphim.
	Since $\mathrm{t}_{\le n}(j_n)$ is an equivalence, it will be sufficient to show that $j_n^* \anL_{U_{n+1}/F}$ is $(n+2)$-connective.
	This fits into a fiber sequence
	\[ j_n^* \anL_{U_{n+1} / F} \to \anL_{U_n / F} \xrightarrow{\phi} \anL_{U_n / U_{n+1}} . \]
	Since $j_n$ is $n$-connective and $U_n$ is $n$-truncated, $\mathrm{cofib}(j_n)$ is $(n+1)$-connective.
	It follows from \cref{cor:connectivity_estimates} that $\anL_{U_n / U_{n+1}}$ is $(n+1)$-connective.
	Moreover, since $n \ge 1$, we can combine \cref{cor:cotangent_complex_closed_immersion} with \cite[2.2.2.8]{HAG-II} to conclude that
	\[ \pi_{n+2}( \anL_{U_n / U_{n+1}} ) \simeq \pi_{n+2}( \anL_{U_n / F} ) . \]
	The proof is therefore complete.
\end{proof}

We are finally ready to complete the proof of \cref{thm:representability}.

\begin{proof}[Proof of \cref{thm:representability}.]
	The implication (\ref{item:F_is_representable}) $\Rightarrow$ (\ref{item:F_satisfies_rep_conditions}) follows from \cref{prop:geometric_implies_inf_cartesian} and \cref{lem:representable_implies_convergent}.
	Let now $F \in \St(\dAfd, \tauet)$ be a stack satisfying \cref{thm:representability} Condition (\ref{item:F_satisfies_rep_conditions}).
	We will prove by induction on $n$ that $F$ is $n$-geometric.
	
	If $n = -1$, then \cref{lem:lifting_etale_morphisms} allows to lift the identity of $\trunc(F)$ to a morphism $U \to F$, where $U \in \dAfd_k$ and $\anL_{U/F} \simeq 0$.
	Let $X \in \dAfd_k$.
	By \cref{cor:Postnikov_tower_analytic_square_zero} and by induction on $m$, we see that the canonical map
	\[ U( \mathrm t_{\le m} X ) \to F( \mathrm t_{\le m} X) \]
	is an equivalence for every $m$.
	Since $F$ and $U$ are convergent, we deduce that $U \simeq F$, so $F$ is representable.
	
	Let now $n \ge 0$. It follows from the induction hypothesis that the diagonal of $F$ is $(n-1)$-representable.
	We are therefore left to prove that $F$ admits an atlas.
	Let $u \colon U_0 \to \trunc(F)$ be a smooth atlas and let $j \colon \trunc(F) \hookrightarrow F$ be the natural inclusion.
	We will construct a sequence of morphisms
	\[ \begin{tikzcd}
		U_0 \arrow[hook]{r}{j_0} \arrow[bend right = 5]{drrrrr}[swap, near start]{u_0} & U_1 \arrow[hook]{r}{j_1} \arrow[bend right = 2]{drrrr}{u_1} & \cdots \arrow[hook]{r}{j_{m-1}} & U_m \arrow{drr}[bend right = 1.5]{u_m} \arrow[hook]{r}{j_m} & \cdots \\
		{} & & & & & F
	\end{tikzcd} \]
	satisfying the following properties:
	\begin{enumerate}
		\item $U_m$ is $m$-truncated;
		\item $U_m \to U_{m+1}$ induces an equivalence on $m$-truncations;
		\item $\anL_{U_m / F}$ is flat to order $m+1$ (cf.\ \cref{def:n_flat}).
	\end{enumerate}
	The construction is carried out by induction on $m$.
	When $m = 0$, we set
	\[ u_0 \coloneqq j \circ u . \]
	It suffices to check that $\anL_{U_0/ F}$ is flat to order $1$.
	Consider the fiber sequence
	\[ u^* \anL_{\trunc(F) / F} \to \anL_{U_0 / F} \to \anL_{U_0 / \trunc(F)} . \]
	\cref{lem:cotangent_complex_truncation} guarantees that $\anL_{\trunc(F) / F}$ is 2-connective and therefore
	\[ u^* \anL_{\trunc(F) / F} \in \Coh^{\ge 2}(U_0) . \]
	In particular, it follows that the natural morphism
	\[ \tau_{\le 1} \anL_{U_0 / F} \to \tau_{\le 1} \anL_{U_0 / \trunc(F)} \]
	is an equivalence.
	Since $u \colon U_0 \to \trunc(F)$ is smooth and $U_0$ is discrete, we conclude that $\tau_{\le 1} \anL_{U_0 / \trunc(F)} \simeq \anL_{U_0 / \trunc(F)}$.
	In particular, $\tau_{\le 1} \anL_{U_0 / F}$ is perfect and in tor-amplitude $0$.
	\cref{prop:n_flat_sorite}(\ref{item:n_flat_truncation}) implies that $\anL_{U_0 / F}$ is flat to order $1$.

	Assume now that $u_m \colon U_m \to F$ has been constructed.
	Since $U_m$ is $m$-truncated and $\anL_{U_m / F}$ is flat to order $m+1$, it follows from \cref{prop:n_flat_implies_flat_truncation} that $\tau_{\le m+1} \anL_{U_m / F}$ is flat.
	Up to shrinking $U_m$ in the complex case, we can assume that $\tau_{\le m+1} \anL_{U_m / F}$ is a retract of a free module.
	In particular, it follows that
	\[ \pi_0 \Map_{\Coh^+(U_m)}(\tau_{\le m+1} \anL_{U_m / F}, \cF) = 0 \]
	for every $\cF \in \Coh^{\ge 1}(U_m)$.
	Taking $\cF = \tau_{\le m+2} \anL_{U_m/F}[1]$, we conclude that the natural map
	\[ \tau_{\le m+1} \anL_{U_m / F}[-1] \to \anL_{U_m / F} \]
	is homotopic to zero.
		Consider now the following diagram
	\[ \begin{tikzcd}
		\tau_{\le m+1} \anL_{U_m}[-1] \arrow{r} \arrow{d} & \tau_{\le m+1} \anL_{U_m / F}[-1] \arrow{d}{\simeq 0} \\
		\tau_{\ge m+2} \anL_{U_m} \arrow{r} \arrow{d} & \tau_{\ge m+2} \anL_{U_m / F} \arrow{d} \\
		\anL_{U_m} \arrow{r} \arrow[dashed]{ur}{\varphi} & \anL_{U_m / F} .
	\end{tikzcd} \]
	The universal property of the cofiber implies the existence of the dotted arrow.
	
	Consider the composition
	\[ d \colon \anL_{U_m} \xrightarrow{\varphi} \tau_{\ge m+2} \anL_{U_m / F} \to \pi_{m+2}(\anL_{U_m / F}) . \]
	This map corresponds to an analytic derivation.
	We let $U_{m+1}$ denote the associated analytic square-zero extension.
	By construction, $U_{m+1}$ is $(m+1)$-truncated and the canonical map
	\[ j_m \colon U_m \to U_{m+1} \]
	induces an equivalence on the $m$-th truncation.
	Furthermore, since $F$ is infinitesimally cartesian and since the composition
	\[ u_m^* \anL_F \to \anL_{U_m} \to \anL_{U_m / F} \]
	is homotopic to zero, there is a map $u_{m+1} \colon U_{m+1} \to F$ fitting in the commutative triangle
	\[ \begin{tikzcd}
		U_m \arrow{d}[swap]{j_m} \arrow{r}{u_m} & F \\
		U_{m+1} \arrow{ur}[swap]{u_{m+1}} & \phantom{F}
	\end{tikzcd} \]
	Conditions (1) and (2) are satisfied by construction.
	We are thus left to check that $\anL_{U_{m+1} / F}$ is flat to order $m+2$.
	Using \cref{prop:n_flat_sorite}(\ref{item:n_flat_testing_on_truncations}), it is enough to check that $j_m^* \anL_{U_{m+1} / F}$ is flat to order $m+2$.
	Consider the transitivity fiber sequence
	\[ \begin{tikzcd}
		j_m^* \anL_{U_{m+1} / F} \arrow{r} & \anL_{U_m / F} \arrow{r}{\phi} & \anL_{U_m / U_{m+1}} .
	\end{tikzcd} \]
	By the induction hypothesis, $\anL_{U_m / F}$ is flat to order $m+1$.
	Moreover, $\anL_{U_m / U_{m+1}}$ is $(m+2)$-connective.
	It follows that $j_m^* \anL_{U_{m+1} / F}$ is flat to order $m+1$.
	Since $U_m$ is $m$-truncated, \cref{cor:n_flat_increasing} shows that $j_m^* \anL_{U_{m+1} / F}$ is flat to order $m+2$ if and only if
	\[ \pi_{m+2}( \anL_{U_{m+1} / F}) = 0 . \]
	To prove the latter, it is enough to show that the map $\phi$ induces an isomorphism on $\pi_{m+2}$ and a surjection on $\pi_{m+3}$.
	Set
	\[ \cF \coloneqq \pi_{m+2}(\anL_{U_m / F})[m+2] . \]
	Combining \cref{cor:cotangent_complex_closed_immersion} and \cite[Lemma 1.4.3.7]{HAG-II}, we see that $\anL_{U_m / U_{m+1}}$ can be computed as the pushout
	\[ \begin{tikzcd}
		\cF \otimes_{\cO_{U_{m+1}}} \cF \arrow{r}{\mu} \arrow{d} & \cF \arrow{d} \\
		0 \arrow{r} & \anL_{U_m / U_{m+1}} ,
	\end{tikzcd} \]
	where $\mu$ is the multiplication map induced by $\cO_{U_{m+1}}$.
	Using \cite[7.4.1.14]{Lurie_Higher_algebra}, we see that $\mu$ is nullhomotopic.
	As a consequence, we obtain
	\[ \anL_{U_m / U_{m+1}} \simeq \cF \oplus ( \cF \otimes_{\cO_{U_{m+1}}} \cF [1] ) . \]
	Since $m > 0$, we have:
	\begin{gather*}
		\pi_{m+2}( \anL_{U_m / U_{m+1}} ) \simeq \cF = \pi_{m+2}( \anL_{U_m / F} ) \\
		\pi_{m+3}(\anL_{U_m / U_{m+1}}) \simeq 0.
	\end{gather*}
	It follows that $\phi$ has the required properties.
	In turn, this completes the construction of the sequence of maps $u_m \colon U_m \to F$.
	
	The same argument given in \cref{lem:lifting_etale_morphisms} shows that the colimit of the diagram
	\[ \begin{tikzcd}
		U_0 \arrow[hook]{r}{j_0} & U_1 \arrow[hook]{r}{j_1} & \cdots \arrow[hook]{r} & U_m \arrow[hook]{r}{j_m} & \cdots
	\end{tikzcd} \]
	exists in $\dAfd_k$.
	We denote it by $\tU$.
	Since $F$ is convergent, we can assemble the maps $u_m \colon U_m \to F$ into a canonical map
	\[ \tu \colon \tU \to F . \]
	Let $i_m \colon U_m \to \tU$ be induced map.
	Consider the fiber sequence
	\[ i_m^* \anL_{\tU / F} \to \anL_{U_m / F} \to \anL_{U_m / \tU} . \]
	Since $\anL_{U_m / F}$ is flat to order $m+1$ by construction and $\anL_{U_m / U}$ is $(m+2)$-connective, it follows that $i_m^* \anL_{\tU / F}$ is flat to order $m+1$.
	Using \cref{prop:n_flat_sorite}(\ref{item:n_flat_truncation}), we conclude that $\anL_{\tU/F}$ is flat to order $m+1$.
	Since this holds for every $m$, we see that $\anL_{\tU / F}$ has tor-amplitude $0$.
	Since it is almost perfect, we conclude that $\anL_{\tU / F}$ is perfect and in degree $0$.
	Using the lifting criterion of \cref{prop:infinitesimal_characterization_smooth}, we conclude that $\tu$ is smooth.
	The proof of \cref{thm:representability} is thus complete.
\end{proof}

\section{Appendices}

\subsection{Modules over a simplicial commutative ring} \label{sec:modules}

Let $\CRing$ denote the \infcat of simplicial commutative rings.
Let $A\in\CRing$ and let $X\coloneqq\Spec (A)$ be the associated derived scheme.
We denote by $\mathrm{dSch}_{/X}$ the \infcat of derived schemes over $X$.
Let $\cT_A$ be the discrete pregeometry whose underlying \infcat is the full subcategory of $\mathrm{dSch}_{/X}$ spanned by the derived schemes $\bbA^n_X \coloneqq \Spec(\Sym_A(A^n))$ for all $n\ge 0$.
Moreover, let us define the discrete pregeometry $\cT_A^{\mathrm{aug}} \coloneqq (\cT_A)_{X/}$, whose underlying \infcat is the full subcategory of $(\mathrm{dSch}_k)_{X/ /X}$ spanned by objects $X \to Y \to X$ with $Y\in\cT_A$.

\begin{prop}
	We have the following equivalences of \infcats:
	\begin{enumerate}
		\item $\CRing_A \simeq \Fun^\times(\cT_A, \cS)$;
		\item $\CRing_{A//A} \simeq \Fun^\times(\cT_A^{\mathrm{aug}}, \cS)$.
	\end{enumerate}
\end{prop}

\begin{proof}
	The first equivalence is the content of \cite[Definition 4.1.1 and Remark 4.1.2]{DAG-V}.
	
	Let us now prove the second one.
	Observe that there is a forgetful functor $\varphi \colon \cT_A^{\mathrm{aug}} \to \cT_A$ that commutes with products.
	In particular, composition with $\varphi$ induces a well-defined functor
	\[ \Phi \colon \Fun^\times(\cT_A, \cS) \to \Fun^\times(\cT_A^{\mathrm{aug}}, \cS) . \]
	This functor commutes with limits and with sifted colimits. In particular, it has a left adjoint, denoted
	\[ \Psi \colon \Fun^\times(\cT_A^{\mathrm{aug}}, \cS) \to \Fun^\times(\cT_A, \cS) \simeq \CRing_A . \]
	Let $\cO_A \coloneqq \Map_{\cT_A^{\mathrm{aug}}}(\Spec(A), -) \in \Fun^\times(\cT_A^{\mathrm{aug}})$.
	Since $X = \Spec(A)$ is a final object in $\cT_A^{\mathrm{aug}}$, it follows that $\cO_A$ is an initial object in $\Fun^\times(\cT_A^{\mathrm{aug}})$.
	In particular, $\Psi(\cO_A)$ is an initial object in $\Fun^\times(\cT_A, \cS) \simeq \CRing_A$.
	In other words, $\Psi(\cO_A) \simeq A$.
	On the other hand, $X$ is also an initial object in $\cT_A^{\mathrm{aug}}$.
	Thus, $\cO_A$ is also a final object in $\Fun^\times(\cT_A^{\mathrm{aug}}, \cS)$.
	It follows that $\Psi$ factors through
	\[ F \colon \Fun(\cT_A^{\mathrm{aug}}, \cS) \to \CRing_{A//A} , \]
	in such a way that the diagram
	\[ \begin{tikzcd}
		\CRing_A & \Fun^\times(\cT_A, \cS) \arrow{l}[swap]{\sim} \\
		\CRing_{A//A} \arrow{u} & \Fun^\times(\cT_A^{\mathrm{aug}}, \cS) \arrow{u}[swap]{\Psi} \arrow{l}[swap]{F}
	\end{tikzcd} \]
	commutes.
	
	The functor $F$ admits a right adjoint $G$ that can be constructed as follows.
	Let $(B,f) \in \CRing_{A//A}$, where $B$ is an $A$-algebra and $f \colon B \to A$ is the augmentation.
	We can review $B$ as an object in $\Fun^\times(\cT_A, \cS)$. Applying $\Phi$ we obtain a product preserving functor $\Phi(B)$ equipped with a map to $\Phi(A) \simeq \Phi(\Psi(\cO_A))$.
	We can thus form the pullback
	\begin{equation} \label{eq:mapping_comma}
		\begin{tikzcd}
			G(B) \arrow{r} \arrow{d} & \Phi(B) \arrow{d} \\ \cO_A \arrow{r} & \Phi(\Psi(\cO_A)) .
		\end{tikzcd}
	\end{equation}
	This construction shows immediately that $G$ is a right adjoint to $F$.
	Let us now remark that for $B \in \cT_A^{\mathrm{aug}} \subset \CRing_{A//A}$, we can canonically identify $G(B)$ with the functor
	\[ \cO_B \colon \cT_A^{\mathrm{aug}} \to \cS \]
	defined by
	\[ \cO_B ( \mathbb A^n_X ) = \Map_{X//X}( \Spec(B), \mathbb A^n_X ) . \]
	Indeed, we remark that evaluating the diagram of natural transformations \eqref{eq:mapping_comma} on $f \colon X \to \bbA^n_X$, we get the pullback diagram
	\[ \begin{tikzcd}
		G(B)(X \xrightarrow{f} \bbA^n_X) \arrow{r} \arrow{d} & \Map_{/X}(\Spec(B), \bbA^n_X) \arrow{d} \\
		\{*\} \arrow{r}{f} & \Map_{/X}(X, \bbA^n_X) .
	\end{tikzcd} \]
	In particular, we obtain a canonical identification
	\[ G(B)(X \xrightarrow{f} \bbA^n_X) \simeq \Map_{X//X}(\Spec(B), \bbA^n_X) . \]
	
	We now remark that both $F$ and $G$ commute with sifted colimits.
	By the statement (1) and \cite[5.5.8.10]{HTT}, it is enough to check that for every $f \colon X \to \bbA^n_X$, the canonical maps
	\[ F(G(\bbA^n,f)) \to (\bbA^n_X, f) , \qquad \cO_f \to G(F(\cO_f)) \]
	are equivalences.
	Observe that the functor
	\[ \Psi \colon \Fun^\times(\cT_A^{\mathrm{aug}}, \cS) \to \Fun^\times(\cT_A, \cS) \]
	can be factored as
	\[ \Fun^\times(\cT_A^{\mathrm{aug}}, \cS) \hookrightarrow \Fun(\cT_A^{\mathrm{aug}}, \cS) \xrightarrow{\mathrm{Lan}_\varphi} \Fun(\cT_A, \cS) \xrightarrow{\pi} \Fun^\times(\cT_A, \cS) . \]
	Now, observe that $\mathrm{Lan}_\varphi(\cO_f) = \Map_{/X}(\bbA^n_X, -)$.
		In particular, $\mathrm{Lan}_{\varphi}(\cO_f)$ is still a product preserving functor.
	As a consequence,
	\[ \Psi(\cO_f) = \pi( \mathrm{Lan}_\varphi(\cO_f) ) \simeq \mathrm{Lan}_\varphi(\cO_f) . \]
	In particular, we obtain
	\[ F(\cO_f) \simeq (\mathbb A^n_X, f) . \]
	The above considerations on the construction of $G$, implies therefore that $\cO_f \simeq G(F(\cO_f))$.
	Vice-versa, $G(\bbA^n_X, f) \simeq \cO_f$, so that the above argument yields $(\bbA^n_X,f) \simeq F(G(\bbA^n_X,f ))$.
	This completes the proof.
\end{proof}

Let $\cX$ be an $\infty$-topos and let $\CRing(\cX)\coloneqq \Sh_{\CRing}(\cX)$ denote the \infcat of sheaves of simplicial commutative rings on $\cX$.
Let $\cA \in \CRing(\cX)$ and let $\cA \Mod$ denote the $\infty$-category of left $\cA$-modules in $\Sh_{\DAb}(\cX)$.
The Dold-Kan correspondence induces a forgetful functor
\[ \CRing(\cX) \to \Sh_{\DAb^{\ge 0}}(\cX) . \]

Let $\cA\Mod(\Sh_{\DAb^{\ge 0}}(X))$ denote the $\infty$-category of left $\cA$-modules in $\Sh_{\DAb^{\ge 0}}(\cX)$.
When $\cX\simeq\cS$, we have $\cA\Mod(\Sh_{\DAb^{\ge 0}}(X))\simeq\cA\Mod^{\ge 0}$, where $\cA\Mod^{\ge 0}$ denotes the connective part of the canonical t-structure on $\cA\Mod$.
Note that the equivalence does not hold for general \inftopos $\cX$.

\begin{cor} \label{cor:modules_in_positive_characteristic}
	Let $\cX$ be an $\infty$-topos and let $\cA \in \CRing(\cX)$ be a sheaf of simplicial commutative rings on $\cX$.
	We have a canonical equivalence of $\infty$-categories
	\[ \Ab(\CRing(\cX)_{/\cA}) \simeq \cA\Mod(\Sh_{\DAb^{\ge 0}}(X)) . \]
	As a consequence, we have a canonical equivalence of stable $\infty$-categories.
	\[ \Sp(\Ab(\CRing(\cX)_{/\cA})) \simeq \cA\Mod. \]
\end{cor}

\begin{proof}
	The second statement follows from the first one.
	Indeed, it is enough to remark that
	\[ \Sp(\Sh_{\DAb^{\ge 0}}(\cX)) \simeq \Sh_{\Sp(\DAb^{\ge 0})}(\cX) \simeq \Sh_{\DAb}(\cX) . \]
	We are therefore reduced to prove the first statement.
	
	Since $\cX$ is an $\infty$-topos, we can choose a small $\infty$-category $\cC$ such that $\cX$ is a left exact and accessible localization of $\PSh(\cC)$.
	It follows that $\Sh_{\CRing}(\cX)$ and $\Sh_{\DAb^{\ge 0}}(\cX)$ are localizations of $\PSh_{\CRing}(\cC)$ and of $\PSh_{\DAb^{\ge 0}}(\cX)$, respectively.
	We can therefore replace $\cX$ by $\PSh(\cC)$.
	For every $C \in \cC$, let $\ev_C \colon \PSh(\cC) \to \cS$ be the functor given by evaluation at $C$.
	The collection of the functors $\{\ev_C\}_{C \in \cC}$ is jointly conservative.
	Furthermore, each $\ev_C$ is part of a geometric morphism of topoi.
	We are therefore reduced to prove the statement in the $\infty$-category of spaces $\cS$, and we will write $A$ instead of $\cA$.
	
	Recall from \cref{def:abelian_group_objects} the Lawvere theory of abelian groups $\cTAb$.
	Using \cref{lem:pointed_objects}, we have
	\begin{align*}
		\Ab(\CRing_{/A}) & \simeq \Ab(\CRing_{A // A}) \\
		& \simeq \Fun^\times(\cT_{\Ab}, \CRing_{A//A}) \\
		& \simeq \Fun^\times( \cT_{\Ab}, \Fun^\times(\cT_A^{\mathrm{aug}}, \cS)) \\
		& \simeq \Fun^\times( \cT_{\Ab} \times \cT_A^{\mathrm{aug}}, \cS ) .
	\end{align*}
	We can now invoke \cite[5.5.9.2]{HTT} to obtain an equivalence
	\[ \Fun^\times(\cT_{\Ab} \times \cT_A^{\mathrm{aug}}, \cS) \simeq \infty( \mathrm{Funct}^\times( \cT_{\Ab} \times \cT_A^{\mathrm{aug}}, \sSet ) ) , \]
	where $\mathrm{Funct}^\times( \cT_{\Ab} \times \cT_A^{\mathrm{aug}}, \sSet )$ is the category of strictly product preserving functors to $\sSet$ equipped with the projective model structure (whose existence is guaranteed by \cite[5.5.9.1]{HTT}), and where $\infty(-)$ denotes the underlying \infcat of a simplicial model category (cf.\ \cite[A.3.7]{HTT}).
	We now remark that
	\begin{align*}
		\mathrm{Funct}^\times( \cT_{\Ab} \times \cT_A^{\mathrm{aug}}, \sSet) & \simeq \mathrm{Funct}^\times( \cTAb, \mathrm{Funct}^\times( \cT_A^{\mathrm{aug}}, \sSet) ) \\
		& \simeq \mathrm{Funct}^{\times}( \cTAb, \sCRing_{A//A}) \\
		& \simeq \Ab( \sCRing_{A//A} ) \simeq A\sMod,
	\end{align*}
	where $\sCRing_{A//A}$ denotes the simplicial model category of simplicial commutative $A$-algebras with an augmentation to $A$.
	Moreover, under this chain of equivalences, the model structure on $\mathrm{Funct}^\times( \cT_{\Ab} \times \cT_A^{\mathrm{aug}}, \sSet)$ corresponds to the standard model structure on $A\sMod$.
	Finally, we can use the Dold-Kan equivalence in order to obtain the equivalence
	\[ \Ab(\CRing_{/A}) \simeq \infty(A\sMod) \stackrel{\mathrm{D-K}}{\simeq} A \Mod^{\ge 0} . \]
\end{proof}

\subsection{Flatness to order $n$} \label{sec:n-flatness}

We introduce in this section the notion of flatness to order $n$, which plays a key role in our proof of the representability theorem.

\begin{defin} \label{def:n_flat}
	Let $A$ be a simplicial commutative algebra and let $M \in A \Mod^{\ge 0}$ be a connective $A$-module.
	We say that $M$ is \emph{flat to order $n$} if for every discrete $A$-module $N \in A \Modh$, we have
	\[ \pi_{i}( M \otimes_A N ) = 0 \]
	for every $0 < i < n + 1$.
\end{defin}

\begin{prop} \label{prop:n_flat_sorite}
	Let $A$ be a simplicial commutative algebra and let $M \in A \Mod^{\ge 0}$ be a connective $A$-module.
	\begin{enumerate}
		\item \label{item:n_flat_decreasing} If $M$ is flat to order $n$, then it is flat to order $m$ for every $m \le n$.
		\item \label{item:n_flat_truncation} $M$ is flat to order $n$ if and only if $\tau_{\le n} M$ is flat to order $n$.
		\item \label{item:n_flat_functorial}  If $f \colon A \to B$ is a morphism of simplicial commutative algebras and $M$ is flat to order $n$, then $f^*(M) = M \otimes_A B$ is flat to order $n$.
		\item \label{item:n_flat_testing_on_truncations} Let $m, n \ge 0$ be integers. Then $M$ is flat to order $n$ if and only if $M \otimes_A \tau_{\le m} A$ is flat to order $n$.
	\end{enumerate}
\end{prop}

\begin{proof}
	Statement (\ref{item:n_flat_decreasing}) follows directly from the definitions.
	We prove (\ref{item:n_flat_truncation}).
	Consider the fiber sequence
	\[ \tau_{\ge n+1} M \to M \to \tau_{\le n} M. \]
	Let $N \in A\Modh$ and consider the induced fiber sequence
	\[ (\tau_{\ge n+1} M) \otimes_A N \to M \otimes_A N \to (\tau_{\le n} M) \otimes_A N. \]
	Since $\tau_{\ge n+1}M \otimes_A N \in A \Mod^{\ge n+1}$, the conclusion follows from the long exact sequence of cohomology groups.

	We now prove (\ref{item:n_flat_functorial}).
	Let $N \in B \Modh$.
	Recall that the functor $f_* \colon B \Mod \to A \Mod$ is $t$-exact and conservative.
	In particular, it is enough to prove that $\pi_{i}( f_*( f^*(M) \otimes_B N ) ) = 0$ for $0 < i < n+1$.
	We have
	\[ f_*(f^*(M) \otimes_B N) \simeq M \otimes_A f_*(N) . \]
	The conclusion now follows from the fact that $M$ is flat to order $n$.
	
	Finally, we prove (\ref{item:n_flat_testing_on_truncations}).
	Since $\pi_0(A) \simeq \pi_0(\tau_{\le m}(A))$, it is enough to deal with the case $m = 0$.
	As the ``only if'' follows from point (\ref{item:n_flat_truncation}), we are left to prove the ``if'' direction.
	Suppose therefore that $M \otimes_A \pi_0(A)$ is flat to order $n$.
	Let $N \in A \Modh$.
	Since $A \Modh \simeq \pi_0(A) \Modh$, we see that $N$ is naturally a $\pi_0(A)$-module.
	Therefore, we can write
	\[ M \otimes_A N \simeq (M \otimes_A \pi_0(A)) \otimes_{\pi_0(A)} N . \]
	Since $M \otimes_A \pi_0(A)$ is flat to order $n$, it follows that
	\[ \pi_{i}( M \otimes_A N ) = 0 \]
	for every $0 < i < n+1$. 
	In other words, $M$ is flat to order $n$.
\end{proof}

\begin{prop} \label{prop:n_flat_implies_flat_truncation}
	Let $A\in\CRing$ and let $M \in A \Mod^{\ge 0}$.
	Assume that $M$ is flat to order $n$, and that $A$ is $m$-truncated with $m \le n$.
	Then $\tau_{\le n} M$ is flat as $A$-module.
\end{prop}

\begin{proof}
	It follows from the same proof of \cite[7.2.2.15, (3) $\Rightarrow$ (1)]{Lurie_Higher_algebra} that
	\[ \pi_{i}(M) \simeq \pi_{i}(A) \otimes_{\pi_0(A)} \pi_0(M) \]
	for $0 \le i \le n$.
		Moreover, since $A$ is $m$-truncated and $m \le n$, we see that
	\[ \pi_{i}(\tau_{\le n} M) \simeq 0 \simeq \pi_{i}(A) \otimes_{\pi_0(A)} \pi_0(M) \]
	for $i > n$.
	Therefore, $\tau_{\le n} M$ is flat.
\end{proof}

\begin{cor} \label{cor:n_flat_increasing}
	Let $A \in\CRing$ and let $M \in A \Mod^{\ge 0}$.
	Assume that $M$ is flat to order $n$, and that $A$ is $m$-truncated with $m \le n$.
	Then $M$ is flat to order $n+1$ if and only if $\pi_{n+1}(M) = 0$.
\end{cor}

\begin{proof}
	Using \cref{prop:n_flat_sorite}(\ref{item:n_flat_truncation}), we deduce that $\tau_{\le n}$ is flat to order $n$.
	For any $N \in A \Modh$, consider the fiber sequence
	\[ \tau_{\ge n+1} M \otimes_A N \to M \otimes_A N \to \tau_{\le n} M \otimes_A N . \]
	Since $A$ is $m$-truncated, \cref{prop:n_flat_implies_flat_truncation} implies that $\tau_{\le n} M$ is flat.
	In particular, $\tau_{\le n} M \otimes_A N$ is discrete.
	Therefore, passing to the long exact sequence of cohomology groups, we obtain
	\[ 0 \to \pi_{n+1}(\tau_{\ge n+1} M \otimes_A N) \to \pi_{n+1}(M \otimes_A N) \to 0 . \]
	It follows from \cite[7.2.1.23]{Lurie_Higher_algebra} that
	\[ \pi_{n+1}(\tau_{\ge n+1} M \otimes_A N) \simeq \pi_{n+1}(M) \otimes_{\pi_0(A)} N . \]
	Therefore, if $\pi_{n+1}(M) = 0$, then $M$ is flat to order $n+1$.
	Vice-versa, if $M$ is flat to order $n+1$, then choosing $N = \pi_0(A)$, we conclude that $\pi_{n+1}(M) = 0$.
	The proof is thus complete.
\end{proof}

\bibliographystyle{plain}
\bibliography{dahemaold}

\end{document}